\documentclass[11pt]{article}
\usepackage[sectionbib]{natbib}
\usepackage{array,epsfig,fancyheadings,rotating}
\usepackage[colorlinks,linkcolor=blue,anchorcolor=blue,citecolor=blue,CJKbookmarks=True]{hyperref}
\usepackage{amsthm,amsmath,amsfonts,amssymb}
\usepackage[page,toc,titletoc,title]{appendix}
\usepackage{sectsty, secdot}
\usepackage{amsbsy}
\usepackage{amsmath}
\usepackage{amssymb}
\usepackage{epsfig}
\usepackage{multicol}
\usepackage{multirow}
\usepackage{color}
\usepackage{bm}
\usepackage{threeparttable}
\usepackage{float}
\usepackage{array}
\usepackage{amsmath}
\usepackage{amssymb}
\usepackage{amsfonts}
\usepackage{multirow}
\usepackage{amsthm}
\usepackage{graphicx}
\usepackage{mathrsfs}
\usepackage{bm}
\usepackage{enumerate}
\usepackage{subfigure}
\usepackage{booktabs}
\usepackage{pdflscape}
\usepackage{url}
\usepackage{xcolor,graphicx,float}
\usepackage{rotating}
\usepackage{times}
\usepackage{indentfirst}
\usepackage{blindtext}
\usepackage{algorithm,algorithmicx,float}
\usepackage{lipsum}
\usepackage[noend]{algpseudocode}
\usepackage{setspace}
\usepackage{booktabs}

\usepackage{array,epsfig,fancyheadings,rotating}
\usepackage{sectsty, secdot}
\usepackage{authblk}
\usepackage{amsmath}
\usepackage{amssymb}
\usepackage{amsfonts}
\usepackage{multirow}
\usepackage{amsthm}
\usepackage{graphicx}
\usepackage{bm}
\usepackage{enumerate}
\usepackage{subfigure}
\usepackage{booktabs}
\usepackage{pdflscape}
\usepackage{url}
\usepackage{xcolor,graphicx,float}
\usepackage{rotating}
\usepackage{times}
\usepackage{indentfirst}
\usepackage{blindtext}
\usepackage{algorithm,algorithmicx,float}
\usepackage{lipsum}
\usepackage[noend]{algpseudocode}
\usepackage{setspace}
\usepackage{booktabs}


\makeatletter
\renewcommand{\@thesubfigure}{\hskip\subfiglabelskip}
\makeatother

\renewcommand{\baselinestretch} {1.3}
\makeatletter 
\def\singlespace{\def\baselinestretch{1}\@normalsize}

\newlength\savewidth

\@addtoreset{equation}{section}
\newtheorem{theorem}{Theorem}
\newtheorem{assumption}{Assumption}
\newtheorem{lemma}{Lemma}
\newtheorem{remark}{Remark}
\newtheorem{corollary}{Corollary}
\newtheorem{example}{Example}
\newtheorem{condition}{Condition}
\newtheorem{proposition}{Proposition}
\newtheorem{definition}{Definition}

\newcommand{\ba}{\mathbf{a}}
\newcommand{\bA}{\mathbf{A}}
\newcommand{\bP}{\mathbf{P}}
\newcommand{\bD}{\mathbf{D}}
\newcommand{\bI}{\mathbf{I}}
\newcommand{\bv}{\mathbf{v}}
\newcommand{\bw}{\mathbf{w}}
\newcommand{\bX}{\mathbf{X}}
\newcommand{\bx}{\mathbf{x}}
\newcommand{\bs}{\mathbf{s}}
\newcommand{\be}{\mathbf{e}}
\newcommand{\by}{\mathbf{y}}

\newcommand{\bk}{\mathbf{k}}
\newcommand{\bbf}{\mathbf{f}}

\newcommand{\btheta}{\boldsymbol{\theta}}
\newcommand{\bbeta}{\boldsymbol{\beta}}
\newcommand{\bomega}{\boldsymbol{\omega}}
\newcommand{\bphi}{\boldsymbol{\phi}}
\newcommand{\beps}{\boldsymbol{\epsilon}}

\newcommand{\bgamma}{\boldsymbol{\gamma}}

\renewcommand{\hat}{\widehat}
\def\tilde{\widetilde}
\def\tr{\operatorname{tr}}
\DeclareMathOperator*{\argmax}{argmax}

\makeatletter
\newcommand{\opnorm}{\@ifstar\@opnorms\@opnorm}
\newcommand{\@opnorms}[1]{%
  \left|\mkern-1.5mu\left|\mkern-1.5mu\left|
   #1
  \right|\mkern-1.5mu\right|\mkern-1.5mu\right|
}
\newcommand{\@opnorm}[2][]{%
  \mathopen{#1|\mkern-1.5mu#1|\mkern-1.5mu#1|}
  #2
  \mathclose{#1|\mkern-1.5mu#1|\mkern-1.5mu#1|}
}
\makeatother

\newcommand{\bg}{\begin{eqnarray}}
\newcommand{\ed}{\end{eqnarray}}
\newcommand{\bgn}{\begin{eqnarray*}}
\newcommand{\edn}{\end{eqnarray*}}

\makeatletter
\renewcommand{\@thesubfigure}{\hskip\subfiglabelskip}
\makeatother

\makeatletter
\def\singlespace{\def\baselinestretch{1}\@normalsize}

\title{On optimality of Mallows model averaging}

\author[a]{Jingfu Peng}
\author[a]{Yang Li}
\author[b]{Yuhong Yang}

\affil[a]{School of Statistics, Renmin University of China}
\affil[b]{School of Statistics, University of Minnesota}

\date{}

\begin{document}
\begin{sloppypar}
\maketitle

\begin{abstract}

In the past decades, model averaging (MA) has attracted much attention as it has emerged as an alternative tool to the model selection (MS) statistical approach. Hansen [\emph{Econometrica} \textbf{75} (2007) 1175--1189] introduced a Mallows model averaging (MMA) method with model weights selected by minimizing a Mallows' $C_p$ criterion. The main theoretical justification for MMA is an asymptotic optimality (AOP), which states that the risk/loss of the resulting MA estimator is asymptotically equivalent to that of the best but infeasible averaged model. MMA's AOP is proved in the literature by either constraining weights in a special discrete weight set or limiting the number of candidate models. In this work, it is first shown that under these restrictions, however, the optimal risk of MA becomes an unreachable target, and MMA may converge more slowly than MS. In this background, a foundational issue that has not been addressed is: When a suitably large set of candidate models is considered, and the model weights are not harmfully constrained, can the MMA estimator perform asymptotically as well as the optimal convex combination of the candidate models? We answer this question in both nested and non-nested settings. In the nested setting, we provide finite sample inequalities for the risk of MMA and show that without unnatural restrictions on the candidate models, MMA's AOP holds in a general continuous weight set under certain mild conditions. In the non-nested setting, a sufficient condition and a negative result are established for the achievability of the optimal MA risk. Implications on minimax adaptivity are given as well. The results from simulations back up our theoretical findings. 

\end{abstract}

\bigskip

\tableofcontents

\addtocontents{toc}{\setcounter{tocdepth}{2}}
\section{Introduction}\label{sec:introduction}

In statistical modeling, multiple candidate models are usually considered to explore the data. Model selection (MS) guides us in search for the best model among candidates based on a traditional selection criterion, such as AIC \citep{Akaike1973}, $C_p$ \citep{Mallows1973}, and BIC \citep{Schwarz1978}, the use of cross-validation \citep{Allen1974, Stone1974}, and solving a penalized regression problem, such as Lasso \citep{Tibshirani1996}, adaptive Lasso \citep{zou2006adaptive}, SCAD \citep{Fan2001}, and MCP \citep{Zhang2010Nearly}. The key theoretical properties of these methods, namely consistency in selection, asymptotic efficiency, and minimax-rate optimality, have been well established in the literature (see \cite{Claeskens2008} and \cite{Ding2018} for general reviews). These MS methodologies have found successful applications in density estimation and classical regression settings, and have also been extended to more complex statistical models (see, e.g., \cite{Muller2013} for a review on mixed effects models). Once a final model is selected, all subsequent estimation, prediction, and inference are typically based on the selected model as if it were given in advance.

However, it has been increasingly recognized that choosing just one model inherently ignores possibly high uncertainty in the selection process \citep{Chatfield1995model, Draper1995, Yuan2005}. Model averaging (MA), on the other hand, provides an alternative to reduce the variability in MS while offering a possibility of reducing modeling bias by averaging over the candidate models properly.

MA has a rich heritage in Bayesian statistics, see, e.g., \cite{Draper1995}, \cite{Clyde1996}, \cite{George1997}, and \cite{Hoeting1999} for more details and references therein. From a frequentist perspective, several attractive strategies have been proposed to combine models, including boosting \citep{FREUND1995256}, bagging \citep{Breiman1996b}, random forest \citep{Amit1997}, information criterion weighting \citep{Buckland1997, Hjort2003}, adaptive regression by mixing \citep{Yang2001, Yuan2005}, exponentially weighted aggregation \citep{Leung2006}, to name a few. In particular, by minimizing some specific performance measures, a growing MA literature develops methods to pursue the optimal convex combination of the candidate models based on the same data. To the best of our knowledge, this problem was first considered by \cite{Blaker1999adaptive} in a two-model case, and studied in a multiple-model setting by \citet{Hansen2007}, who proposed a Mallows model averaging (MMA) method to select weights for averaging across nested linear models by minimizing the Mallows' $C_p$ criterion \citep{Mallows1973}. Adopting other performance measures like cross-validation error and Kullback-Leibler divergence, the MMA-type strategies have been developed explicitly for other or more general frameworks, such as heteroskedastic error regression model \citep{Hansen2012, Liu2013}, time-series error model \citep{HANSEN2008342, ZHANG201382, Cheng2015}, high-dimensional regression model \citep{Ando2014, Zhang2020}, generalized linear model \citep{Ando2017, Zhang2016jlm}, quantile regression model \citep{Lu2015}, expectile regression model \citep{TU2020109607}, varying-coefficient model \citep{zhu2019mallows}, semiparametric model \citep{FANG2022219}, among many useful others.

Given the increasing and potential wide applications of the MMA-type methods, an essential question is how good this popular class of methods for constructing an MA estimator is. This paper focuses on MMA introduced by \cite{Hansen2007} and revisits its optimality. Note that the MMA criterion is an unbiased estimate of the squared risk of the MA estimator plus a constant, and the resulting MMA estimator targets the minimization of the squared risk/loss of MA.

The optimality of MMA has certainly been studied from an asymptotic viewpoint in the MA literature. An asymptotic optimality (AOP) theory states that a good MA estimator can be asymptotically equivalent to the optimal convex combination of the given candidates in terms of the statistical risk/loss. There are two major scenarios where the MMA's AOP is proved. In a nested setup where the candidates are strictly nested, \citet{Hansen2007} proved it when the weight vectors are contained in a special discrete set. His results do not impose any additional assumption on the number of nested models. In a non-nested setup, \citet{Wan2010} considered the minimization of the MMA criterion over a continuous weight set. Their paper justifies the MMA's AOP but requires a restriction on the candidate model set. In summarizing the literature in relation to the real goal of AOP, while the aforementioned theoretical advancements are novel and valuable, the consequences of the restrictions imposed on weight/candidate models are still unclear.

Consider a typical nested framework with the $m$-th candidate containing the first $m$ regressors. For \citet{Hansen2007}'s theory, a sensible choice for the candidate model set is to include $M_n \geq m_n^*$ nested models, where $m_n^*$ is the size of the optimal nested model. We show in Section~\ref{subsec:review_hansen} that when $m_n^*$ is not too small relative to the sample size $n$ (e.g., $m_n^*$ grows at order $n^{\alpha}$ for some $0<\alpha<1$), the best possible MA risk in the discrete weight set is suboptimal. For the theory in \cite{Wan2010}, although it allows for the combination of possibly non-nested models, it still relies on a \emph{reorder-then-combine} strategy to implement MMA, which first reorders the regressors in terms of  importance and then constructs a typically nested MA estimators \citep[see, e.g.,][]{Zhang2020}. As shown in Section~\ref{subsec:review_wan}, the required restriction in \cite{Wan2010} is so strong that $m_n^*$ is excluded, and the reorder-then-combine strategy can only combine a set of underperforming models. Note that the MMA-type literature often motivates their approaches to overcome the drawbacks of MS and hence perform better. However, the MA estimator based on such nested model sets actually converges more slowly than MS.

In this background, a critical issue that has not been addressed in the existing literature is: When the weight vector is allowed for the full potential of MA, and the number of candidate models is not harmfully constrained, can the MMA estimator perform asymptotically as well as the infeasible optimal averaged model?

This paper addresses the aforementioned foundational question in both the nested and non-nested setups, making three main contributions. (1) We establish a non-asymptotic risk bound for MMA when the candidate models are nested. When the coefficients do not decay too fast (e.g., in a polynomial rate), this risk bound demonstrates that the MMA estimator based on all nested models asymptotically attains the optimal MA risk without discretizing the weight set. Several specific candidate model sets are also proposed to fully exploit the advantages of MMA in the nested setup. (2) We show that if the coefficients are partially ordered and do not decay too fast, the optimal MA risk over all non-nested models can still be realized by some specific nested MMA strategies. In contrast, if there is no prior knowledge of the coefficients' order at all, the optimal non-nested MA risk cannot be reached by any estimation procedures, and the asymptotic advantage of MA over MS cannot be materialized either. (3) We prove that the MMA estimator exhibits optimal minimax adaptivity over some general coefficient classes, such as ellipsoids and hyperrectangles. The results from our finite sample simulations support these findings.

The rest of the paper is organized as follows. In Section~\ref{sec:setup}, we set up the regression framework and give the MMA estimators. In Section~\ref{sec:review}, we theoretically investigate the consequences of using the discrete weight set or restricting the candidate model set. We then in Section~\ref{sec:main_results} develop the MMA's AOP theory in the nested setup. The non-nested setup is addressed in Section~\ref{sec:non-nested}. Section~\ref{sec:minimax} shows the minimax adaptivity of MMA. Section \ref{sec:simulation} presents the results of simulation experiments. Concluding remarks are given in Section~\ref{sec:conclusion}. The discussions on the other related works, proofs, and additional simulation and theoretical results can be found in the Appendix. 

\section{Problem setup}\label{sec:setup}

\subsection{Setup and notation}\label{sec:setup:1}

Consider the linear regression model

\begin{equation}\label{eq:model}
  y_i=f_i+\epsilon_i=\sum_{j=1}^{p_n}\beta_jx_{ij}+\epsilon_i,\quad i=1,\ldots,n,
\end{equation}
where $\epsilon_1, \ldots, \epsilon_n$ are i.i.d. sub-Gaussian random variables with $\mathbb{E}{\epsilon_i}=0$ and $\mathbb{E}{\epsilon_i^2}=\sigma^2$, and $\bx_j = (x_{1j}, \ldots , x_{nj})^{\top}$, $j=1,\ldots,p_n$, are nonstochastic regressor vectors. Define the response vector $\by=(y_1,\ldots,y_n)^{\top}$, the regression mean vector $\bbf=(f_1,\ldots,f_n)^{\top}$, the regression coefficient vector $\bbeta=\left(\beta_1,\ldots,\beta_{p_n}\right)^{\top}$, the regressor matrix $\bX=\left[\bx_1,\ldots, \bx_{p_n}\right]\in \mathbb{R}^{n \times p_n}$, and the noise vector $\beps=(\epsilon_1,\ldots,\epsilon_n)^{\top}$. We can write (\ref{eq:model}) in a matrix form
\begin{equation}\label{eq:model_matrix}
  \by = \bbf+ \beps = \bX\bbeta + \beps.
\end{equation}
For the sake of simplicity, we assume $p_n \leq n$, and $\bX$ has full column rank.

To estimate the true regression mean vector $\bbf$, $M_n$ linear regression models are considered as candidates. The $m$-th candidate model includes the regressors in the index set $\mathcal{I}_m \subseteq \{1,\ldots,p_n\}$, and the set of the candidate models is $\mathcal{M}=\left\{\mathcal{I}_1,\ldots,\mathcal{I}_{M_n} \right\}$ with $M_n=|\mathcal{M}|$, where $|\mathcal{S}|$ denotes the cardinality of a set $\mathcal{S}$ throughout this paper. Let $\bX_{\mathcal{I}_m}$ be the design matrix of the $m$-th candidate model, which estimates $\bbf$ by the least squares method $\hat{\bbf}_{m|\mathcal{M}}=\bX_{\mathcal{I}_m}(\bX_{\mathcal{I}_m}^{\top}\bX_{\mathcal{I}_m})^{-1}\bX_{\mathcal{I}_m}^{\top}\by= \bP_{m|\mathcal{M}}\by$, where $\bP_{m|\mathcal{M}}\triangleq\bX_{\mathcal{I}_m}(\bX_{\mathcal{I}_m}^{\top}\bX_{\mathcal{I}_m})^{-1}\bX_{\mathcal{I}_m}^{\top}$.

Let $\bw=(w_1,\ldots,w_{M_n})^{\top}$ denote a weight vector in the unit simplex of $\mathbb{R}^{M_n}$:
\begin{equation}\label{eq:weight_general}
  \mathcal{W}_{M_n}=\left\{\bw\in[0,1]^{M_n}:\sum_{m=1}^{M_n}w_m=1\right\}.
\end{equation}
Given the candidate model set $\mathcal{M}$, the MA estimator of $\bbf$ is $\hat{\bbf}_{\bw | \mathcal{M}}=\sum_{m=1}^{M_n}w_m\hat{\bbf}_{m|\mathcal{M}} =\bP_{\bw|\mathcal{M}}\by$, where $\bP_{\bw|\mathcal{M}}\triangleq\sum_{m=1}^{M_n}w_m\bX_{\mathcal{I}_m}(\bX_{\mathcal{I}_m}^{\top}\bX_{\mathcal{I}_m})^{-1}\bX_{\mathcal{I}_m}^{\top}$, and the subscripts $m|\mathcal{M}$ and $\bw | \mathcal{M}$ are to emphasize the dependences of the MS and MA estimators on $\mathcal{M}$, respectively.

We consider the normalized squared $\ell_2$ loss $L_n(\hat{\bbf},\bbf)= n^{-1}\|\hat{\bbf}-\bbf \|^2$ and its corresponding risk $R_n(\hat{\bbf},\bbf)=\mathbb{E}L_n(\hat{\bbf},\bbf)$ as two measures of the performance of an estimator $\hat{\bbf}$, where $\|\cdot\|$ refers to the Euclidean norm. For abbreviation, let $L_n(m|\mathcal{M},\bbf)$, $R_n(m|\mathcal{M},\bbf)$, $L_n(\bw | \mathcal{M},\bbf)$ and $R_n(\bw | \mathcal{M},\bbf)$ stand for $L_n(\hat{\bbf}_{m|\mathcal{M}},\bbf)$, $R_n(\hat{\bbf}_{m|\mathcal{M}},\bbf)$, $L_n(\hat{\bbf}_{\bw | \mathcal{M}},\bbf)$ and $R_n(\hat{\bbf}_{\bw | \mathcal{M}},\bbf)$, respectively. We denote $m^*|\mathcal{M}=\arg\min_{m \in \{1,\ldots,M_n \}}R_n(m|\mathcal{M},\bbf)$ the optimal single model in $\mathcal{M}$, and $\bw^*|\mathcal{M}=\arg\min_{\mathbf{w} \in \mathcal{W}_{M_n}}R_n(\bw | \mathcal{M},\bbf)$ the optimal weight vector based on the candidate model set $\mathcal{M}$ and the general continuous weight set $\mathcal{W}_{M_n}$. The quantities $m^*|\mathcal{M}$ and $\bw^*|\mathcal{M}$ are unknown in practice since they depend on the unknown parameters $\bbf$ and $\sigma^2$.

In this paper, we choose the weights by minimizing the MMA criterion proposed by \cite{Hansen2007}
\begin{equation}\label{eq:criterion}
  C_n(\bw|\mathcal{M},\by)=\frac{1}{n}\left\|\by-\hat{\bbf}_{\bw | \mathcal{M}} \right\|^2+\frac{2\hat{\sigma}^2}{n}\bk^{\top}\bw,
\end{equation}
that is, $\hat{\bw}|\mathcal{M}=\arg\min_{\mathbf{w} \in \mathcal{W}_{M_n}}C_n(\bw|\mathcal{M},\by)$, where $\hat{\sigma}^2$ is an estimator of $\sigma^2$, $\bk=(k_1,\ldots,k_{M_n})^{\top}$, and $k_m = |\mathcal{I}_m|$. Note that when $\mathbb{E}\hat{\sigma}^2=\sigma^2$, $\hat{\bw}|\mathcal{M}$ can be seen as a minimizer of an unbiased estimate for the MA risk
\begin{equation}\label{eq:MA_risk}
  R_n(\bw|\mathcal{M},\bbf)=\frac{1}{n}\left\| \bbf - \bP_{\bw|\mathcal{M}}\bbf \right\|^2 + \frac{\sigma^2}{n}\tr\left(\bP_{\bw|\mathcal{M}}^{\top}\bP_{\bw|\mathcal{M}}\right),
\end{equation}
since $\mathbb{E}C_n(\bw|\mathcal{M},\by)= R_n(\bw|\mathcal{M},\bbf) + \sigma^2$. The resulting MMA estimator of $\bbf$ is
\begin{equation}\label{eq:MMA_estimator}
  \hat{\bbf}_{\hat{\bw} | \mathcal{M}} =\sum_{m=1}^{M_n}\hat{w}_m\hat{\bbf}_{m|\mathcal{M}}.
\end{equation}

Let $\mathbb{E}L_n(\hat{\bw}|\mathcal{M},\bbf)\triangleq n^{-1}\mathbb{E}\|\hat{\bbf}_{\hat{\bw} | \mathcal{M}}-\bbf \|^2$ and $\mathbb{E}R_n(\hat{\bw}|\mathcal{M},\bbf)\triangleq n^{-1}\mathbb{E}\| \bbf - \bP_{\hat{\bw}|\mathcal{M}}\bbf \|^2+n^{-1}\sigma^2\mathbb{E}\tr(\bP_{\hat{\bw}|\mathcal{M}}^{\top}\bP_{\hat{\bw}|\mathcal{M}})$ denote the risk functions of (\ref{eq:MMA_estimator}). The former is a little different from the latter since in the latter function, $\hat{\bw}$ is directly plugged into the right hand side of (\ref{eq:MA_risk}). Let $Q_n(\hat{\bw}|\mathcal{M},\bbf)$ denote any one of the two quantities: $\mathbb{E}L_n(\hat{\bw}|\mathcal{M},\bbf)$ and $\mathbb{E}R_n(\hat{\bw}|\mathcal{M},\bbf)$.

From now on, we will use the symbols $\lesssim$, $\gtrsim$, and $\asymp$ for comparison of positive sequences, where $a_n\lesssim b_n$ means $a_n=O(b_n)$, $a_n\gtrsim b_n$ means $b_n=O(a_n)$, and $a_n\asymp b_n$ means both $a_n\lesssim b_n$ and $a_n\gtrsim b_n$. Also, $a_n\sim b_n $ means that $a_n/b_n\rightarrow 1$ as $n\rightarrow \infty$. Let $\lfloor a \rfloor$ and $\lceil a \rceil$ return the floor and the ceiling of $a$, respectively. For any two real numbers $a$ and $b$, we use notation $a \wedge b = \min(a,b)$ and $a \vee b = \max(a,b)$.

\subsection{Definitions of optimality}\label{sec:setup2}

In the literature, a common approach to defining the optimality of MMA is AOP.

\begin{definition}
   Given a candidate model set $\mathcal{M}$ and a weight set $\mathcal{W}$, an MA estimator $\hat{\bbf}_{\tilde{\bw} | \mathcal{M}}$ with $\tilde{\bw}$ trained on the data is said to be asymptotically optimal (AOP) at $\bbf$ if it satisfies
  \begin{equation}\label{eq:opt1}
    Q_n\left(\tilde{\bw}|\mathcal{M},\bbf\right)\leq [1+o(1)]\min_{\mathbf{w} \in \mathcal{W}}R_n\left(\bw|\mathcal{M},\bbf\right)
  \end{equation}
  as $n\to \infty$.
\end{definition}

The existing literature showed that the AOP property of MMA can be obtained with certain restrictions on the weight set $\mathcal{W}$ or the candidate model set $\mathcal{M}$. Specifically, in a nested model setting with $\mathcal{M}=\{ \{1,2,\ldots,k_m \}: 1\leq k_1<k_2<\cdots<k_{M_n}\leq p_n\}$ and
\begin{equation}\label{eq:condition_discrete}
  \mathcal{W}=\mathcal{W}_{|\mathcal{M}|}(N)\triangleq\left\{\sum_{m=1}^{|\mathcal{M}|}w_m=1, w_m\in \left\{0,\frac{1}{N},\frac{2}{N},\ldots,1\right\}\right\},
\end{equation}
in which $N$ is a fixed positive integer, \citet{Hansen2007} proved the MMA's AOP when $nR_n(\bw^* | \mathcal{M},\bbf) \to \infty$. Under a non-nested setup, \citet{Wan2010} established the MMA's AOP for a general non-nested $\mathcal{M}$ and the continuous weight set $\mathcal{W}_{|\mathcal{M}|}$ but with an additional condition on $\mathcal{M}$, that is
\begin{equation}\label{eq:condition_wan}
  \frac{|\mathcal{M}|\sum_{m=1}^{|\mathcal{M}|}R_n\left(m|\mathcal{M},\bbf\right)}{nR_n^2\left(\bw^*|\mathcal{M},\bbf  \right)}\to 0.
\end{equation}
In this paper, we refer the AOP theories in \cite{Hansen2007} and \cite{Wan2010} as the \emph{restricted AOP}, because the former does not allow all possible convex combinations of the nested models, and the latter excludes useful models in the candidate set, hence may lead to a suboptimal convergence rate (see Section~\ref{sec:review} for the detailed discussion).

Define $\mathcal{M}_A=\{\{1,\ldots,j \}: 1\leq j \leq p_n  \}$ the candidate model set with all the nested models. The optimal MA risk $R_n\left(\bw^*|\mathcal{M}_A,\bbf\right)$ can be seen as the full potential of MA under the nested model setting. Therefore, in the context of nested candidate models, as opposed to the restricted AOP, a more natural definition of the optimality is the \emph{full AOP}.
\begin{definition}\label{def:full_aop}
   An MA estimator $\hat{\bbf}_{\tilde{\bw} | \mathcal{M}}$ with $\tilde{\mathbf{w}}$ trained on the data is said to achieve the full AOP at $\bbf$ if it satisfies
  \begin{equation}\label{eq:fullopt}
    Q_n\left(\tilde{\bw}|\mathcal{M},\bbf\right)\leq [1+o(1)]R_n\left(\bw^*|\mathcal{M}_A,\bbf\right)
  \end{equation}
  as $n\to \infty$.
\end{definition}

Then two important questions arise:
\begin{description}

\item[Q1.] Does the MMA estimator (\ref{eq:MMA_estimator}) achieve the full AOP by combining candidates in $\mathcal{M}_A$ and minimizing the criterion (4) over $\mathcal{W}_{|\mathcal{M}_A|}$ directly?

\item[Q2.] Can we use a reduced set of nested models $\mathcal{M} \subset \mathcal{M}_A$ yet the MMA estimator based on $\mathcal{M}$ still satisfies the full AOP property?

\end{description}

Note the full AOP aims at achieving the optimal MA risk in the nested setting. To further clarify the potential of AOP, let $\mathcal{M}_{AS}=\{ \mathcal{I}: \mathcal{I} \subseteq \{1,\ldots,p_n \} \}$ be the candidate model set comprising all subsets of $\{1,\ldots,p_n \}$, and let $R_n\left(\bw^*|\mathcal{M}_{AS},\bbf\right)$ be the corresponding optimal MA risk, which represents the ideal performance in the non-nested setting.
\begin{definition}\label{def:ideal_aop}
   An MA estimator $\hat{\bbf}_{\tilde{\bw} | \mathcal{M}}$ with $\tilde{\mathbf{w}}$ estimated on the data is said to achieve the ideal AOP at $\bbf$ if it satisfies
  \begin{equation}\label{eq:idealopt}
    Q_n\left(\tilde{\bw}|\mathcal{M},\bbf\right)\leq [1+o(1)]R_n\left(\bw^*|\mathcal{M}_{AS},\bbf\right)
  \end{equation}
  as $n\to \infty$.
\end{definition}
\begin{description}

\item[Q3.] Is there an estimator that can achieve the ideal AOP generally?

\end{description}

The answers to Q1--Q3 may provide a previously unavailable insight on the AOP theory of MMA. Furthermore, this paper also introduces another approach to evaluating the optimality of MMA: minimax adaptivity across a list of classes of $\bbf$. Detailed formulation and discussion on the minimax theory of MMA are presented in Section~\ref{sec:minimax}.

\subsection{Reparameterization of the regression model}\label{sec:reparameterization}

Before comprehensively addressing the aforementioned questions, we first introduce some notations helpful in our theoretical analysis. As pointed out in \cite{Xu2022From}, $\bP_{j|\mathcal{M}_A}-\bP_{(j-1)|\mathcal{M}_A}=\bphi_j\bphi_j^{\top}$, $j=1,\ldots,p_n$, are mutually orthogonal, where $\bP_{0|\mathcal{M}_A}=\boldsymbol{0}_{n \times n}$, and $\bphi_j\in \mathbb{R}^{n}$ is an eigenvector satisfying $\|\bphi_j\|=1$. Obviously, $\{\bphi_1,\ldots,\bphi_{p_n} \}$ forms an orthonormal basis for the column space of $\bX$. Let us denote the \emph{transformed coefficients} $\btheta=(\theta_1,\ldots,\theta_{p_n})^{\top}$ of $\bbf$ by
\begin{equation}\label{eq:trans_para}
  \theta_j=\theta_j(\bbf)=\bphi_j^{\top}\bbf/\sqrt{n},\quad j=1,\ldots,p_n.
\end{equation}

Based on (\ref{eq:trans_para}), the regression model is equivalent to $\by = \bbf+ \beps = \sqrt{n}\sum_{j=1}^{p_n}\theta_j\bphi_j + \beps$, and $\hat{\bbf}_{\bw|\mathcal{M}_A}$ can be expressed as the MA estimator utilizing all the nested models along $\{\bphi_1,\ldots,\bphi_{p_n}\}$. In addition, when the columns of $\bX$ are mutually orthogonal with $\ell_2$ norm $\sqrt{n}$, the transformed coefficient $\theta_j$ coincides with the regression coefficient $\beta_j$, and $\hat{\bbf}_{\bw|\mathcal{M}_{AS}}$ aligns with the non-nested MA estimator based on all subsets of $\{\bphi_1,\ldots,\bphi_{p_n}\}$.

\section{Revisiting the existing AOP theories on MMA}\label{sec:review}

\subsection{Discrete weight set}\label{subsec:review_hansen}

This subsection investigates the consequences of using the discrete weight set in the nested setting \citep{Hansen2007}. For simplicity, we consider a set of successive candidate models $\mathcal{M}_{S}=\{\{1,\ldots,m \}:1\leq m \leq M_n\}$. Let $m_n^*=m^*|\mathcal{M}_A$ be the optimal nested model and $\bw_N^*|\mathcal{M}_S=\arg\min_{\bw \in \mathcal{W}_{|\mathcal{M}_S|}(N)}R_n(\bw | \mathcal{M}_S,\bbf)$ denote the optimal discrete weight vector in $\mathcal{W}_{|\mathcal{M}_S|}(N)$. We focus on the magnitude of the risk increment $R_n\left(\bw_N^*|\mathcal{M}_S,\bbf\right)-R_n\left(\bw^*|\mathcal{M}_S,\bbf\right)$ under certain assumptions on $\bbf$.

\begin{assumption}\label{asmp:square_summable}
   The regression mean vector $\bbf$ satisfies $\lim\sup_n n^{-1}\|\bbf \|^2 < \infty$.
\end{assumption}

\begin{assumption}\label{asmp:regressor_order}
  (Completely ordered coefficients) The transformed coefficients (\ref{eq:trans_para}) are completely ordered, which means $\{|\theta_j|,j\geq 1\}$ is a non-increasing positive sequence.
\end{assumption}

Assumption~\ref{asmp:square_summable} is a standard assumption for regression estimation problems. When the columns of $\bX$ are mutually orthogonal, we see that $\theta_j$ is proportional to $\beta_j$. In this case, Assumption~\ref{asmp:regressor_order} ensures that the regressors are ordered from most important to least important, which is reasonable for the nested setup. The implication of this assumption in the non-nested setup is discussed in Section~\ref{sec:non-nested}.

We further provide two different conditions on the transformed coefficients (\ref{eq:trans_para}).
\begin{condition}\label{condition1}
(Slowly decaying coefficients) There exist constants $\kappa>1$ and $0<\delta\leq\nu<1$ with $\kappa\nu^2<1$ such that $\delta\leq|\theta_{\lfloor \kappa l \rfloor}/\theta_l|\leq\nu$ when $l$ is large enough.
\end{condition}
\begin{condition}\label{condition2}
(Fast decaying coefficients) For every constant $\kappa>1$, $\lim\limits_{l \to \infty}|\theta_{\lfloor \kappa l \rfloor}/\theta_l|= 0$.
\end{condition}

Condition~\ref{condition1} includes the case $\theta_j=j^{-\alpha_1}$ for $\alpha_1 > 1/2$, which serves as the principal example in the MA literature since in this case, the optimal MA risk can significantly reduce the optimal MS risk \citep{Peng2021}. In contrast, the coefficients satisfying Condition~\ref{condition2} decay much faster. An example is the exponentially decaying coefficients $\theta_j=\exp(-j^{\alpha_2})$ for some $\alpha_2>0$. In this scenario, the asymptotic improvement of MA over MS is negligible \citep{Peng2021}.

\begin{proposition}\label{lemma:delta_1}
    Suppose Assumptions~\ref{asmp:square_summable}--\ref{asmp:regressor_order} hold. When both Condition~\ref{condition1} and $M_n\gtrsim m_n^*$ are satisfied, we have
    $
    R_n\left(\bw_N^*|\mathcal{M}_S,\bbf\right)-R_n\left(\bw^*|\mathcal{M}_S,\bbf\right) \asymp R_n(\bw^*|\mathcal{M}_S,\bbf).
    $
When either Condition~\ref{condition2} or $M_n=o(m_n^*)$ holds, we have
    $
    R_n\left(\bw_N^*|\mathcal{M}_S,\bbf\right)-R_n\left(\bw^*|\mathcal{M}_S,\bbf\right)=o\left[ R_n(\bw^*|\mathcal{M}_S,\bbf)\right].
    $
\end{proposition}

Proposition~\ref{lemma:delta_1} theoretically clarifies the effects of weight discretization and $M_n$ on the optimal MA risk (see also \cite{Xu2022From}). When $\theta_l$ decays slowly and $M_n$ is large, the difference $R_n\left(\bw_N^*|\mathcal{M}_S,\bbf\right)-R_n\left(\bw^*|\mathcal{M}_S,\bbf\right)$ is of the same order as the risk $R_n(\bw^*|\mathcal{M}_S,\bbf)$. In this case, weight discretization increases the optimal MA risk in the general continuous weight set $\mathcal{W}_{|\mathcal{M}_S|}$ by a significant fraction. However, when $\theta_l$ decays fast or $M_n$ is small relative to the size of the optimal nested model, the discrete weight set restriction asymptotically does not influence the optimal risk of MA. Consider a representative setting where $M_n=p_n=n$ and $\mathcal{M}_S = \mathcal{M}_A$ in the following two examples. The proofs for all examples in this paper can be found in the Appendix.
\begin{example}[Polynomially decaying coefficients]
  When $\theta_j=j^{-\alpha_1}, \alpha_1>1/2$, we have $m_n^* \sim (n/\sigma^2)^{1/(2\alpha_1)}$ and $R_n(\bw^*|\mathcal{M}_A,\bbf)\asymp m_n^*/n \asymp n^{-1+1/(2\alpha_1)}$. There exists a constant $1<C<2$ such that  $R_n\left(\bw_N^*|\mathcal{M}_A,\bbf\right)>CR_n(\bw^*|\mathcal{M}_A,\bbf)$ when $n$ is sufficiently large. Thus, it is impossible to achieve the full AOP by using the discrete weight set with any fixed $N$.
\end{example}
\begin{example}[Exponentially decaying coefficients]
  Now the transformed coefficients decay fast: $\theta_j=\exp(-cj^{\alpha_2})$, $\alpha_2>0$. In this case, we have $m_n^{\ast}\sim [\log(n/\sigma^2)^{1/(2c)}]^{1/\alpha_2}$, and the full AOP is achievable due to $R_n\left(\bw_N^*|\mathcal{M}_A,\bbf\right)\sim R_n(\bw^*|\mathcal{M}_A,\bbf)$.
\end{example}

\begin{remark}
Note that MS can be viewed as MA in the discrete set $\mathcal{W}_{|\mathcal{M}_S|}(1)$. A natural question to ask is whether $R_n\left(\bw_N^*|\mathcal{M}_S,\bbf\right)$ has a substantial improvement over $R_n\left(m^*|\mathcal{M}_S,\bbf\right)$. In Section~\ref{sec:A_5_1} of the Appendix, we show that $R_n\left(\bw_N^*|\mathcal{M}_S,\bbf\right)$ still significantly reduces $R_n\left(m^*|\mathcal{M}_S,\bbf\right)$ when $\theta_l$ decays slowly and $M_n$ is large.
\end{remark}

\subsection{Restriction of the candidate model set}\label{subsec:review_wan}

Directly minimizing the MMA criterion over the continuous weight set $\mathcal{W}_{|\mathcal{M}|}$ and the non-nested $\mathcal{M}$ was considered by \cite{Wan2010}. In their original paper, however, there is no theoretical guidance on how to construct $\mathcal{M}$. Subsequently, \cite{Zhang2020} suggested setting $\mathcal{M}$ as $\mathcal{M}_{S}$ after reordering the regressors. As will be seen next, this reorder-then-combine strategy fails to attain the optimal rate when Condition (\ref{eq:condition_wan}) is required, let alone achieving the full/ideal AOP.

\begin{proposition}\label{lemma:delta_2}
    Suppose Assumptions~\ref{asmp:square_summable}--\ref{asmp:regressor_order} are satisfied. Under Condition~\ref{condition1}, a necessary condition of (\ref{eq:condition_wan}) is $M_n=o\left(m_n^* \right)$. In such a case, we have
    \begin{equation}\label{eq:prop3}
      R_n(m^*|\mathcal{M}_A,\bbf)=o[R_n(\bw^*|\mathcal{M}_S,\bbf)].
    \end{equation}
    Under Condition~\ref{condition2}, for (\ref{eq:condition_wan}) to hold, it is also necessary to require $M_n\leq \lfloor C m_n^* \rfloor $ with a constant $0<C<1$. In this case, (\ref{eq:condition_wan}) still leads to the relation (\ref{eq:prop3}).
\end{proposition}

Proposition~\ref{lemma:delta_2} confirms that $\mathcal{M}_S$ excludes the optimal single model $m_n^*$ under the widely used condition (\ref{eq:condition_wan}). When $\theta_l$ decays slowly, the restricted MA strategy based on $\mathcal{M}_S$ has a significant disadvantage compared to MS in terms of rate of convergence, which is against the motivation of MA. When $\theta_l$ decays fast, MA generally does not have any real benefit compared to MS \citep{Peng2021}. The restricted MMA with (\ref{eq:condition_wan}), however is actually worse. Recall that $R_n(m^*|\mathcal{M}_A,\bbf)$ converges at the same rate of $R_n(\bw^*|\mathcal{M}_A,\bbf)$ regardless of Condition~\ref{condition1} or Condition~\ref{condition2} \citep{Peng2021}. An undesirable consequence of restricting $\mathcal{M}_S$ with (\ref{eq:condition_wan}) is that the resulting MA estimators cannot even attain the rate of $R_n(\bw^*|\mathcal{M}_A,\bbf)$. Consider the following examples for a more specific illustration.

\addtocounter{example}{-2}
\begin{example}[continued]
  When $\theta_j=j^{-\alpha_1}, \alpha_1>1/2$, and $p_n=n$, Condition (\ref{eq:condition_wan}) is equivalent to the restriction on the rate of increase of the number of candidate models in $\mathcal{M}_S$
\begin{equation}\label{eq:M_n_order1}
M_n=\left\{\begin{array}{ll}
o(n^{\frac{1}{2\alpha_1+1}} ) &\quad 1/2<\alpha_1<1, \\
o(n^{\frac{1}{4\alpha_1-1}} ) &\quad \alpha_1 \geq 1.
\end{array}\right.
\end{equation}
Therefore we need $M_n=c_n(m_n^*)^{2\alpha_1/(2\alpha_1+1)}$ with $c_n \to 0$ as $n \to \infty$, where $m_n^*\sim(n/\sigma^2)^{1/(2\alpha_1)}$. In this case, the optimal rate is $R_n(\bw^*|\mathcal{M}_A,\bbf) \asymp R_n(m^*|\mathcal{M}_A,\bbf)\asymp n^{-1+1/(2\alpha_1)}$ \citep{Peng2021}. But the rate of convergence of MA based $\mathcal{M}_S$ is $M_n^{-2\alpha_1+1}$, which converges no faster than $n^{-(2\alpha_1-1)/(2\alpha_1+1)}$ and thus much slower than the optimal rate. For a specific example, if $\alpha_1=1$, the MMA converges slower than $n^{-1/3}$ in contrast to the optimal rate $n^{-1/2}$.

\end{example}
\begin{example}[continued]
  Consider $\theta_j=\exp(-cj^{\alpha_2})$, $\alpha_2>0$. A sufficient condition for (\ref{eq:condition_wan}) is $M_n<\left(1/2 \right)^{1/\alpha_2}m_n^*$, where $m_n^{\ast}\sim [\log(n/\sigma^2)^{1/(2c)}]^{1/\alpha_2}$. In this case, MA based on $\mathcal{M}_S$ converges at the rate of $M_n^{1-\alpha_2}/n^{1/2}$, which is still slower than the optimal MS rate $m_n^*/n$.
\end{example}

Another interesting case is that there are $s \leq l$ non-zero coefficients, where $l$ is a fixed integer. In such a case, typically we have $m_n^*=s$, and $M_n\leq s-1$ since $nR_n\left(\bw^*|\mathcal{M}_S,\bbf  \right)\to \infty$ is a necessary condition for (\ref{eq:condition_wan}). This implies that MA converges more slowly (or not at all) than $m_n^*$ because $\mathcal{M}_S$ excludes at least one non-zero signal.

Comparing the two restricted-AOP theories given by \cite{Hansen2007} and \cite{Wan2010, Zhang2020}, it seems that MA with the discrete weight set is safer in the nested setting, since it always leads to the optimal MS rate when $M_n\gtrsim m_n^*$, while the reorder-then-combine strategy based on the restrictive candidate set does not. Nevertheless, both theories have the same drawbacks of not achieving the MA's full potential. Therefore, the questions Q1--Q3 raised earlier remain largely unanswered.

\begin{remark}
\cite{Zhang2021} and \cite{Fang_Yuan_Tian_2023} proved the MMA's AOP under some more interpretable assumptions
than (\ref{eq:condition_wan}). Following the proof in Proposition~\ref{lemma:delta_2} in the Appendix, we can see that these assumptions exclude $m_n^*$ and thus suffer the same consequence (\ref{eq:prop3}).
\end{remark}

\section{Nested candidate models}\label{sec:main_results}

\subsection{A risk bound}\label{sec:main_results:non-asymptotic}

Given a general nested candidate model set $\mathcal{M}=\{ \{1,2,\ldots,k_m \}: 1\leq k_1<k_2<\cdots<k_{M_n}\leq p_n\}$, define
\begin{equation}\label{eq:psi_M}
  \psi(\mathcal{M}) = \left[M_n\wedge\left(1+ \sum_{j=1}^{M_n-1}\frac{k_{j+1} - k_{j}}{4k_{j}}+\sum_{j=1}^{M_n'-1} \frac{S_{j-1} - S_{j}}{4S_{j}}\right)\right](1+\log M_n)^2,
\end{equation}
where $k_0=0$, $S_j=\sum_{l=k_j+1}^{k_{M_n}}\theta_l^2$ for $1\leq j \leq M_n-1$, $S_{M_n}=0$, and $M_n'=\min\{1 \leq j \leq M_n: S_j=0 \}$. Then we have the following upper bound on $Q_n(\hat{\bw}|\mathcal{M},\bbf)$.

\begin{theorem}\label{theorem:main}
Suppose that Assumption~\ref{asmp:square_summable} holds, then we have
\begin{equation}\label{eq:risk_bound_general}
\begin{split}
   Q_n(\hat{\bw}|\mathcal{M},\bbf)\leq R_n(\bw^*|\mathcal{M},\bbf)&+\frac{C\sigma^2}{n}\psi(\mathcal{M})+\frac{C\sigma}{\sqrt{n}}[\psi(\mathcal{M})]^{\frac{1}{2}}[R_n(\bw^*|\mathcal{M},\bbf)]^{\frac{1}{2}}\\
   &+C\rho\left(n, \mathcal{M},\bbf,\hat{\sigma}^2,\sigma^2\right),\\
\end{split}
\end{equation}
where $C$ is a positive constant that depends on the sub-Gaussian parameter of $\epsilon_i$, and $\rho\left(n, \mathcal{M},\bbf,\hat{\sigma}^2,\sigma^2\right)$ is the estimation error related to $\hat{\sigma}^2$, which is defined by
\begin{equation*}
\begin{split}
  \rho\left(n, \mathcal{M},\bbf,\hat{\sigma}^2,\sigma^2\right) & = \frac{k_{M_n}}{n\sigma^2}\mathbb{E}\left( \hat{\sigma}^2-\sigma^2 \right)^2+ \left[ \frac{k_{M_n}}{n\sigma^2}\mathbb{E}\left( \hat{\sigma}^2-\sigma^2 \right)^2 \right]^{\frac{1}{2}}[R_n(\bw^*|\mathcal{M},\bbf)]^{\frac{1}{2}}.
   \end{split}
\end{equation*}
\end{theorem}

The risk bound (\ref{eq:risk_bound_general}) is valid for any sample size and does not rely on Assumption~\ref{asmp:regressor_order} that the transformed coefficients are ordered. Note that the risk of the MMA estimator $Q_n(\hat{\bw}|\mathcal{M},\bbf)$ is bounded by the infeasible optimal MA risk $R_n(\bw^*|\mathcal{M},\bbf)$ plus three additional terms. The first two terms are related to the candidate model set $\mathcal{M}$. The third term $\rho\left(n, \mathcal{M},\bbf,\hat{\sigma}^2,\sigma^2\right)$ is mainly about the estimation error of $\hat{\sigma}^2$.

\subsection{Estimation of $\sigma^2$}\label{sec:estimator_sigma}

As the risk bound suggests, the variance estimation may also have a significant effect on the performance of MMA. When a poor estimator of $\sigma^2$ with non-converging squared risk is considered, the upper bound in (\ref{eq:risk_bound_general}) becomes non-converging if the largest size $k_{M_n}$ is of order $n$. In contrast, when $\mathbb{E}\left( \hat{\sigma}^2-\sigma^2 \right)^2$ converges at the parametric rate $1/n$, the term $\rho\left(n, \mathcal{M},\bbf,\hat{\sigma}^2,\sigma^2\right)$ does not affect the rate of convergence of the upper bound. Here we give two specific estimators for $\sigma^2$.

Consider a model-based estimator from the least squares theory
\begin{equation}\label{eq:sigma_estimator_1}
  \hat{\sigma}_{m_n}^2=\frac{1}{n-m_n}\| \by - \bP_{m_n|\mathcal{M}_A}\by \|^2.
\end{equation}
Section~\ref{sec:proof:eq:variance_risk} of the Appendix proves
\begin{equation}\label{eq:variance_risk}
  \mathbb{E}(\hat{\sigma}_{m_n}^2-\sigma^2)^2\lesssim \frac{1}{n-m_n}\vee\frac{n\|\btheta_{-m_n} \|^2}{(n-m_n)^2} \vee  \frac{n^2 \|\btheta_{-m_n} \|^4}{(n-m_n)^2},
\end{equation}
where $\btheta_{-m_n}=(\theta_{m_n+1},\ldots, \theta_{p_n})^{\top}$. When $n-p_n \asymp n$, the variance estimator $\hat{\sigma}_{p_n}^2$ based on the largest candidate model converges at the parametric rate $1/n$. When $p_n = n$, the estimation error of $\hat{\sigma}_{m_n}^2$ with $m_n = \lfloor \kappa n \rfloor$ ($0<\kappa<1$) is not slower than $(1/n)\vee \|\btheta_{-m_n} \|^4$. As will be seen in the next subsection, $\hat{\sigma}_{m_n}^2$ may be sufficient for the full AOP of MMA (e.g., in the examples of polynomially and exponentially decaying coefficients), even if it does not converge at the parametric rate in some cases.

Moreover, when $p_n=n$, the first difference variance estimator proposed by \cite{rice1984bandwidth} can also be used. For the one-dimensional nonparametric regression $y_i=f(u_i)+\epsilon_i$, where the model (\ref{eq:model}) is a linear approximation for $f$, consider
$
  \hat{\sigma}_{D}^2=\frac{1}{2(n-1)}\sum_{i=2}^{n}\left[y_{(i+1)}-y_{(i)}\right]^2,
$
where $y_{(i)}$ denotes the observed response at the $i$-th smallest $u$ value. Under a mild smoothness assumption on $f$, $\hat{\sigma}_{D}^2$ has the property $\mathbb{E}(\hat{\sigma}_D^2 -\sigma^2)^2 \sim cn^{-1}\mbox{\mbox{Var}}(\epsilon^2)$. This estimator can be extended to design points in a multidimensional case \citep{Munk2005}.

\subsection{The full AOP}\label{subsec:aop}

With a suitable estimator $\hat{\sigma}^2$, the AOP of MMA is readily available as shown in the following theorem.

\begin{theorem}\label{theorem:aop}
  Suppose Assumption~\ref{asmp:square_summable} holds, if
  \begin{equation}\label{eq:variance_rate}
    k_{M_n}\mathbb{E}\left( \hat{\sigma}^2-\sigma^2 \right)^2=o\left[nR_n(\bw^*|\mathcal{M},\bbf)\right]
  \end{equation}
  and
  \begin{equation}\label{eq:minimum_marisk_rate}
    \psi(\mathcal{M})=o\left[nR_n(\bw^*|\mathcal{M},\bbf)\right],
  \end{equation}
  then $\hat{\bbf}_{\hat{\bw} | \mathcal{M}}$ achieves AOP in the sense that (\ref{eq:opt1}) holds for the continuous weight set $\mathcal{W}_{|\mathcal{M}|}$. In particular, using the variance estimator (\ref{eq:sigma_estimator_1}) with $m_n = \lfloor \kappa n \rfloor \wedge p_n$ ($0<\kappa<1$), if
  \begin{equation}\label{eq:variance_rate_2}
    (1/n)\vee \|\btheta_{-m_n} \|^4=o\left[R_n(\bw^*|\mathcal{M}_A,\bbf)\right]
  \end{equation}
  and
  \begin{equation}\label{eq:minimum_marisk_rate_2}
    \left[ \log p_n +  \varphi(\btheta) \right](\log p_n)^2=o\left[nR_n(\bw^*|\mathcal{M}_A,\bbf)\right],
  \end{equation}
  where $\varphi(\btheta)\triangleq\sum_{j=1}^{p_n'-1} (\theta_j^2/\sum_{l=j+1}^{p_n}\theta_l^2)$ and $p_n'=\min\{1 \leq j \leq p_n: \sum_{l=j+1}^{p_n}\theta_l^2=0 \}$, then $\hat{\bbf}_{\hat{\bw} | \mathcal{M}_A}$ achieves the full AOP in terms of (\ref{eq:fullopt}).
\end{theorem}

Theorem~\ref{theorem:aop} establishes the MMA's AOP for the general nested model set $\mathcal{M}$ and weight set $\mathcal{W}_{|\mathcal{M}|}$ with variance estimation. Compared with the restricted-AOP theory in \cite{Hansen2007}, our result does not restrict the model weights to the discrete set $\mathcal{W}_{|\mathcal{M}|}(N)$. As demonstrated in Proposition~\ref{lemma:delta_1}, relaxing the model weights from $\mathcal{W}_{|\mathcal{M}|}(N)$ to $\mathcal{W}_{|\mathcal{M}|}$ improves the optimal MA risk substantially in various situations. Second, the condition~(\ref{eq:minimum_marisk_rate}) in Theorem~\ref{theorem:aop} substantially improves the condition (\ref{eq:condition_wan}) in \cite{Wan2010} by allowing most helpful models in the nested setting. In fact, Theorem~\ref{theorem:aop} permits the use of the largest candidate model set $\mathcal{M}_A$. The conditions (\ref{eq:variance_rate_2})--(\ref{eq:minimum_marisk_rate_2}) are two sufficient conditions of (\ref{eq:variance_rate})--(\ref{eq:minimum_marisk_rate}) when $\hat{\sigma}^2=\hat{\sigma}_{m_n}^2$ and $\mathcal{M}=\mathcal{M}_A$. When $p_n=n$ and $m_n = \lfloor \kappa n \rfloor$, $0<\kappa<1$, the condition (\ref{eq:variance_rate_2}) holds in Examples~1--2, and (\ref{eq:minimum_marisk_rate_2}) is satisfied with $p_n'=n$ when $|\theta_j|$ decays polynomially, as seen below.

\addtocounter{example}{-2}
\begin{example}[continued]\label{example:11}
When the coefficients decay as $\theta_j=j^{-\alpha_1}$, $\alpha_1>1/2$, we have $nR_n(\bw^*|\mathcal{M}_A,\bbf) \asymp n^{1/(2\alpha_1)}$. The condition (\ref{eq:variance_rate_2}) is satisfied since $\|\btheta_{-m_n} \|^4 =o\left[R_n(\bw^*|\mathcal{M}_A,\bbf)\right]$. And the left side of (\ref{eq:minimum_marisk_rate_2}) is upper bounded by $(\log n)^3$. Thus, the condition (\ref{eq:minimum_marisk_rate_2}) also holds.

\end{example}
\begin{example}[continued]

When $\theta_j=\exp(-j^{\alpha_2})$, $\alpha_2>0$, we have $nR_n(\bw^*|\mathcal{M}_A,\bbf) \asymp  (\log n)^{1/\alpha_2}$. The condition (\ref{eq:variance_rate_2}) is still satisfied by noting $\|\btheta_{-m_n} \|^4 =O[n^{2-2\alpha_2}\exp( - 4(\kappa n)^{\alpha_2} )]=o(1/n)$. However, (\ref{eq:minimum_marisk_rate_2}) does not hold due to $\psi(\mathcal{M}_A)/[nR_n(\bw^*|\mathcal{M}_A,\bbf)] \to \infty$.

\end{example}

Based on the above analysis, we observe that with the optimal single model $m_n^*$ being included in $\mathcal{M}_A$ in both examples, the full AOP is achieved for the MMA estimator based on $\mathcal{M}_A$ when the coefficients do not decay too fast. Recalling that Example~1 serves as an important scenario for preferring MA to MS \citep{Peng2021}, our result here shows that the substantial risk reduction of MA over MS can be realized by utilizing $\mathcal{M}_A$. In contrast, the theory of \cite{Hansen2007} fails to justify the full AOP in Example 1, and the result given by \cite{Wan2010} does not even attain the optimal rate in either of the two examples.

\subsection{Construction of candidate model set}\label{sec:reduced}

This subsection proposes two types of nested candidate model sets, on which the MMA estimators achieve the full AOP on broader parameter regions than that based on $\mathcal{M}_A$.

\subsubsection{Candidate model set with grouped regressors}\label{subsec:grouped}

Instead of combining the candidate models with successively increasing sizes, we consider a smaller set $\mathcal{M}_G=\{ \{1,2,\ldots,k_m \}: 1\leq k_1<k_2<\cdots<k_{M_n}\leq p_n\}$, where the size of each candidate model is group-wise added. Let $k_0=0$, $k_{M_n}=p_n$, and $\max_{1\leq j \leq M_n-1}[(k_{j+1}-k_j)/(k_j-k_{j-1})]\leq 1+\zeta_n$, where $\zeta_n \geq 0$.

\begin{theorem}\label{cor:grouped}
  Suppose Assumption~\ref{asmp:square_summable} holds. If $\zeta_n =o(1)$, $k_1\vee \psi(\mathcal{M}_G)=o[nR_n(\bw^*|\mathcal{M}_A,\bbf)]$, and $p_n\mathbb{E}\left( \hat{\sigma}^2-\sigma^2 \right)^2=o[nR_n(\bw^*|\mathcal{M}_A,\bbf)]$, then $\hat{\bbf}_{\hat{\bw}|\mathcal{M}_G}$ achieves the full AOP in terms of (\ref{eq:fullopt}).
\end{theorem}

Theorem~\ref{cor:grouped} indicates that, rather than combining all nested models, the MMA estimator based on $\mathcal{M}_G$ still achieves the optimal risk $R_n(\bw^*|\mathcal{M}_A,\bbf)$ asymptotically when the sizes of the candidate models are appropriately selected. Here we provide a specific construction of $\mathcal{M}_G$ inspired by the weakly geometrically increasing groups in \cite{cavalier2001penalized}. For two constants $t_1>0$ and $t_2>0$, define $\zeta_n = t_1/(\log n)^{t_2}$. Consider $\mathcal{M}_{G1}$ with $k_1 = \lceil \zeta_{n}^{-1} \rceil$, $k_m  =k_{m-1} +\lfloor k_1(1+\zeta_{n})^{m-1}\rfloor$ for $m=2,\ldots,M_n-1$, and $k_{M_n}=p_n$, where $M_n = \arg\min_{m \in \mathbb{N}}\{ (k_1 + \sum_{j=2}^{m}\lfloor k_1(1+\zeta_{n})^{j-1}\rfloor) \geq p_n \}$. Then the conditions in Theorem~\ref{cor:grouped} are satisfied with milder restrictions on the rate of decay of $|\theta_j|$, as seen in the following two examples. To keep in line with the analysis in Section~\ref{subsec:aop}, we here focus on the case $p_n = n$ and set $\zeta_n = t_1/(\log n)^{t_2}$ with $t_2=1$.
\addtocounter{example}{-2}
\begin{example}[continued]
\label{example1}

When $\theta_j=j^{-\alpha_1}$, $\alpha_1>1/2$, we have $nR_n(\bw^*|\mathcal{M}_A,\bbf) \asymp n^{1/(2\alpha_1)}$ and $\psi(\mathcal{M}_{G1}) \lesssim (\log n)^2(\log\log n)^2$. We see that $k_1/n^{1/(2\alpha_1)} \to 0$, $\psi(\mathcal{M}_{G1})/n^{1/(2\alpha_1)} \to 0$, and the condition $n\mathbb{E}\left( \hat{\sigma}^2-\sigma^2 \right)^2=o[n^{1/(2\alpha_1)}]$ is satisfied for the variance estimators given in Section~\ref{sec:estimator_sigma}. Thus, the MMA estimator based on $\mathcal{M}_{G1}$ attains the same full AOP property as that based on $\mathcal{M}_{A}$.

\end{example}

\begin{example}[continued]
\label{example2}

When $\theta_j=\exp(-j^{\alpha_2})$, $0<\alpha_2< 1/2$, we see $nR_n(\bw^*|\mathcal{M}_A,\bbf) \asymp (\log n)^{1/\alpha_2}$, $k_1/(\log n)^{1/\alpha_2} \to 0$ and $\psi(\mathcal{M}_{G1}) \lesssim (\log n)^2(\log\log n)^2 = o[(\log n)^{1/\alpha_2}]$. Thus the MMA estimator with $\mathcal{M}_{G1}$ achieves the full AOP on a broader parameter region compared to that based on $\mathcal{M}_{A}$.

\end{example}

\begin{remark}
  Another construction of $\mathcal{M}_{G}$ involves equal-sized groups. Consider $\mathcal{M}_{G2}$ with $\zeta_n = 0$, $k_1 = \lceil(\log n)^{t}\rceil $, $k_m = mk_1$ for $m=2,\ldots,M_n-1$, and $k_{M_n}=p_n$, where $0<t<\infty$. We have $\psi(\mathcal{M}_{G2}) \lesssim (\log M_n)^3 \lesssim (\log n - t \log\log n)^3$. Thus the MMA estimator based on $\mathcal{M}_{G2}$ still attains the full AOP in the case of polynomially decaying coefficients. However, it cannot be proved to be full AOP by the current theory when $\theta_j=\exp(-j^{\alpha_2})$, $0<\alpha_2< \infty$.
\end{remark}

\subsubsection{Candidate model set based on MS}\label{subsec:minimum}

This subsection considers a smaller number of nested models with the size centering on $m_n^*$. Since $m_n^*$ is unknown in practice, we estimate it by some MS method and then consider the candidate model set $\hat{\mathcal{M}}_{MS}=\hat{\mathcal{M}}_{MS}(\kappa_l, \kappa_u)=\{ \{1,\ldots,m \} : \hat{l}_n \leq m \leq \hat{u}_n \}$, where $\hat{l}_n=1\vee\left\lfloor \kappa_l^{-1}\hat{m}_n\right\rfloor$, $\hat{u}_n= p_n\wedge \left\lfloor \kappa_u\hat{m}_n\right\rfloor $, $\kappa_l>1$, and $\kappa_u>1$.

To get asymptotic properties of $\hat{\mathcal{M}}_{MS}$, we need another assumption on the transformed coefficients, which is satisfied for both polynomially and exponentially decaying coefficients.
\begin{assumption}\label{asmp:regressor_order2}
  The transformed coefficients satisfy $\lim_{\kappa \to \infty}\left|\theta_{\lfloor \kappa l\rfloor}/\theta_l\right| \to 0$ for any $l \in \mathbb{N}$.
\end{assumption}
Define $c_1$ and $c_2$ two constants with $0<c_1<1<c_2$. Let $F_n$ denote the event $\left\lfloor c_1m_n^*\right\rfloor\leq\hat{m}_n\leq \left\lfloor c_2m_n^*\right\rfloor$ and $\bar{F}_n$ be its complement. Consider a candidate model set $\hat{\mathcal{M}}_{MS1}=\hat{\mathcal{M}}_{MS}(\kappa_l, \kappa_u)$ with $\kappa_l \to \infty$ and $\kappa_u \to \infty$.

\begin{theorem}\label{cor:minimum_1}
Suppose that Assumptions~\ref{asmp:square_summable}--\ref{asmp:regressor_order2} hold. If the condition (\ref{eq:variance_rate}) is satisfied for $\mathcal{M}_A$, $\mathbb{E}\psi(\hat{\mathcal{M}}_{MS1}) = o(m_n^*)$, and the event $F_n$ satisfies $\mathbb{P}(\bar{F}_n)= o\left(m_n^*/n\right)$, then we have $\mathbb{E}Q_n(\hat{\bw}|\hat{\mathcal{M}}_{MS1},\bbf)\leq[1+o(1)] R_n\left(\bw^*|\mathcal{M}_A,\bbf\right)$.
\end{theorem}

Theorem~\ref{cor:minimum_1} states that MMA achieves the full AOP with the reduced candidate model set $\hat{\mathcal{M}}_{MS1}$. To verify the regularity conditions in this theorem, we set $\kappa_l=\kappa_u=\log n$ and select $\hat{m}_n$ by Mallows’ $C_p$ MS criterion \citep{Mallows1973}. Consider the following two examples when $\sigma^2$ is known.
\addtocounter{example}{-2}
\begin{example}[continued]

When $\theta_j=j^{-\alpha_1}, \alpha_1>1/2$, we have $m_n^* \asymp n^{1/(2\alpha_1)}$, $\mathbb{E}\psi(\hat{\mathcal{M}}_{MS1})\lesssim (\log\log n + \log m_n^*)^3 = o\left(m_n^*\right)$, and $\mathbb{P}\left(\bar{F}_n\right) \lesssim \exp{[-C(m_n^*)^{1/2}]}=o\left(m_n^*/n\right)$ for any fixed $C$, which meet the conditions in Theorem~\ref{cor:minimum_1}. Thus, the MMA estimator with $\hat{\mathcal{M}}_{MS1}$ retains the full AOP as that based on $\mathcal{M}_A$ when the transformed coefficients decay slowly.

\end{example}

\begin{example}[continued]

When $\theta_j=\exp(-j^{\alpha_2}),0<\alpha_2< 1/2$, we have $m_n^* \asymp (\log n)^{1/\alpha_2}$, $\mathbb{E}\psi(\hat{\mathcal{M}}_{MS1}) \precsim (\log n)^{\alpha_2+1} (\log\log n)^2 = o\left(m_n^*\right)$, and $\mathbb{P}\left(\bar{F}_n\right) = o\left( 1/n \right) = o(m_n^*/n)
$
for any constant $C$. Thus, $\hat{\mathcal{M}}_{MS1}$ also expands the region for the full AOP of $\mathcal{M}_A$ when the coefficients decay fast.

\end{example}

\begin{remark}
  Let $\hat{\mathcal{M}}_{MS2}=\hat{\mathcal{M}}_{MS}$ with $\kappa_l\vee \kappa_u$ being upper bounded by some positive constant $C$. In Section~\ref{sec:MS_bounded} of the Appendix, we prove that the MMA estimator based on $\hat{\mathcal{M}}_{MS2}$ cannot achieve the full potential of MA under Condition~\ref{condition1}. But it can reach the full AOP under Condition~\ref{condition2}.
\end{remark}

Table~\ref{tab:method} summarizes the MMA strategies discussed in Sections~\ref{sec:review}--\ref{sec:main_results}. We emphasize that the parameter regions given in the last two columns are the known sufficient conditions for the full AOP of MMA. Whether these methods achieve the full AOP in
larger regions remains open. Overall, the MMA estimators based on $\mathcal{M}_{G1}$ exhibit widest parameter regions for achieving the full AOP when the variance $\sigma^2$ is unknown. The roles of $\mathcal{M}_{G1}$ in the non-nested setting and minimax theory will be further explored in Sections~\ref{sec:non-nested}--\ref{sec:minimax}, respectively.

\begin{table}[!htbp]
  \centering
  \caption{MA methods with different weight set or candidate model set restrictions. The last two columns summarize the ranges of $\alpha_1$ and $\alpha_2$ on which MMA is shown to achieve the full AOP in two representative examples respectively. }
  \resizebox{\columnwidth}{!}{
    \begin{tabular}{clllll}
    \hline
          & Method & Candidate model set  & Weight set & $\theta_j=j^{-\alpha_1}$ & $\theta_j=\exp\left(-j^{\alpha_2}\right)$ \\
          \hline
    \cite{Hansen2007} & WR & $\mathcal{M}_A$     & $\mathcal{W}_{|\mathcal{M}_A|}(N)$ with fixed $N\geq 1$     & $\emptyset$     & $(0, + \infty)$ \\
    \cite{Wan2010}      & MR & $\mathcal{M}_S$ with (\ref{eq:condition_wan})     & $\mathcal{W}_{|\mathcal{M}_S|}$     & $\emptyset$      & $\emptyset$  \\
    \hline
     & M-ALL & $\mathcal{M}_A$     & $\mathcal{W}_{|\mathcal{M}_A|}$     & $(1/2,+\infty)$     & $\emptyset$ \\
          & M-G1 & $\mathcal{M}_{G1}$     & $\mathcal{W}_{|\mathcal{M}_{G1}|}$     & $(1/2,+\infty)$     & $(0,1/2)$ \\
      This paper    & M-G2 & $\mathcal{M}_{G2}$     & $\mathcal{W}_{|\mathcal{M}_{G2}|}$     & $(1/2,+\infty)$    & $\emptyset$ \\
          & M-MS1 & $\hat{\mathcal{M}}_{MS1}$     & $\mathcal{W}_{|\hat{\mathcal{M}}_{MS1}|}$     & $(1/2,+\infty)$     & $(0,1/2)$ \\
          & M-MS2 & $\hat{\mathcal{M}}_{MS2}$     & $\mathcal{W}_{|\hat{\mathcal{M}}_{MS2}|}$     & $\emptyset$      & $(0, + \infty)$ \\
          \hline
    \end{tabular}%
    }
  \label{tab:method}%
\end{table}%

\section{Non-nested candidate models}\label{sec:non-nested}

We now turn to the non-nested setup. Suppose $\{\bphi_1,\ldots,\bphi_{p_n} \}$ is a given orthonormal basis of the column space of $\bX$, which can be the eigenvectors of $\bP_{j|\mathcal{M}_A}-\bP_{(j-1)|\mathcal{M}_A},j=1,\ldots,p_n$ as introduced in Section~\ref{sec:reparameterization} based on the given arbitrary order of the regressors, or generated by other orthogonalization algorithms. For consistency in notation, define $\theta_j\triangleq\bphi_j^{\top}\bbf/\sqrt{n},j=1,\ldots,p_n$. Let $\hat{\bbf}_{\bw | \mathcal{M}_{AS}}$ and $\hat{\bbf}_{m | \mathcal{M}_{AS}}$ denote the MA and MS estimators based on all subsets of $\{\bphi_1,\ldots,\bphi_{p_n}\}$. Define $R_n(\bw^*|\mathcal{M}_{AS},\bbf)=\min_{\bw \in \mathcal{W}_{2^{p_n}}}R_n(\bw|\mathcal{M}_{AS},\bbf)$ and $R_n\left(m^*|\mathcal{M}_{AS},\bbf\right)=\min_{m \in \{1,\ldots,2^{p_n} \}}R_n\left(m|\mathcal{M}_{AS},\bbf\right)$ as the ideal MA and MS risks, respectively.

\subsection{Improvability of the ideal MA risk over the ideal MS risk}\label{sec:improvability}

For any $\btheta = (\theta_1,\ldots,\theta_{p_n})^{\top}$, let $|\theta|_{(1)},\ldots,|\theta|_{(p_n)}$ denote a decreasing rearrangement of $|\theta_1|,\ldots,|\theta_{p_n}|$. The following conditions extend Conditions~\ref{condition1}--\ref{condition2} to distinguish two different rates of decay of the ordered coefficients.
\begin{condition}\label{condition3}
(Slowly decaying ordered coefficients) There exist constants $\kappa>1$ and $0<\delta\leq\nu<1$ with $\kappa\nu^2<1$ such that $\delta\leq|\theta|_{(\lfloor \kappa l \rfloor)}/|\theta|_{(l)}\leq\nu$ when $l$ is large enough.
\end{condition}
\begin{condition}\label{condition4}
(Fast decaying ordered coefficients) For every constant $\kappa>1$, the ordered coefficients satisfy $\lim\limits_{l \to \infty}|\theta|_{(\lfloor \kappa l \rfloor)}/|\theta|_{(l)}= 0$.
\end{condition}

While we already know $R_n\left(\bw^*|\mathcal{M}_{AS},\bbf\right) \leq R_n\left(m^*|\mathcal{M}_{AS},\bbf\right)$, the following theorem examines the scenarios under which $R_n\left(\bw^*|\mathcal{M}_{AS},\bbf\right)$ can significantly reduce $R_n\left(m^*|\mathcal{M}_{AS},\bbf\right)$.

\begin{theorem}\label{theo:improvability}
Suppose Assumption~\ref{asmp:square_summable} holds. Under Condition~\ref{condition3}, the risk reduction satisfies $R_n\left(m^*|\mathcal{M}_{AS},\bbf\right) - R_n\left(\bw^*|\mathcal{M}_{AS},\bbf\right) \asymp R_n\left(m^*|\mathcal{M}_{AS},\bbf\right)$. Under Condition~\ref{condition4}, the risk reduction satisfies $R_n(m^*|\mathcal{M}_{AS},\bbf) - R_n(\bw^*|\mathcal{M}_{AS},\bbf) = o[R_n(m^*|\mathcal{M}_{AS},\bbf)]$.
\end{theorem}

Theorem~\ref{theo:improvability} states that the potential advantage of MA over MS is substantial when the ordered coefficients decay slowly. However, in the case of fast decaying $|\theta|_{(j)}$, there is no asymptotic improvement of MA over MS. Therefore, to justify the superiority of MA, it is of great significance to investigate whether $R_n\left(\bw^*|\mathcal{M}_{AS},\bbf\right)$ can be realized under Condition~\ref{condition3}.

\subsection{Achievability of the ideal AOP}\label{sec:achievability}

For a vector $\be$, define $\mathfrak{S}_{\be}$ as the set of all vectors whose elements are permutations of the elements in $\be$.
\begin{assumption}\label{asmp:partial_order}
  (Weakly ordered coefficients) Suppose there exists a list of indices $0=d_0<d_1<\cdots < d_{D_n}=p_n$ such that $\max_{1\leq l \leq D_n-1}[(d_{l+1}-d_l)/(d_l-d_{l-1})]\leq 1+z_n$ with $z_n \to 0$, and $(|\theta_{d_{l-1}+1}|,\ldots,|\theta_{d_{l}}|)^{\top} \in \mathfrak{S}_{(|\theta|_{(d_{l-1}+1)},\ldots,|\theta|_{(d_{l})})^{\top}}$ for $1\leq l \leq D_n$.
\end{assumption}

Obviously, the completely ordered $\btheta$ in Assumption~\ref{asmp:regressor_order} satisfies Assumption~\ref{asmp:partial_order} since $(|\theta_{d_{j-1}+1}|,\ldots,|\theta_{d_{l}}|)^{\top} = (|\theta|_{(d_{l-1}+1)},\ldots,|\theta|_{(d_{l})})^{\top}\in \mathfrak{S}_{(|\theta|_{(d_{l-1}+1)},\ldots,|\theta|_{(d_{l})})^{\top}}$  for $1\leq l \leq D_n $. Differently, Assumption~\ref{asmp:partial_order} allows the coefficients with index between $d_{l-1}+1$ and $d_{l}$ to be totally unordered.
\begin{theorem}\label{theo:suff}

Suppose Assumption~\ref{asmp:square_summable} and Assumption~\ref{asmp:partial_order} hold. If $d_1 = o\left[nR_n(\bw^*|\mathcal{M}_{AS},\bbf)\right]$, then we have $R_n(\bw^*|\mathcal{M}_{A},\bbf) =\left[1+o(1) \right] R_n\left(\bw^*|\mathcal{M}_{AS},\bbf\right)$. Moreover, the nested MMA estimator $\hat{\bbf}_{\hat{\bw} | \mathcal{M}_{A}}$ achieves the ideal AOP (\ref{eq:idealopt}) if $d_1 \vee \psi(\mathcal{M}_A) = o\left[nR_n(\bw^*|\mathcal{M}_{AS},\bbf)\right]$, $\hat{\bbf}_{\hat{\bw} | \mathcal{M}_{G1}}$ achieves  (\ref{eq:idealopt}) provided $d_1 \vee k_1 \vee \psi(\mathcal{M}_{G1})=o\left[nR_n(\bw^*|\mathcal{M}_{AS},\bbf)\right]$.

\end{theorem}

Theorem~\ref{theo:suff} first establishes an asymptotic equivalence between $R_n\left(\bw^*|\mathcal{M}_A,\bbf\right)$ and $R_n(\bw^*|\mathcal{M}_{AS},\bbf)$. When $|\theta_j|,j\geq 1$ are weakly ordered, the nested MMA estimators achieve the ideal MA risk when the ordered coefficients do not decay too fast. Consequently, the asymptotic advantage of MA over MS, as characterized in Theorem~\ref{theo:improvability}, is also materialized.

\begin{example}
  Consider a representative example under Condition~\ref{condition3}: $|\theta|_{(j)}=j^{-\alpha_1}$, $\alpha_1>1/2$. Suppose $(|\theta_{d_{l-1}+1}|,\ldots,|\theta_{d_{l}}|)^{\top} \in \mathfrak{S}_{(|\theta|_{(d_{l-1}+1)},\ldots,|\theta|_{(d_{l})})^{\top}}$ for $1\leq l \leq D_n$, where $d_1\asymp\log n$, $d_{l+1} -d_{l} \leq (1+1/\log n)(d_{l} -d_{l-1})$ for $l=1,\ldots,D_n-1$. In this case, Assumption~\ref{asmp:partial_order} is satisfied with $z_n = 1/\log n$, and the conditions in Theorem~\ref{theo:suff} hold due to $nR_n(\bw^*|\mathcal{M}_{AS},\bbf) \asymp n^{1/(2\alpha_1)}$, $\psi(\mathcal{M}_A) \lesssim (\log n)^{3+2\alpha_1}$, and $k_1\vee\psi(\mathcal{M}_{G1}) \lesssim (\log n)^2(\log\log n)^2$. Combining with the results in Theorem~\ref{theo:improvability}, the nested MMA estimators $\hat{\bbf}_{\hat{\bw} | \mathcal{M}_A}$ and $\hat{\bbf}_{\hat{\bw} | \mathcal{M}_{G1}}$ have a significant advantage compared to any MS estimators in the non-nested setup.
\end{example}

Next, we present a negative result on the achievability of the ideal MA risk. Since our statement is on the negative side, we assume $\sigma^2$ is known, $p_n=n$, and $\epsilon_1, \ldots, \epsilon_n$ are i.i.d. $N(0,\sigma^2)$. Unlike Assumption~\ref{asmp:partial_order}, we now suppose the statistician has zero knowledge of the order of the coefficients in $\btheta$ in the sense that any permutation of the coefficient values is possible, which is perhaps the case in many applications.

\begin{theorem}\label{prop:nece}
Suppose Assumption~\ref{asmp:square_summable} holds, and $|\theta|_{(j)}= j^{-\alpha_1}$ with a fixed $\alpha_1>1/2$. For any estimator $\hat{\bbf}$, there must exist an $\bbf$ with $\btheta \in \mathfrak{S}_{(|\theta|_{(1)},\ldots,|\theta|_{(n)})^{\top}}$ such that
  \begin{equation*}
    \lim_{n\to \infty}\frac{R_n(\hat{\bbf},\bbf)}{R_n(\bw^*|\mathcal{M}_{AS},\bbf)} \to \infty.
  \end{equation*}
\end{theorem}

This theorem shows that even with the precise knowledge of magnitude of $|\theta|_{(j)}$, the ideal MA risk $R_n(\bw^*|\mathcal{M}_{AS},\bbf)$ remains unattainable by any estimation procedures without any prior information on the order of $|\theta_j|$. Theorem~\ref{prop:nece} also offers the first theoretical insight, albeit negative, into the reorder-then-combine strategy in \cite{Zhang2016jlm, Zhang2020}. It reveals that there is no data-based regressor ordering method that can achieve the real goal of non-nested MA in the absence of order assumption on the regression coefficients. Furthermore, based on the findings in Theorem~\ref{theo:improvability}, making a certain order assumption on the regression coefficients is essential to substantiate the asymptotic improvement of MA over MS.

\begin{remark}
Combining the results in Theorems~\ref{theo:suff}--\ref{prop:nece}, there may be possibilities that the ideal AOP can be achieved when Assumption~\ref{asmp:partial_order} is relaxed, although Theorem~\ref{prop:nece} rules out the possibility of achieving the ideal AOP without any perhaps subjective and hard to justify assumption on the order of the coefficients.
\end{remark}

\section{Minimax adaptivity}\label{sec:minimax}

Suppose the transformed coefficients $\btheta$ defined in (\ref{eq:trans_para}) belongs to the parameter space $\Theta \subseteq \mathbb{R}^{p_n}$, and the corresponding space of $\bbf$ is defined by $\mathcal{F}_{\Theta}=\{\bbf=\sqrt{n}\sum_{j=1}^{p_n}\theta_j\bphi_j:\btheta\in\Theta\}$. Define the minimax risk
$
  R_{M}(\mathcal{F}_{\Theta})=\inf_{\hat{\bbf}}\sup_{\bbf \in \mathcal{F}_{\Theta}}R_n(\hat{\bbf},\bbf),
$
where the infimum is over all estimator $\hat{\bbf}$. In addition, define the minimax risk of the linear-combined estimators
$
  R_L(\mathcal{F}_{\Theta})=\inf_{\bw}\sup_{\bbf \in \mathcal{F}_{\Theta}}R_n(\bw|\mathcal{M}_A,\bbf),
$
where $\inf_{\mathbf{w}}$ denote the infimum over all $\mathbf{w}\in\mathbb{R}^{p_n}$, and the subscript $L$ here is to emphasize that $\hat{\bbf}$ is restricted to the class of all the linear combinations of the models in $\mathcal{M}_A$. In this section, we assume $p_n=n$, and $\epsilon_1, \ldots, \epsilon_n$ are i.i.d. $N(0,\sigma^2)$ for simplicity.

We investigate the exact minimax adaptivity of the MMA estimator based on $\mathcal{M}_{G1}$ when the transformed coefficient $\btheta$ belongs to two types of classes, respectively. The first class is the ellipsoid
$
  \Theta(\alpha, R)=\{ \btheta \in \mathbb{R}^{n}:\sum_{j=1}^{n}j^{2\alpha}\theta_j^2\leq R \},
$
where $\alpha>0$ and $R>0$. Let $\mathcal{F}_{\Theta(\alpha, R)}=\{\bbf=\sqrt{n}\sum_{j=1}^{n}\theta_j\bphi_j:\btheta\in\Theta(\alpha, R)\}$ be the class of regression mean vector associated with $\Theta(\alpha, R)$. Let
$
  \Theta^{H}(c,q)=\{\btheta \in \mathbb{R}^{n}: |\theta_j|\leq cj^{-q},j=1,\ldots,n \}
$
be the hyperrectangle set of $\btheta$, where $c>0$ and $q>1/2$. And let $\mathcal{F}_{\Theta^{H}(c,q)}$ be the corresponding mean vector class of $\Theta^{H}(c,q)$.

\begin{theorem}\label{tho:minimax}
   Suppose $\hat{\sigma}_D^2$ or $\hat{\sigma}_{m_n}^2$ with $m_n = \lfloor \kappa n \rfloor$ $(0<\kappa<1)$ is adopted. Then the MMA estimator $\hat{\bbf}_{\hat{\bw}|\mathcal{M}_{G1}}$ is adaptive in the exact minimax sense on the family of the ellipsoids $\boldsymbol{\mathcal{F}}=\left\{\mathcal{F}_{\Theta(\alpha, R)}:\alpha>0,R>0\right\}$, i.e.,
   \begin{equation}\label{eq:minimax_all}
    \sup_{\bbf \in \mathcal{F}}\mathbb{E}L_n(\hat{\bw}|\mathcal{M}_{G1},\bbf)=[1+o(1)]R_M(\mathcal{F})
   \end{equation}
   holds for every $\mathcal{F}\in \boldsymbol{\mathcal{F}}$. And $\hat{\bbf}_{\hat{\bw}|\mathcal{M}_{G1}}$ is adaptive in the exact linear-combined minimax sense on the family of the hyperrectangles $\boldsymbol{\mathcal{F}}^{H}=\left\{\mathcal{F}_{\Theta^{H}(c,q)}:c>0,q>1/2\right\}$, i.e.,
   \begin{equation}\label{eq:minimax_ma}
    \sup_{\bbf \in \mathcal{F}}\mathbb{E}L_n(\hat{\bw}|\mathcal{M}_{G1},\bbf)=[1+o(1)]R_L(\mathcal{F})
   \end{equation}
   holds for every $\mathcal{F}\in \boldsymbol{\mathcal{F}}^{H}$.

\end{theorem}

This theorem demonstrates that the nested MMA estimator also has appealing exact minimax properties with the estimated $\sigma^2$. Indeed, $\hat{\bbf}_{\hat{\bw}|\mathcal{M}_{G1}}$ is adaptive in the exact minimax sense on the family of the ellipsoids. However, based on Theorem~5 of \cite{Donoho1990}, we deduce that $\hat{\bbf}_{\hat{\bw}|\mathcal{M}_{G1}}$ cannot be adaptive in the exact minimax sense on the family of the hyperrectangles due to $R_L[\mathcal{F}_{\Theta^{H}(c,q)}]/R_M[\mathcal{F}_{\Theta^{H}(c,q)}]\to \rho, 1<\rho<\infty$. But it is still seen that $\hat{\bbf}_{\hat{\bw}|\mathcal{M}_{G1}}$ adaptively achieves the minimax-rate optimality over all the estimators on the family of the hyperrectangles.

\begin{remark}
  The coefficient vectors within the ellipsoid $\Theta(\alpha, R)$ and the hyperrectangle $\Theta^{H}(c,q)$ inherently exhibit order constraints. For minimax theories under general $\ell_q$-balls without such order structures, see, e.g., \cite{Donoho1994}, \cite{Ye2010}, \cite{Raskutti2011}, and \cite{Wang2014}. It remains open whether MMA also attains the minimax adaptivity over these broader classes.
\end{remark}

\section{Simulation studies}\label{sec:simulation}

The data is simulated from the linear regression model (\ref{eq:model}), where $p_n=\lfloor 2n/3 \rfloor$, $x_{1i} = 1$, the remaining $x_{ji}$ are independently generated from $N(0,1)$, and the random error terms $\epsilon_i$ are i.i.d. from $N(0,\sigma^2)$ and are independent of ${x_{ji}}'$s. We consider two cases of the monotone regression coefficients, wherein the ideal AOP and the full AOP pursue the same goal.

\begin{itemize}
  \item \emph{Case 1} (Polynomially decaying coefficients). Here, $\beta_j=j^{-\alpha_1}$ and $\alpha_1$ is varied from $0.5$ to $1.5$.
  \item \emph{Case 2} (Exponentially decaying coefficients). Here, $\beta_j=\exp(-j^{\alpha_2})$ and $\alpha_2$ is varied from $0.25$ to $1.25$.
\end{itemize}

The signal-to-noise ratio, which is defined by $\sum_{j=2}^{p_n}\beta_j^2/\sigma^2$, is set to be one via the parameter $\sigma^2$. And the sample size $n$ increases from $30$ to $1000$. The candidate models used to implement MA are nested and estimated by least squares. To highlight the issue of the weight/candidate model restriction, we assume that $\sigma^2$ is known for all the methods.

Let $\bbf=(f_1,\ldots,f_n)^{\top}$ denote the mean vector of the true regression function. The accuracy of an estimation procedure is evaluated in terms of the squared $\ell_2$ loss $n^{-1}\|\bbf - \hat{\bbf} \|^2$, where $\hat{\bbf}=(\hat{f}_1,\ldots,\hat{f}_n)^{\top}$ is the estimated mean vector. We replicate the data generation process $R=1000$ times to approximate the risks of the competing methods.

The restricted-AOP MMA estimators considered are WR with $N=2$ (WR1), WR with $N=5$ (WR2), MR with $M_n = 2\vee\lfloor(m_n^*)^{1/2}\rfloor$ (MR1), and MR with $M_n = 2\vee\lfloor m_n^*/2\rfloor$ (MR2). Detailed definitions of these methods are given in Section~\ref{sec:review} and Table~\ref{tab:method}. In each replication, we normalize the squared $\ell_2$ loss of these four methods by dividing the $\ell_2$ loss of the MMA estimator based on $\mathcal{M}_A$ and $\mathcal{W}_{p_n}$ (representing a full-AOP MMA method in Case 1). The relative risks are given in Table~\ref{tab:simu2} based on 1000 simulation runs (the standard errors are in the parentheses).

\begin{table}[!b]
  \centering
  \caption{Comparisons of the restricted-AOP MMA methods. The squared $\ell_2$ loss of each method is divided by the $\ell_2$ loss of the MMA estimator based on $\mathcal{M}_A$ and $\mathcal{W}_{p_n}$ in each simulation.}
  \resizebox{\columnwidth}{!}{
    \begin{tabular}{cccccccc}
    \hline
     \multirow{2}[0]{*}{$n$}     &   \multirow{2}[0]{*}{method}    & \multicolumn{3}{c}{Case 1} & \multicolumn{3}{c}{Case 2} \\
          \cline{3-5} \cline{6-8}
     &  & \multicolumn{1}{c}{$\alpha_1=0.51$} & \multicolumn{1}{c}{$\alpha_1=1$} & \multicolumn{1}{c}{$\alpha_1=1.5$} & \multicolumn{1}{c}{$\alpha_2=0.25$} & \multicolumn{1}{c}{$\alpha_2=0.75$} & \multicolumn{1}{c}{$\alpha_2=1.25$} \\
    \hline\specialrule{0em}{1.5pt}{1.5pt}
    \multicolumn{1}{c}{30} & WR1   & 1.091 (0.016) & 1.135 (0.017) & 1.137 (0.020) & 1.103 (0.011) & 1.111 (0.017) & 1.131 (0.030) \\
          & WR2   & 1.020 (0.007) & 1.044 (0.008) & 1.033 (0.010) & 1.020 (0.005) & 1.042 (0.009) & 1.011 (0.012) \\
          & MR1   & 1.972 (0.071) & 1.624 (0.050) & 1.923 (0.080) & 1.747 (0.040) & 2.243 (0.096) & 0.954 (0.043) \\
          & MR2   & 1.441 (0.043) & 1.388 (0.035) & 1.254 (0.042) & 1.280 (0.021) & 1.167 (0.036) & 0.954 (0.043) \\[6pt]
    \multicolumn{1}{c}{100} & WR1   & 1.113 (0.011) & 1.126 (0.017) & 1.124 (0.022) & 1.126 (0.007) & 1.125 (0.021) & 1.093 (0.028) \\
          & WR2   & 1.025 (0.004) & 1.037 (0.007) & 1.031 (0.009) & 1.022 (0.003) & 1.028 (0.009) & 1.051 (0.013) \\
          & MR1   & 2.081 (0.041) & 1.926 (0.041) & 2.072 (0.072) & 2.179 (0.031) & 1.821 (0.058) & 1.397 (0.079) \\
          & MR2   & 1.420 (0.022) & 1.306 (0.024) & 1.491 (0.043) & 1.440 (0.015) & 1.018 (0.026) & 1.397 (0.079) \\[6pt]
    \multicolumn{1}{c}{300} & WR1   & 1.129 (0.006) & 1.116 (0.015) & 1.065 (0.019) & 1.133 (0.005) & 1.025 (0.021) & 1.025 (0.037) \\
          & WR2   & 1.031 (0.003) & 1.041 (0.006) & 1.047 (0.011) & 1.029 (0.002) & 1.036 (0.009) & 1.081 (0.020) \\
          & MR1   & 2.286 (0.027) & 2.601 (0.050) & 3.735 (0.107) & 2.586 (0.024) & 3.703 (0.116) & 3.415 (0.358) \\
          & MR2   & 1.496 (0.013) & 1.356 (0.020) & 1.392 (0.032) & 1.514 (0.010) & 1.647 (0.044) & 3.415 (0.358) \\[6pt]
    \multicolumn{1}{c}{1000} & WR1   & 1.123 (0.004) & 1.090 (0.009) & 1.052 (0.017) & 1.124 (0.003) & 0.957 (0.021) & 0.898 (0.030) \\
          & WR2   & 1.026 (0.002) & 1.055 (0.006) & 1.062 (0.014) & 1.034 (0.002) & 1.016 (0.013) & 1.011 (0.023) \\
          & MR1   & 2.541 (0.018) & 3.740 (0.056) & 4.945 (0.128) & 3.558 (0.022) & 10.469 (0.443) & 8.506 (0.757) \\
          & MR2   & 1.525 (0.008) & 1.432 (0.015) & 1.447 (0.030) & 1.524 (0.007) & 4.076 (0.166) & 8.506 (0.757) \\
          \hline
    \end{tabular}%
    }
  \label{tab:simu2}%
\end{table}%

From Table~\ref{tab:simu2}, the relative risks of the methods WR1 and WR2 are significantly larger than 1 in Case 1, which implies that using the discrete weight sets increases the risk of the full-AOP MMA by a sizable edge. This result is consistent with Proposition~\ref{lemma:delta_1}. In Case 2, however, when $\alpha_2=1.25$ and $n=1000$, the relative risks of WR1 and WR2 are 0.898 (0.030) and 1.011 (0.023), respectively, which shows that WR methods perform better than and comparably to the MMA based on $\mathcal{W}_{p_n}$. This phenomenon is not surprising. Although Proposition~\ref{lemma:delta_1} states that MA with the discrete weight restriction has an asymptotically equivalent oracle risk to that under the continuous weight set in Case 2, the latter actually pays a higher price to pursue the oracle MA risk when $n$ is finite, and the trade-off favors simplicity in this special case.

We find that the MR1 and MR2 methods mostly have much larger relative risks than the WR methods in both cases. Moreover, their relative risks become increasingly greater as the sample size increases from 30 to 1000. These findings support our theoretical understandings in Section~\ref{subsec:review_wan}.

Another interesting observation is about the result when $\alpha_2= 0.25$ in Case 2. Although the data is generated from a true regression model with exponentially decaying coefficients, this setting is more like a polynomially decaying case in the finite sample situation. Indeed, when $\alpha_1= 0.75$ and $n=1000$, we have $m_n^* \approx 75$. While in Case 2 with $\alpha_2= 0.25$ and $n=1000$, $m_n^*$ is around $77$, which does not exhibit the significant difference as in Case 1. Thus it is not surprising that the numerical performance of the competing methods in Case 2 ($\alpha_2= 0.25$) is similar to that in Case 1. More discussions related to this phenomenon can be found in \cite{Liu2011PI}, \cite{ZHANG201595}, and \cite{Peng2021}.

In the Appendix, we provide more simulation results to assess the full-AOP theory in Section~\ref{subsec:aop} and to compare the different candidate model sets given in Section~\ref{sec:reduced}. Overall, these results support our AOP theory on MMA and present evidence favoring the use of the candidate model sets with reduced sizes.

\section{Conclusion}\label{sec:conclusion}

MMA is one of the most studied MA methods in the literature, with its hallmark feature of AOP as a central focus. As shown in this work, however, the previous AOP theories, which rely on either discretizing the weight parameters or restricting the set of candidate models, actually do not establish the intended AOP. This raises doubt on validity of the MMA approach.

In this paper, we have addressed three key questions in this regard. (1) Can MMA achieve the performance of the optimal convex combination of all the nested models (i.e., the full AOP)? (2) Under what conditions can the optimal risk over all the convex combinations of all subset models (i.e., the ideal AOP) be realized? (3) Is MMA adaptive in an exact minimax sense for nonparametric classes?

The first and third questions are answered positively in this work, which rigorously justify the MMA methodology in those contexts. Therefore, without discretizing model weight or restricting the candidate model set harmfully, MMA achieves the full AOP when the optimal MA risk does not converge too fast, and consequently outperforms any MS method when for instance the true regression coefficients decay polynomially. The exact minimax optimalities proved in our work further support MMA as a powerful tool for regression estimation when the variables are properly ordered.

The answer to the second question unfortunately is generally negative. That is, if there is no prior information on the order of the coefficients, MMA (or any other method) cannot generally achieve the best performance of a convex combination of all the subset models even when the regressors are orthogonal to each other. A silver lining is that if the coefficients are ordered to some large extent, the ideal AOP may be reachable, as shown in Theorem~\ref{theo:suff}.

The above findings may have important implications for applications. In practice, it is perhaps rare that the data analyst knows the correct order of the true coefficients. From a practical viewpoint, one may pursue an importance learning of the regressors based on parametric or nonparametric approaches such as Random Forest and SOIL (see \cite{Ye2018Sparsity} for references). If one feels reasonably confident about the ranking of the regressors, then conducting MMA based on the corresponding nested models may work very well. In contrast, if the order of many relevant variables is highly uncertain, as is often the case in high-dimensional regression especially when the regressors are highly correlated, the ideal AOP is simply an unrealizable objective. Then it is better to turn to the less aggressive goal of minimax-rate optimality, which can be obtained based on MA methods such as ARM \citep{Yang2001, Wang2014}.

In closing, we provide several directions for future research. The focus of this paper has been on a linear regression setup with nonstochastic regressors. It is of great interest to extend the theoretical framework to the random design regression, such as the high-dimensional regression with sparse constraints \citep[see, e.g.,][]{Wang2014}, the linear mixed effects models \citep[see][]{jiang2007linear, Chen2003, Bondell2010, Fan2012}, and other practically important models. Another extension, motivated by an observation from Table~\ref{tab:method}, is to develop an MA method that can achieve the full AOP on the whole parameter region, if possible. Based on the works of \cite{ZHANG201595} and \cite{QIAN2022193}, we conjecture that a universally full AOP may be established by properly using cross-validation or hypothesis testing techniques. We leave these for future work.

\appendix

\section*{Appendix}\label{appendix}
\addcontentsline{toc}{section}{Appendix}
\addtocontents{toc}{\setcounter{tocdepth}{1}}
\renewcommand{\thesection}{\Alph{section}}
\numberwithin{equation}{section}

Section~\ref{sec:a:proof:s3} contains the proofs of all the propositions and examples in Section~\ref{sec:review}. The proofs of the main results under the nested and the non-nested setups are given in Section~\ref{sec:a:proof:s4} and Section~\ref{sec:a:proof_non_nested}, respectively. Section~\ref{sec:proof_minimax} provides the proof of the minimax adaptivity of Mallows model averaging (MMA). Section~\ref{sec:a:add_theory} and Section~\ref{sec:a:simu} present additional theoretical results and simulation results, respectively. And other related works are discussed in Section~\ref{sec:a:related}.

\section{Proofs of the results in Section~\ref{sec:review}}\label{sec:a:proof:s3}

\subsection{An equivalent sequence model}\label{sec:a:proof:s3:preliminaries}

Recall that $\bP_{j|\mathcal{M}_A}$ represents the projection matrix based on the first $j$ columns of $\bX$. Define $\bD_j\triangleq\bP_{j|\mathcal{M}_A} - \bP_{(j-1)|\mathcal{M}_A}$ for $j=1,\ldots,p_n$, where $\bP_{0|\mathcal{M}_A}=\boldsymbol{0}_{n\times n}$. Note that $\bD_j$ is a projection matrix, and $\bD_j,j=1,\ldots,p_n$ are mutually orthogonal, i.e., $\bD_j\bD_{j'}=\bD_{j'}\bD_j=\bD_j\delta_{jj'}$, where $\delta_{jj'}$ denotes the Kronecker delta \citep[see, e.g.,][]{Xu2022From}. Using eigendecomposition, we express $\bD_j$ as $\bphi_j\bphi_j^{\top}$, where $\bphi_j\in \mathbb{R}^{n}$ satisfies $\|\bphi_j\|=1$. The orthogonality of $\bD_j,j=1,\ldots,p_n$ implies that $\{\bphi_1,\ldots,\bphi_{p_n} \}$ constitutes an orthonormal basis for the column space of $\bX$. Consequently, we can rewrite the model (\ref{eq:model_matrix}) as an equivalent sequence model
\begin{equation}\label{eq:sequence}
  \hat{\theta}_{j} = \theta_j + e_j, \quad j=1,\ldots,p_n,
\end{equation}
where $\hat{\theta}_{j} =\bphi_j^{\top}\by/\sqrt{n}$, $\theta_j=\bphi_j^{\top}\bbf/\sqrt{n}$, and $e_j=\bphi_j^{\top}\beps/\sqrt{n}$. Suppose $\epsilon_1, \ldots, \epsilon_n$ are i.i.d. $\eta$-sub-Gaussian random variables. Then, $e_j,j=1,\ldots,p_n$ are $(\eta/\sqrt{n})$-sub-Gaussian variables, which satisfy $\mathbb{E}e_j=0$, $\mathbb{E}e_j^2=\sigma^2/n$, and $\mathbb{E}e_je_{j'}=0$ when $j \neq j'$.

Based on the sequence model (\ref{eq:sequence}), the least squares estimator $\hat{\bbf}_{m|\mathcal{M}_A}$ has the following equivalent form
\begin{equation}\label{eq:bbf_m_eq}
  \hat{\bbf}_{m|\mathcal{M}_A}=\bP_{m|\mathcal{M}_A}\by = \sum_{j=1}^{m}\bD_j\by=\sum_{j=1}^{m}\bphi_j\bphi_j^{\top}\by=\sqrt{n}\sum_{j=1}^{m}\bphi_j\hat{\theta}_{j}.
\end{equation}
The $\ell_2$ risk of $\hat{\bbf}_{m|\mathcal{M}_A}$ is
\begin{equation}\label{eq:risk_m}
\begin{split}
   R_n(m|\mathcal{M}_A,\bbf) & =\frac{1}{n}\mathbb{E}\left\|\hat{\bbf}_{m|\mathcal{M}_A}-\bbf \right\|^2=\frac{1}{n}\mathbb{E}\left\|\sum_{j=1}^{m}\bD_j\by-\sum_{j=1}^{p_n}\bD_j\bbf \right\|^2\\
    &=\mathbb{E} \left\|  \sum_{j=1}^{m}\bphi_j\hat{\theta}_{j} - \sum_{j=1}^{p_n}\bphi_j\theta_{j} \right\|^2 =\mathbb{E}\left\|\sum_{j=1}^{m}\bphi_je_j-\sum_{j=m+1}^{p_n}\bphi_j\theta_j \right\|^2\\
    &=\frac{m\sigma^2}{n}+\sum_{j=m+1}^{p_n}\theta_j^2,
\end{split}
\end{equation}
where the last equality arising from the orthogonality of $\{\bphi_1,\ldots,\bphi_{p_n} \}$ and $\mathbb{E}e_j^2=\sigma^2/n$.

Define the cumulative weight $\gamma_j\triangleq\sum_{m=j}^{M_n}w_m$ for $1\leq j \leq M_n$. The nested model averaging (MA) estimator based on $\mathcal{M}=\{ \{1,2,\ldots,k_m \}: 1\leq k_1<k_2<\cdots<k_{M_n}\leq p_n\}$ can be expressed as
\begin{equation}\label{eq:fw}
\begin{split}
   \hat{\bbf}_{\mathbf{w}|\mathcal{M}} &     =\sum_{m=1}^{M_n}w_m\hat{\bbf}_{m|\mathcal{M}}=\sum_{m=1}^{M_n}w_m\hat{\bbf}_{k_m|\mathcal{M}_A}\\
   &=\sum_{m=1}^{M_n}w_m\left(\sqrt{n}\sum_{j=1}^{k_m}\bphi_j\hat{\theta}_{j}\right)\\
   &=\sqrt{n}\sum_{j=1}^{M_n}\sum_{l=k_{j-1}+1}^{k_j}\gamma_j\bphi_l\hat{\theta}_l,
\end{split}
\end{equation}
where $k_0=0$, and the third equality follows from (\ref{eq:bbf_m_eq}). A similar calculation to (\ref{eq:risk_m}) yields the $\ell_2$ loss of $\hat{\bbf}_{\mathbf{w}|\mathcal{M}}$
\begin{equation}\label{eq:lossw}
\begin{split}
    L_n(\bw|\mathcal{M},\bbf)&=\frac{1}{n}\left\|\hat{\bbf}_{\mathbf{w}|\mathcal{M}}-\bbf \right\|^2 \\
     & =\sum_{j=1}^{M_n}\sum_{l=k_{j-1}+1}^{k_j}\left( \gamma_j\hat{\theta}_l-\theta_l  \right)^2+\sum_{j=k_{M_n}+1}^{p_n}\theta_j^2
\end{split}
\end{equation}
and its corresponding MA risk
\begin{equation}\label{eq:riskw}
\begin{split}
    R_n(\bw|\mathcal{M},\bbf)&=\mathbb{E}L_n(\bw|\mathcal{M},\bbf) \\
     & =\sum_{j=1}^{M_n}\sum_{l=k_{j-1}+1}^{k_j}\left[ \left(1-\gamma_j \right)^2\theta_l^2+\frac{\sigma^2}{n}\gamma_j^2 \right]+\sum_{j=k_{M_n}+1}^{p_n}\theta_j^2.
\end{split}
\end{equation}
Furthermore, the MMA criterion (\ref{eq:criterion}) can be equivalently rewritten as
\begin{equation}\label{eq:cr}
\begin{split}
   &C_{n}(\bw|\mathcal{M},\by) =\frac{1}{n}\left\|\by-\hat{\bbf}_{\bw | \mathcal{M}} \right\|^2+\frac{2\hat{\sigma}^2}{n}\bk^{\top}\bw\\
     & = \frac{1}{n}\left\|\by-\sqrt{n}\sum_{j=1}^{M_n}\sum_{l=k_{j-1}+1}^{k_j}\gamma_j\bphi_l\hat{\theta}_l \right\|^2+\sum_{j=1}^{M_n}\sum_{l=k_{j-1}+1}^{k_j}2\gamma_j\frac{\hat{\sigma}^2}{n}\\
     & =\frac{1}{n}\left\|\by\right\|^2 - \frac{2}{\sqrt{n}}\sum_{j=1}^{M_n}\sum_{l=k_{j-1}+1}^{k_j}\gamma_j\by^{\top}\bphi_l\hat{\theta}_l + \sum_{j=1}^{M_n}\sum_{l=k_{j-1}+1}^{k_j}\gamma_j^2\hat{\theta}_l^2 +\sum_{j=1}^{M_n}\sum_{l=k_{j-1}+1}^{k_j}2\gamma_j\frac{\hat{\sigma}^2}{n}\\
     & =\sum_{j=1}^{M_n}\sum_{l=k_{j-1}+1}^{k_j}\left[ \gamma_j^2\hat{\theta}_l^2+2\gamma_j\left(\frac{\hat{\sigma}^2}{n}-\hat{\theta}_l^2\right) \right]+\frac{1}{n}\sum_{i=1}^{n}y_i^2,
\end{split}
\end{equation}
where the last equality follows from $\by^{\top}\bphi_l/\sqrt{n}=\hat{\theta}_l$ for $l=1,\ldots,p_n$.


\subsection{Technical lemma}

Lemma~\ref{lemma:peng} compares the optimal risks of model selection (MS) and MA based on the successive candidate model set $\mathcal{M}_{S}=\{\{1,\ldots,m \}:1\leq m \leq M_n\}$. Let us recall some notations. We define $m_n^*=\arg\min_{m \in \{1,\ldots,p_n \}}R_n(m|\mathcal{M}_A,\bbf)$ as the size of the optimal nested model, $m^*|\mathcal{M}_S=\arg\min_{m \in \{1,\ldots,M_n \}}R_n(m|\mathcal{M}_{S},\bbf)$ as the size of the optimal candidate model in $\mathcal{M}_S$, and $\bw^*|\mathcal{M}_S=\arg\min_{\mathbf{w} \in \mathcal{W}_{M_n}}R_n(\bw | \mathcal{M}_S,\bbf)$ as the optimal weight vector based on the candidate model set $\mathcal{M}_S$.

\begin{lemma}\label{lemma:peng}
Suppose that Assumptions~\ref{asmp:square_summable}--\ref{asmp:regressor_order} hold. For any set $\mathcal{M}_{S}$ with a size $1 \leq M_n \leq p_n$, the following relationship holds:
\begin{equation*}\label{eq:lemma:peng:1}
R_n(m^*|\mathcal{M}_{S},\bbf) \asymp R_n(\bw^*|\mathcal{M}_S,\bbf).
\end{equation*}

For a large set $\mathcal{M}_{S}$ with $M_n\gtrsim m_n^*$, we have
\begin{equation*}\label{eq:lemma:peng:2}
R_n(m^*|\mathcal{M}_{S},\bbf) \asymp R_n(\bw^*|\mathcal{M}_{S},\bbf) \asymp R_n(m^*|\mathcal{M}_A,\bbf) \asymp \frac{m_n^*}{n}.
\end{equation*}
Furthermore, under Condition~\ref{condition1}, we have
\begin{equation*}\label{eq:lemma:peng:3}
  R_n(m^*|\mathcal{M}_{S},\bbf) - R_n(\bw^*|\mathcal{M}_{S},\bbf) \asymp R_n(m^*|\mathcal{M}_{S},\bbf).
\end{equation*}
Under Condition~\ref{condition2}, we obtain
\begin{equation*}\label{eq:lemma:peng:4}
  R_n(m^*|\mathcal{M}_{S},\bbf) - R_n(\bw^*|\mathcal{M}_{S},\bbf) = o\left[R_n(m^*|\mathcal{M}_{S},\bbf)\right].
\end{equation*}

For a small set $\mathcal{M}_{S}$ with $M_n=o(m_n^*)$, we have \begin{equation*}\label{eq:lemma:peng:5}
  R_n(m^*|\mathcal{M}_{S},\bbf) - R_n(\bw^*|\mathcal{M}_{S},\bbf) = o\left[R_n(m^*|\mathcal{M}_{S},\bbf)\right].
\end{equation*}

\end{lemma}
\begin{proof}
  Note that Assumption~\ref{asmp:square_summable} is equivalent to
  \begin{equation}\label{eq:sum_theta}
    \frac{1}{n}\left\|\bbf \right\|^2 = \frac{1}{n}\left\|\sum_{j=1}^{p_n}\bD_j\bbf \right\|^2= \frac{1}{n}\left\|\sqrt{n}\sum_{j=1}^{p_n}\bphi_j\theta_j \right\|^2 = \sum_{j=1}^{p_n}\theta_j^2<\infty.
  \end{equation}
  This coincides with Assumption 1 in \cite{Peng2021}. Thus, Theorems~1--2 of \cite{Peng2021} imply the results when $M_n \geq m_n^*$. For cases where $M_n < m_n^*$, see \cite{Xu2022From} for the complete proofs.
\end{proof}

\subsection{Proof of Proposition~\ref{lemma:delta_1}}

We begin by recalling some established results about the optimal MS and MA risks in \cite{Peng2021}. From (\ref{eq:risk_m}), we see that when Assumption~\ref{asmp:regressor_order} holds, the optimal nested model $m_n^{\ast}$ satisfies
\begin{equation}\label{eq:mnstar}
    \theta_{m_n^{\ast}}^2>\frac{\sigma^2}{n}\geq \theta_{m_n^{\ast}+1}^2.
\end{equation}
Thus, the optimal MS risk can be expressed as
\begin{equation}\label{eq:optimal_ms_risk}
  R_n(m^{\ast}|\mathcal{M}_A,\bbf)=\frac{m_n^{\ast}\sigma^2}{n}+\sum_{j=m_n^{\ast}+1}^{p_n}\theta_j^2.
\end{equation}
Using (\ref{eq:riskw}), we obtain the MA risk based on $\mathcal{M}_S$ as
\begin{equation}\label{eq:riskw_large}
  \begin{split}
     R_n\left(\bw|\mathcal{M}_S,\bbf\right)&=\sum_{j=1}^{M_n}\left( \frac{\sigma^2}{n} + \theta_j^2 \right) \left( \gamma_j - \frac{\theta_j^2}{\theta_j^2+\frac{\sigma^2}{n}} \right)^2+\sum_{j=1}^{M_n}\frac{\theta_j^2\sigma^2}{n\theta_j^2+\sigma^2}+\sum_{j=M_n+1}^{p_n}\theta_j^2.
  \end{split}
\end{equation}
The infeasible optimal weight vector $\bw^*|\mathcal{M}_S=(w_1^*,\ldots,w_{M_n}^*)^{\top}$ can be obtained by setting
\begin{equation}\label{eq:gamma_star}
  \gamma_1^*=1,\,\gamma_j^*=\frac{\theta_j^2}{\theta_j^2+\frac{\sigma^2}{n}},\, j=2,\ldots,M_n,
\end{equation}
where $\gamma_j^*=\sum_{m=j}^{M_n}w_m^*$. Thus, the optimal MA risk based on $\mathcal{M}_S$ is
$$
R_n(\bw^*|\mathcal{M}_S,\bbf)=\frac{\sigma^2}{n}+\sum_{j=2}^{M_n}\frac{\theta_j^2\sigma^2}{n\theta_j^2+\sigma^2}+\sum_{j=M_n+1}^{p_n}\theta_j^2.
$$

We first prove the results when Condition~\ref{condition1} and $M_n\gtrsim m_n^*$ hold. Let $G:\mathbb{N}\rightarrow \mathbb{N}$ by
$$
G(x)=\arg\min_{t \in \mathbb{N}}\left( \lfloor \kappa t \rfloor \geq x  \right),
$$
where $\kappa$ is the constant given in Condition~\ref{condition1} and $\mathbb{N}$ is the set of natural numbers. Define a sequence of functions $G_d(x)$ indexed by integer $d$
\begin{equation}\label{eq:Gfunction}
G_d(x) =\left\{\begin{array}{ll}
x &\quad d = 0, \\
\left(G \circ G_{d-1} \right)(x) &\quad d \geq 1,
\end{array}\right.
\end{equation}
where the notation $(f\circ g)(x)$ means the composition of functions $f(g(x))$.

Given a fixed $N$, define $d_1^*=\arg\min_{d \in \mathbb{N}}\nu^{2d} \leq 1/(N-1)$ and $i_n^*=M_n \wedge G_{d_1^*+1}(m_n^*)$, where $0<\nu<1$ is the constant defined in Condition~\ref{condition1}. Since $M_n\gtrsim m_n^*$ and $d_1^*$ is a fixed integer, we see $i_n^* \asymp m_n^*$.
We have
\begin{equation}\label{eq:a135}
  \begin{split}
     \frac{\theta_{m_n^*}^2}{\theta_{i_n^*}^2} & \leq\frac{\theta_{m_n^*}^2}{\theta_{G_1(m_n^*)}^2}\times\frac{\theta_{G_1(m_n^*)}^2}{\theta_{G_2(m_n^*)}^2}\times\cdots\times\frac{\theta_{G_{d_1^*}(m_n^*)}^2}{\theta_{G_{d_1^*+1}(m_n^*)}^2}\times\frac{\theta_{G_{d_1^*+1}(m_n^*)}^2}{\theta_{i_n^*}^2} \\
       &\leq  \nu^{2d_1^*+2}\leq\frac{\nu^2}{N-1},
  \end{split}
\end{equation}
where the second inequality follows from Condition~\ref{condition1} and $\theta_{i_n^*}^2\geq \theta_{G_{d_1^*+1}(m_n^*)}^2$, and the last inequality is due to the definition of $d_1^*$. Therefore
\begin{equation}\label{eq:dis_weight_1}
\gamma_{i_n^*}^{*}- \frac{N-1}{N}\geq \frac{\theta_{i_n^*}^2}{\theta_{i_n^*}^2+\theta_{m_n^*}^2}- \frac{N-1}{N} \geq \frac{N-1}{N-1+\nu^2}- \frac{N-1}{N}\triangleq C_1>0,
\end{equation}
where the first inequality is due to (\ref{eq:mnstar}) and (\ref{eq:gamma_star}), and the second inequality is due to (\ref{eq:a135}).

Define another model index $j_n^*=G_{1}(i_n^*)$. Note that
\begin{equation*}
  \begin{split}
     \frac{\theta_{m_n^*+1}^2}{\theta_{j_n^*}^2} & =\frac{\theta_{m_n^*+1}^2}{\theta_{G_1(m_n^*+1)}^2}\times\frac{\theta_{G_1(m_n^*+1)}^2}{\theta_{G_2(m_n^*+1)}^2}\times\cdots\times\frac{\theta_{G_{d_1^*+1}(m_n^*+1)}^2}{\theta_{i_n^*}^2}\times\frac{\theta_{i_n^*}^2}{\theta_{G_{1}(i_n^*)}^2} \\
       & \geq \delta^{2d_1^*+4}\frac{\theta_{G_{d_1^*+1}(m_n^*+1)}^2}{\theta_{i_n^*}^2},
  \end{split}
\end{equation*}
where $0<\delta<1$ is the constant defined in Condition~\ref{condition1}. Since $i_n^*=M_n \wedge G_{d_1^*+1}(m_n^*)$ and $M_n\gtrsim m_n^*$, there must exist a constant $0<c\leq 1$ such that
\begin{equation*}
  \frac{\theta_{G_{d_1^*+1}(m_n^*+1)}^2}{\theta_{i_n^*}^2}>c
\end{equation*}
under Condition~\ref{condition1}. We thus have
\begin{equation}\label{eq:dis_weight_2}
  1-\gamma_{j_n^*}^{*}\geq 1- \frac{\theta_{j_n^*}^2}{\theta_{j_n^*}^2+\theta_{m_n^*+1}^2}\geq 1- \frac{1}{1+ c\delta^{2d_1^*+4}}\triangleq C_2>0.
\end{equation}

Let $\bw_N^*|\mathcal{M}_S=\arg\min_{\bw \in \mathcal{W}_{|\mathcal{M}_S|}(N)}R_n(\bw | \mathcal{M}_S,\bbf)$ denote the optimal discrete weight vector in $\mathcal{W}_{|\mathcal{M}_S|}(N)$. Note that restricting $\bw_N|\mathcal{M}_{S}=(w_{N,1},\ldots,w_{N,M_n})^{\top} \in \mathcal{W}_{|\mathcal{M}_S|}(N)$ is equivalent to restricting $\bgamma_N|\mathcal{M}_S=(\gamma_{N,1},\ldots,\gamma_{N,M_n})^{\top} \in \Gamma_{|\mathcal{M}_S|}(N)=\{\gamma_{N,j}=t_j/N:N=t_1\geq t_2\geq\cdots\geq t_{M_n}\geq0, t_j \in \mathbb{N}\cup \{0\} \}$, where $\gamma_{N,j}\triangleq\sum_{m=j}^{M_n}w_{N,m}$. Based on (\ref{eq:dis_weight_1}) and (\ref{eq:dis_weight_2}), when $j_n^*<j\leq i_n^*$, we see that the optimal cumulative weights satisfy
$$
\frac{N-1}{N}+C_1\leq \gamma_{j}^{*}\leq 1-C_2.
$$
However, the optimal discrete cumulative weight $\gamma_{N,j}^*=\sum_{m=j}^{M_n}w_{N,m}^*$ is either 1 or $(N-1)/N$ when $j_n^*<j\leq i_n^*$. Combining (\ref{eq:riskw_large}) with (\ref{eq:dis_weight_1}) and (\ref{eq:dis_weight_2}), we see at once that
\begin{equation*}\label{eq:risk_difference_bound}
\begin{split}
&R_n\left(\bw_N^*|\mathcal{M}_{S},\bbf\right)-R_n\left(\bw^*|\mathcal{M}_{S},\bbf\right)\\
&=\sum_{j=1}^{M_n}\left( \frac{\sigma^2}{n} + \theta_j^2 \right) \left( \gamma_{N,j}^* - \frac{\theta_j^2}{\theta_j^2+\frac{\sigma^2}{n}} \right)^2-\sum_{j=1}^{M_n}\left( \frac{\sigma^2}{n} + \theta_j^2 \right) \left( \gamma_{j}^* - \frac{\theta_j^2}{\theta_j^2+\frac{\sigma^2}{n}} \right)^2\\
& \geq \sum_{j=2}^{M_n}\left( \frac{\sigma^2}{n} + \theta_j^2 \right) \left( \gamma_{N,j}^*- \gamma_{j}^* \right)^2\geq \sum_{j=j_n^*+1}^{i_n^*}\frac{\sigma^2}{n}(C_1^2\wedge C_2^2) \\
&= \frac{(C_1^2\wedge C_2^2)(i_n^*-j_n^*)\sigma^2}{n}\asymp \frac{m_n^*}{n}
   \asymp R_n(\bw^*|\mathcal{M}_{S},\bbf),
\end{split}
\end{equation*}
where the constants $C_1$ and $C_2$ are defined in (\ref{eq:dis_weight_1}) and (\ref{eq:dis_weight_2}) respectively, and the last approximation follows from Lemma~\ref{lemma:peng}.

Due to
\begin{equation}\label{eq:relation3}
  R_n\left(\bw^*|\mathcal{M}_S,\bbf\right)\leq R_n\left(\bw_N^*|\mathcal{M}_S,\bbf\right)\leq R_n\left(m^*|\mathcal{M}_S,\bbf\right),
\end{equation}
the proof of the results under Condition~\ref{condition2} or $M_n=o(m_n^*)$ is a direct application of Lemma~\ref{lemma:peng}. This completes the proof.

\subsection{Proof of Examples~1--2 in Section~\ref{subsec:review_hansen}}

When $\theta_j=j^{-\alpha_1}, \alpha_1>1/2$, according to Theorem~1 and Example~1 of \cite{Peng2021}, we have $m_n^* \sim (n/\sigma^2)^{1/(2\alpha_1)}$ and $$
R_n(\bw^*|\mathcal{M}_A,\bbf)\asymp R_n(m^*|\mathcal{M}_A,\bbf) \asymp m_n^*/n \asymp n^{-1+1/(2\alpha_1)}.$$
In addition, we have
$$
R_n(\bw^*|\mathcal{M}_A,\bbf) \leq R_n\left(\bw_N^*|\mathcal{M}_A,\bbf\right)\leq R_n(m^*|\mathcal{M}_A,\bbf)\leq 2R_n(\bw^*|\mathcal{M}_A,\bbf),
$$
where the last inequality follows from (A.6) of \cite{Peng2021}. Thus, combining with Proposition~\ref{lemma:delta_1}, there exists a constant $1<C<2$ such that  $R_n\left(\bw_N^*|\mathcal{M}_A,\bbf\right)>CR_n(\bw^*|\mathcal{M}_A,\bbf)$ when $n$ is sufficiently large.

For the case when $\theta_j=\exp(-cj^{\alpha_2})$, $\alpha_2>0$, the results follow directly from Example~2 of \cite{Peng2021} and Proposition~\ref{lemma:delta_1} in this paper.

\subsection{Proof of Proposition~\ref{lemma:delta_2}}

Under Condition~\ref{condition1}, in a manner of proof by contradiction, we first check that a necessary condition for (\ref{eq:condition_wan}) is $M_n=o(m_n^*)$. Suppose $M_n\geq m_n^*$, it is already seen from \cite{Peng2021} that $$R_n\left(\bw^*|\mathcal{M}_S,\bbf \right) \asymp R_n(m^*|\mathcal{M}_A,\bbf) \asymp \frac{m_n^*}{n}$$
under Condition~\ref{condition1}. We thus obtain
\begin{equation}\label{eq:condition_wan2}
\begin{split}
   & \frac{|\mathcal{M}_S|\sum_{m=1}^{|\mathcal{M}_S|}R_n\left(m|\mathcal{M}_S,\bbf\right)}{nR_n^2\left(\bw^*|\mathcal{M}_S,\bbf  \right)} \gtrsim \frac{M_n^2m_n^*/n}{(m_n^*)^2/n}\\
   &=\frac{M_n^2}{m_n^*}\geq M_n \geq m_n^* \to \infty,
\end{split}
\end{equation}
which contradicts the assumption (\ref{eq:condition_wan}). Suppose $M_n < m_n^*$ but $M_n \asymp m_n^*$, there must exist a constant $C>1$ and a positive integer $K$ such that for any $n>K$, we have $m_n^*<CM_n$. In this case, the main task is to show the risk of the optimal single model in $\mathcal{M}_S$ and the risk of the optimal averaged model based on $\mathcal{M}_S$ both have the order $m_n^*/n$. Note first that the optimal single model in $\mathcal{M}_S$ needs to include the first $M_n$ regressors, which has the risk
$$
R_n\left( m^*|\mathcal{M}_S,\bbf \right) = \frac{M_n\sigma^2}{n} + \sum_{j=M_n+1}^{m_n^*}\theta_j^2+\sum_{j=m_n^*+1}^{p_n}\theta_j^2.
$$ As there must exist an index $d_2^*$ such that $G_{d_2^*}(m_n^*+1)\leq m_n^*/C<M_n$, it follows that the second term in $R_n\left( m^*|\mathcal{M}_S,\bbf \right)$ is bounded by
\begin{equation}
\begin{split}
   \sum_{j=M_n+1}^{m_n^*}\theta_j^2 & \leq \left(m_n^*-M_n\right)\theta_{G_{d_2^*}(m_n^*+1)}^2\leq \frac{\left(m_n^*-M_n\right)\theta_{m_n^*+1}^2}{\delta^{2d_2^*}} \\
     & \leq \frac{\left(m_n^*-M_n\right)\sigma^2}{n\delta^{2d_2^*}}\lesssim \frac{m_n^*}{n},
\end{split}
\end{equation}
where the second inequality follows from Condition~\ref{condition1}, and the third inequality follows from (\ref{eq:mnstar}). Since the order of the last term in $R_n\left( m^*|\mathcal{M}_S,\bbf \right)$ is also no bigger than $m_n^*/n$ \citep{Peng2021}, we thus get $R_n\left( m^*|\mathcal{M}_S,\bbf \right)\asymp m_n^*/n$. Furthermore, it is easy to check that
\begin{equation*}
\begin{split}
    & R_n\left( m^*|\mathcal{M}_S,\bbf \right) \geq R_n(\bw^*|\mathcal{M}_S,\bbf) \\
     & \geq R_n(\bw^*|\mathcal{M}_L,\bbf) \asymp R_n(m^*|\mathcal{M}_A,\bbf) \asymp \frac{m_n^*}{n},
\end{split}
\end{equation*}
where $\mathcal{M}_L=\{\{1,\ldots,m \}:1\leq m \leq M_n\}$ with $M_n \geq m_n^*$, and the last two approximations are due to \cite{Peng2021}. It follows immediately that $R_n(\bw^*|\mathcal{M}_S,\bbf) \asymp m_n^*/n$. In the same manner of (\ref{eq:condition_wan2}), we also obtain a contradiction of assumption (\ref{eq:condition_wan}) when $M_n < m_n^*$ and $M_n \asymp m_n^*$. Thus, under Condition~\ref{condition1}, a necessary condition for (\ref{eq:condition_wan}) is $M_n=o(m_n^*)$.

Define $d_3^*=\arg\max_{d\in \mathbb{N}}\{G_{d}(m_n^*)\geq M_n\}$. Since $M_n=o(m_n^*)$, we have $d_3^*\to \infty$ as $n\to \infty$. Then the MA risk is lower bounded by
\begin{equation}\label{eq:dndkna}
\begin{split}
&R_n(\bw^*|\mathcal{M}_S,\bbf) \geq \sum_{j=M_n+1}^{m_n^*}\theta_j^2\\
&= \sum_{j=G_1(m_n^*)+1}^{m_n^*}\theta_j^2 + \sum_{j=G_2(m_n^*)+1}^{G_1(m_n^*)}\theta_j^2 + \cdots + \sum_{j=G_{d_3^*}(m_n^*)+1}^{G_{d_3^*-1}(m_n^*)}\theta_j^2 \\
   &\geq \theta_{m_n^*}^2[m_n^*-G_1(m_n^*)]+ \theta_{G_1(m_n^*)}^2[G_1(m_n^*)-G_2(m_n^*)]+\cdots+\theta_{G_{d_3^*-1}(m_n^*)}^2[G_{d_3^*-1}(m_n^*)-G_{d_3^*}(m_n^*)] \\
   &\geq \frac{\sigma^2}{n}\left(m_n^*-\frac{m_n^*}{\kappa}\right)+\frac{\sigma^2}{n\nu^2}\left(\frac{m_n^*}{\kappa}-\frac{m_n^*}{\kappa^2}\right)+\cdots+\frac{\sigma^2}{n\nu^{2(d_3^*-1)}}\left(\frac{m_n^*}{\kappa^{d_3^*-1}}-\frac{m_n^*}{\kappa^{d_3^*}}\right)\\
     & \geq \frac{m_n^*\sigma^2}{n}\left(1-\frac{1}{\kappa}\right)\sum_{l=0}^{d_3^*-1}\frac{1}{(\kappa\nu^{2})^l},
\end{split}
\end{equation}
where the first inequality follows from (\ref{eq:riskw}), and the third inequality is due to (\ref{eq:mnstar}) and Condition~\ref{condition1}. Since $d_3^*\to \infty$ and $\kappa\nu^{2}<1$, we thus get $$\sum_{l=0}^{d_3^*-1}\frac{1}{(\kappa\nu^{2})^l}\to \infty.$$
Due to $R_n(m^*|\mathcal{M}_A,\bbf)\asymp m_n^*/n$, from (\ref{eq:dndkna}) we conclude $R_n(m^*|\mathcal{M}_A,\bbf)=o[R_n(\bw^*|\mathcal{M}_S,\bbf)]$.

When Condition~\ref{condition2} holds, using the proof by contradiction again, we see that a necessary condition for (\ref{eq:condition_wan}) is $M_n \leq \lfloor C m_n^* \rfloor$ with a constant $0<C<1$. Note that $\lfloor C m_n^* \rfloor \leq \lfloor (C+1) m_n^*/2 \rfloor\leq m_n^*$. Then the MA risk is lower bounded by
\begin{equation*}
\begin{split}
    & R_n(\bw^*|\mathcal{M}_S,\bbf)\geq \sum_{j=M_n+1}^{\lfloor (C+1) m_n^*/2 \rfloor}\theta_j^2 \\
     & \geq ( \lfloor (C+1) m_n^*/2 \rfloor - M_n)\theta_{\lfloor (C+1) m_n^*/2 \rfloor}^2.
\end{split}
\end{equation*}
Under Condition~\ref{condition2}, we have $\theta_{\lfloor (C+1) m_n^*/2 \rfloor}^2/\theta_{m_n^*}^2\to \infty$ and $\theta_{m_n^*}^2 \asymp 1/n$. Thus we get $R_n(m^*|\mathcal{M}_A,\bbf)=o[R_n(\bw^*|\mathcal{M}_S,\bbf)]$, which proves the proposition.

\subsection{Proof of Examples~1--2 in Section~\ref{subsec:review_wan}}

Based on the risk of MS (\ref{eq:risk_m}), we have
\begin{equation}\label{eq:add1}
\begin{split}
\sum_{m=1}^{|\mathcal{M}_S|}R_n\left(m|\mathcal{M}_S,\bbf\right)&= \sum_{m=1}^{M_n}R_n\left(m|\mathcal{M}_S,\bbf \right)\\
    &= \sum_{j=1}^{M_n}\frac{j}{n}\sigma^2+\sum_{j=2}^{p_n}\theta_j^2+\cdots+\sum_{j=M_n+1}^{p_n}\theta_j^2 \\
     & = \sum_{j=1}^{M_n}\frac{j}{n}\sigma^2 + \sum_{j=2}^{M_n}(j-1)\theta_j^2+M_n\sum_{j=M_n+1}^{p_n}\theta_j^2.
\end{split}
\end{equation}
When $\theta_j=j^{-\alpha_1}, \alpha_1>1/2$, approximating the sums in (\ref{eq:add1}) by integrals, we obtain that the numerator of (\ref{eq:condition_wan}) has the order
\begin{equation*}
M_n\sum_{m=1}^{M_n}R_n\left(m|\mathcal{M}_S,\bbf \right) \asymp\left\{\begin{array}{ll}
M_n^{-2\alpha_1+3} &\quad 1/2<\alpha_1<1, \\
M_n\log{M_n} &\quad \alpha_1=1,\\
M_n &\quad \alpha_1>1.
\end{array}\right.
\end{equation*}

We now turn to evaluate the order of the denominator of (\ref{eq:condition_wan}). Define $g(x)=\int_{0}^{\frac{1}{1+x^{2\alpha_1}}}t^{1-\frac{1}{2\alpha_1}}(1-t)^{\frac{1}{2\alpha_1}}dt$ and $g'(x)=-\frac{2\alpha_1}{1+x^{2\alpha_1}}$. Based on the proof of Example~1 in \cite{Peng2021}, we have
\begin{equation}\label{eq:add2}
\begin{split}
R_n(\bw^*|\mathcal{M}_S, \bbf)& \asymp n^{-1+1/(2\alpha_1)}\left[g(0) - g\left(\frac{M_n}{m_n^*}\right) \right]+M_n^{-2\alpha_1+1} \\
     &\asymp n^{-1+1/(2\alpha_1)} \left[-g'(0) \left(\frac{M_n}{m_n^*}\right) \right]+M_n^{-2\alpha_1+1}\\
     &\asymp \frac{M_n}{n} + M_n^{-2\alpha_1+1} \asymp M_n^{-2\alpha_1+1},
\end{split}
\end{equation}
where the second approximation follows from Taylor's expansion, the third approximation follows from $m_n^*\asymp n^{1/(2\alpha_1)}$, and the last approximation follows from the fact $M_n=o(m_n^*)$ and $m_n^*\asymp n^{1/(2\alpha_1)}$. Combining (\ref{eq:add1}) with (\ref{eq:add2}) gives (\ref{eq:M_n_order1}).

When $\theta_j=\exp(-j^{\alpha_2})$, $\alpha_2>0$, in the same manner, we can see that the numerator of (\ref{eq:condition_wan}) has the order $M_n$. Define $\mbox{Ga}(x;a)=\int_{t=x}^{\infty}t^{a-1}\exp(-t)dt$ for $x>0$. Based on the proof of Example~2 in \cite{Peng2021}, we have
\begin{equation*}
  \begin{split}
      R_n(\bw^*|\mathcal{M}_S,\bbf)& \asymp R_n(m^*|\mathcal{M}_S,\bbf)\\
      &\asymp \frac{M_n}{n}+\mbox{Ga}\left(2M_n^{\alpha_2};\frac{1}{\alpha_2}\right) \\
       & \asymp \frac{M_n}{n} +(2M_n^{\alpha_2})^{\frac{1}{\alpha_2}-1}\exp(-2M_n^{\alpha_2}),
  \end{split}
\end{equation*}
where the first approximation follows from Lemma~\ref{lemma:peng}, and the third approximation is based on the asymptotic expansion of the incomplete gamma-function. Thus (\ref{eq:condition_wan}) is reduced to  $M_n<\left(1/2 \right)^{1/\alpha_2}m_n^*$, where $m_n^*=[(1/2)\log(n/\sigma^2)]^{1/\alpha_2}$. This completes the proof.

\section{Proofs of the results under the nested setup}\label{sec:a:proof:s4}

\subsection{Notations}\label{sec:a:proof:s4:notation}

In this section, we employ the symbols defined in Section~\ref{sec:a:proof:s3:preliminaries}. In addition, for any $n\times n$ real matrix $\bA$, we denote the operator norm and the Frobenius norm of $\bA$ as $\|\bA \|_2 = \max_{\|\bx \| \leq 1}\| \bA \bx\|$ and $\|\bA \|_{\mathrm{F}} = (\sum_{i,j}a_{i,j}^2)^{1/2}$, respectively. We define $\frac{0}{0} = 0$ in the proof for simplicity of notation.

\subsection{Technical lemma}\label{sec:a:proof:s4:preliminaries}

\begin{lemma}\label{lemma}
Let $\{\xi(t), t \in \mathcal{T}\}$ be a stochastic process with $\mathbb{E}\xi(t) = 0$ and finite
variance $\mathbb{E}[\xi(t)]^2 = \sigma^2(t)$ for all $t \in \mathcal{T}$, where $\mathcal{T}$ is a finite index set. Suppose that there exist $\lambda>0$ and $\varphi(\lambda)<\infty$ such that
\begin{equation}\label{eq:lemma_con}
  \max_{t \in \mathcal{T}} \mathbb{E}\exp(\lambda| \xi(t) |)\leq \varphi(\lambda).
\end{equation}
Then for all $r\geq 1$, there exists a constant $C$ depending on $\lambda$ and $r$ such that
\begin{equation*}
  \left(\mathbb{E} \max_{t \in \mathcal{T}}| \xi(t) |^r\right)^{\frac{1}{r}} \leq C (\log |\mathcal{T}|+1).
\end{equation*}
\end{lemma}

\begin{proof}
The proof of this lemma is motivated by Lemma~2.1 in \cite{Golubev2008}. Notice that for $r \geq 1$, the function $F(x)=\log^r[x+\exp(r-1)]$ is concave on $(0,\infty)$ since
\begin{equation*}
  F''(x) = \frac{r\log^{r-2}[x+\exp(r-1)]}{[x+\exp(r-1)]^2}\left\{r-1-\log\left[ x+\exp(r-1)\right]  \right\}\leq 0.
\end{equation*}
Using Jensen's inequality, we have
  \begin{equation*}
  \begin{split}
\left[\mathbb{E} \max_{t \in \mathcal{T}}| \xi(t) |^r\right]^{\frac{1}{r}}
&=\frac{1}{\lambda}\left\{ \mathbb{E} \left[\max_{t \in \mathcal{T}}|\lambda \xi(t) |\right]^r \right\}^{\frac{1}{r}}= \frac{1}{\lambda}\left\{\mathbb{E}\log^r \left[\exp\left( \max_{t \in \mathcal{T}}|\lambda \xi(t) | \right)\right]\right\}^{\frac{1}{r}}  \\
& \leq \frac{1}{\lambda}\left\{\mathbb{E}\log^r \left[\exp\left( \max_{t \in \mathcal{T}}|\lambda \xi(t) | \right)+\exp(r-1)\right]\right\}^{\frac{1}{r}} \\
&\leq \frac{1}{\lambda}\log \left[\mathbb{E}\exp\left(\max_{t \in \mathcal{T}}\lambda| \xi(t) |\right)+ \exp(r-1) \right]\\
&\leq \frac{1}{\lambda}\log \left[\sum_{t\in\mathcal{T}}\mathbb{E}\exp(\lambda| \xi(t) |)+ \exp(r-1) \right]\\
&\leq \frac{\log \left[ \varphi(\lambda)|\mathcal{T}| + \exp(r-1) \right]}{\lambda} \leq C(\log |\mathcal{T}|+1),
  \end{split}
  \end{equation*}
  which proves the lemma.

\end{proof}

\subsection{Proof of Theorem~\ref{theorem:main}}\label{sec:Proof_main}

Recall that $\hat{\theta}_{l} =\bphi_l^{\top}\by/\sqrt{n}$, $\theta_l=\bphi_l^{\top}\bbf/\sqrt{n}$, and $e_l=\bphi_l^{\top}\beps/\sqrt{n}$, $l=1,\ldots,p_n$. Define $z_l=\sqrt{n}e_l/\sigma,l=1,\ldots,k_{M_n}$, $\hat{\gamma}_j=\sum_{m=j}^{M_n}\hat{w}_m$, $\gamma_j^*=\sum_{m=j}^{M_n}w_m^*$, $j=1,\ldots,M_n$, where $\hat{w}_m$ and $w_m^*$ are $m$-th elements of $\hat{\bw}|\mathcal{M}$ and $\bw^*|\mathcal{M}$, respectively.
Based on (\ref{eq:lossw}) and (\ref{eq:cr}), we have
\begin{equation}\label{eq:addd1}
\begin{split}
    & L_n(\bw|\mathcal{M},\bbf)- C_{n}(\bw|\mathcal{M},\by) \\
     & = 2\sum_{j=1}^{M_n}\sum_{l=k_{j-1}+1}^{k_j}\left[\gamma_j\hat{\theta}_l(\hat{\theta}_l-\theta_l) \right]-2\sum_{j=1}^{M_n}\sum_{l=k_{j-1}+1}^{k_j}\gamma_j\frac{\hat{\sigma}^2}{n}+\sum_{j=1}^{p_n}\theta_j^2-\frac{1}{n}\sum_{i=1}^{n}y_i^2\\
     & = 2\sum_{j=1}^{M_n}\sum_{l=k_{j-1}+1}^{k_j}\gamma_j\left( e_l^2-\frac{\sigma^2}{n}+\theta_le_l \right)+2\sum_{j=1}^{M_n}\sum_{l=k_{j-1}+1}^{k_j}\gamma_j\left( \frac{\sigma^2}{n} - \frac{\hat{\sigma}^2}{n} \right)\\
     & \quad+ \frac{1}{n}\sum_{j=1}^{p_n} \left(\bphi_j^{\top}\bbf\right)^2 - \frac{1}{n}\|\bbf \|^2 - \frac{1}{n}\bbf^{\top}\beps - \frac{1}{n}\|\beps \|^2\\
     & = 2\sum_{j=1}^{M_n}\sum_{l=k_{j-1}+1}^{k_j}\gamma_j\left( e_l^2-\frac{\sigma^2}{n}+\theta_le_l \right)+2\sum_{j=1}^{M_n}\sum_{l=k_{j-1}+1}^{k_j}\gamma_j\left( \frac{\sigma^2}{n} - \frac{\hat{\sigma}^2}{n} \right)\\
     & \quad- \frac{1}{n}\bbf^{\top}\beps - \frac{1}{n}\|\beps \|^2,
\end{split}
\end{equation}
where the second equality follows from $\hat{\theta}_{l} = \theta_l + e_l$ and $\theta_j=\bphi_j^{\top}\bbf/\sqrt{n}$, and the last step follows from $\|\bbf \|^2=\sum_{j=1}^{p_n} \left(\bphi_j^{\top}\bbf\right)^2$. In addition, for any non-random $\bw|\mathcal{M}$, we have
\begin{equation}\label{eq:addd2}
  \begin{split}
      & \mathbb{E}C_{n}(\bw|\mathcal{M},\by)-R_n(\bw|\mathcal{M},\bbf) \\
       & =\mathbb{E}\sum_{j=1}^{M_n}\sum_{l=k_{j-1}+1}^{k_j}\left[(\gamma_j^2-2\gamma_j)(\hat{\theta}_l^2 - \theta_l^2) + 2\gamma_j \frac{\hat{\sigma}^2}{n} - \gamma_j^2\frac{\sigma^2}{n}  \right]\\
       &\quad + \frac{1}{n}\mathbb{E}\sum_{i=1}^{n}y_i^2 - \sum_{j=1}^{p_n}\theta_j^2\\
       & = 2\mathbb{E}\sum_{j=1}^{M_n}\sum_{l=k_{j-1}+1}^{k_j}\gamma_j\left( \frac{\hat{\sigma}^2}{n} - \frac{\sigma^2}{n}\right)+ \frac{1}{n}\mathbb{E}\left(\|\bbf\|^2+2\bbf^{\top}\beps + \|\beps \|^2 \right) - \sum_{j=1}^{p_n}\theta_j^2\\
       & = 2\mathbb{E}\sum_{j=1}^{M_n}\sum_{l=k_{j-1}+1}^{k_j}\gamma_j\left( \frac{\hat{\sigma}^2}{n} - \frac{\sigma^2}{n}\right) + \sigma^2,
  \end{split}
\end{equation}
where the first equality follows from (\ref{eq:riskw}) and (\ref{eq:cr}), the second equality follows from $\mathbb{E}{\hat{\theta}_l^2}= \theta_l^2+\sigma^2/n$, and the last equality is due to $\|\bbf \|^2=\sum_{j=1}^{p_n} \left(\bphi_j^{\top}\bbf\right)^2= n\sum_{j=1}^{p_n}\theta_j^2$. Combining (\ref{eq:addd1}) with (\ref{eq:addd2}), we have
\begin{equation}\label{eq:difference1}
  \begin{split}
     \mathbb{E}L_n(\hat{\bw}|\mathcal{M},\bbf)&= \mathbb{E}C_n(\hat{\bw}|\mathcal{M},\by)-\sigma^2
       +2\mathbb{E}\sum_{j=1}^{M_n}\sum_{l=k_{j-1}+1}^{k_j}\hat{\gamma}_j\left( e_l^2-\frac{\sigma^2}{n}+\theta_le_l \right)\\
       &+2\mathbb{E}\sum_{j=1}^{M_n}\sum_{l=k_{j-1}+1}^{k_j}\hat{\gamma}_j\left( \frac{\sigma^2}{n} - \frac{\hat{\sigma}^2}{n} \right)\\
\leq R_n(\bw^*|\mathcal{M},\bbf)&+\frac{2\sigma^2}{n}\mathbb{E}\left|\sum_{j=1}^{M_n}\sum_{l=k_{j-1}+1}^{k_j}\hat{\gamma}_j\left( z_l^2-1\right) \right| +\frac{2\sigma}{\sqrt{n}}\mathbb{E}\left|\sum_{j=1}^{M_n}\sum_{l=k_{j-1}+1}^{k_j}(1-\hat{\gamma}_j)\theta_lz_l\right|\\
&+2\mathbb{E}\left|\sum_{j=1}^{M_n}\sum_{l=k_{j-1}+1}^{k_j}\hat{\gamma}_j\left(\frac{\sigma^2}{n}-\frac{\hat{\sigma}^2}{n}\right)\right|+2\mathbb{E}\left|\sum_{j=1}^{M_n}\sum_{l=k_{j-1}+1}^{k_j}\gamma_j^*\left(\frac{\sigma^2}{n}-\frac{\hat{\sigma}^2}{n}\right)\right|,
  \end{split}
\end{equation}
where the inequality in (\ref{eq:difference1}) follows from $C_n(\hat{\bw}|\mathcal{M},\by) \leq C_n(\bw^*|\mathcal{M},\by)$ and the absolute value inequalities, and $z_l=\sqrt{n}e_l/\sigma$, $l=1,\ldots,k_{M_n}$. From (\ref{eq:riskw}) with (\ref{eq:cr}), in the same manner we can see that
\begin{equation}\label{eq:addd3}
  \begin{split}
      & R_n(\bw|\mathcal{M},\bbf)- C_{n}(\bw|\mathcal{M},\by) \\
       & = \sum_{j=1}^{M_n}\sum_{l=k_{j-1}+1}^{k_j} \left[ (\gamma_j^2-2\gamma_j)(\theta_l^2 - \hat{\theta}_l^2) +\gamma_j^2\frac{\sigma^2}{n} -2\gamma_j\frac{\hat{\sigma}^2}{n} \right]+\sum_{j=1}^{p_n}\theta_j^2-\frac{1}{n}\sum_{i=1}^{n}y_i^2\\
       & = \sum_{j=1}^{M_n}\sum_{l=k_{j-1}+1}^{k_j}(\gamma_j^2-2\gamma_j)\left(\frac{\sigma^2}{n}-e_l^2-2\theta_le_l\right)+2    \sum_{j=1}^{M_n}\sum_{l=k_{j-1}+1}^{k_j}\gamma_j\left(\frac{\sigma^2}{n}-\frac{\hat{\sigma}^2}{n}\right)\\
       & \quad- \frac{1}{n}\bbf^{\top}\beps - \frac{1}{n}\|\beps \|^2,
  \end{split}
\end{equation}
where the second equality follows from $\hat{\theta}_l^2=\theta_l^2 + 2\theta_le_l +e_l^2$. Combining (\ref{eq:addd3}) with (\ref{eq:addd2}), we have
\begin{equation}\label{eq:difference2}
  \begin{split}
     \mathbb{E}R_n(\hat{\bw}|\mathcal{M},\bbf) &=\mathbb{E}C_n(\hat{\bw}|\mathcal{M},\by)-\sigma^2+\mathbb{E}\sum_{j=1}^{M_n}\sum_{l=k_{j-1}+1}^{k_j}(\hat{\gamma}_j^2-2\hat{\gamma}_j)\left(\frac{\sigma^2}{n}-e_l^2-2\theta_le_l\right)\\
     &+2    \mathbb{E}\sum_{j=1}^{M_n}\sum_{l=k_{j-1}+1}^{k_j}\hat{\gamma}_j\left(\frac{\sigma^2}{n}-\frac{\hat{\sigma}^2}{n}\right)  \\
\leq R_n(\bw^*|\mathcal{M},\bbf)&+\frac{\sigma^2}{n}\mathbb{E}\left|\sum_{j=1}^{M_n}\sum_{l=k_{j-1}+1}^{k_j}\hat{\gamma}_j^2\left( z_l^2-1\right) \right| +\frac{2\sigma^2}{n}\mathbb{E}\left|\sum_{j=1}^{M_n}\sum_{l=k_{j-1}+1}^{k_j}\hat{\gamma}_j\left( z_l^2-1\right) \right| \\
&+\frac{2\sigma}{\sqrt{n}}\mathbb{E}\left|\sum_{j=1}^{M_n}\sum_{l=k_{j-1}+1}^{k_j}(1-\hat{\gamma}_j)^2\theta_lz_l\right|+2\mathbb{E}\left|\sum_{j=1}^{M_n}\sum_{l=k_{j-1}+1}^{k_j}\hat{\gamma}_j\left(\frac{\sigma^2}{n}-\frac{\hat{\sigma}^2}{n}\right)\right|\\
&+2\mathbb{E}\left|\sum_{j=1}^{M_n}\sum_{l=k_{j-1}+1}^{k_j}\gamma_j^*\left(\frac{\sigma^2}{n}-\frac{\hat{\sigma}^2}{n}\right)\right|.
  \end{split}
\end{equation}
The main idea of the proof is to take the upper bounds of the terms in (\ref{eq:difference1}) and (\ref{eq:difference2}).

\subsubsection{Upper bounding $\mathbb{E}\left|\sum_{j=1}^{M_n}\sum_{l=k_{j-1}+1}^{k_j}\hat{\gamma}_j( z_l^2-1) \right|$}\label{sec:upper_bounding_term_1}

We first bound $\mathbb{E}\left|\sum_{j=1}^{M_n}\sum_{l=k_{j-1}+1}^{k_j}\hat{\gamma}_j( z_l^2-1) \right|$. Define $k_0=0$, $\hat{\gamma}_{M_n+1}=0$, and a random variable $\kappa_1=\max_{1\leq j \leq M_n}\{| \sum_{l=1}^{k_j}(z_l^2-1) |k_j^{-1/2}\}$. Note that
\begin{equation}\label{eq:F1}
\begin{split}
   &\sum_{j=1}^{M_n}\frac{\left(  k_j^{\frac{1}{2}} -k_{j-1}^{\frac{1}{2}}\right)^2}{k_j-k_{j-1}} = 1+ \sum_{j=2}^{M_n}\left(\frac{  k_j^{\frac{1}{2}} -k_{j-1}^{\frac{1}{2}}}{k_j-k_{j-1}}\right)^2(k_j-k_{j-1})\\
   &
   \leq 1+ \sum_{j=2}^{M_n}\frac{k_{j} - k_{j-1}}{4k_{j-1}}
   =1+ \sum_{j=1}^{M_n-1}\frac{k_{j+1} - k_{j}}{4k_{j}},
\end{split}
\end{equation}
where the inequality is due to the concavity of the function $h(x)=x^{1/2}$. Using summation by parts, we can rewrite $\mathbb{E}\left|\sum_{j=1}^{M_n}\sum_{l=k_{j-1}+1}^{k_j}\hat{\gamma}_j( z_l^2-1) \right|$ as
\begin{equation}\label{eq:step21}
  \begin{split}
     \mathbb{E}\left|\sum_{j=1}^{M_n}\sum_{l=k_{j-1}+1}^{k_j}\hat{\gamma}_j( z_l^2-1) \right|
      &= \mathbb{E}\left|\sum_{j=1}^{M_n}(\hat{\gamma}_j-\hat{\gamma}_{j+1})\sum_{l=1}^{k_j}(z_l^2-1)\right|\\
     &\leq \mathbb{E}\left\{\kappa_1\sum_{j=1}^{M_n}(\hat{\gamma}_j-\hat{\gamma}_{j+1})k_j^{\frac{1}{2}}\right\}\\
     &=\mathbb{E}\left\{ \kappa_1\sum_{j=1}^{M_n}\hat{\gamma}_j\left(  k_j^{\frac{1}{2}}-k_{j-1}^{\frac{1}{2}}\right)\right\}\\
     &\leq\mathbb{E}\left\{ \kappa_1 \left[\sum_{j=1}^{M_n}\hat{\gamma}_j^2\left(k_j-k_{j-1} \right)\right]^{\frac{1}{2}}  \left[\sum_{j=1}^{M_n}\frac{\left(  k_j^{\frac{1}{2}} -k_{j-1}^{\frac{1}{2}}\right)^2}{k_j-k_{j-1}}\right]^{\frac{1}{2}}\right\}\\
      &\leq \frac{\sqrt{n}}{\sigma}\left(\mathbb{E}\kappa_1^2\right)^{\frac{1}{2}}\left[\mathbb{E}R_n(\hat{\bw}|\mathcal{M},\bbf)\right]^{\frac{1}{2}}\left(1+ \sum_{j=1}^{M_n-1}\frac{k_{j+1} - k_{j}}{4k_{j}}\right)^{\frac{1}{2}},\\
  \end{split}
\end{equation}
where the first inequality follows from the definition of $\kappa_1$, the second inequality follows from the Cauchy-Schwarz inequality, and the third inequality follows from the Cauchy-Schwarz inequality, (\ref{eq:riskw}), and (\ref{eq:F1}).

The task is now to construct an upper bound for $\left(\mathbb{E}\kappa_1^2\right)^{1/2}$ by Lemma~\ref{lemma}. It remains to check (\ref{eq:lemma_con}) for the stochastic process $\{\xi_1(t)=\sum_{l=1}^{k_t}(z_l^2-1) k_t^{-1/2},t=1,\ldots,M_n\}$. Recall that $z_l=\sqrt{n}e_l/\sigma=\bphi_l^{\top}\beps/\sigma$. Define an $n \times n$ matrix
$$
\bA \triangleq \frac{\sum_{l=1}^{k_t}\bphi_l \bphi_l^{\top}}{\sigma^2\sqrt{k_t}}.
$$
Then we can write $\xi_1(t)$ as
\begin{equation*}
\begin{split}
   \xi_1(t) & = \beps^{\top}\left(\frac{\sum_{l=1}^{k_t}\bphi_l \bphi_l^{\top}}{\sigma^2\sqrt{k_t}} \right)\beps - \sqrt{k_t} = \beps^{\top} \bA \beps - \mathbb{E}\beps^{\top} \bA \beps,
\end{split}
\end{equation*}
where the second inequality follows from
\begin{equation*}
\begin{split}
   \mathbb{E}\beps^{\top} \bA \beps = \sigma^2 \tr \bA& = \frac{\tr\left(\sum_{l=1}^{k_t}\bphi_l \bphi_l^{\top} \right)}{\sqrt{k_t}} = \frac{\sum_{l=1}^{k_t}\tr\left(\bphi_l \bphi_l^{\top} \right)}{\sqrt{k_t}} \\
     & = \frac{\sum_{l=1}^{k_t}\tr\left(\bphi_l^{\top}\bphi_l \right)}{\sqrt{k_t}} = \sqrt{k_t}.
\end{split}
\end{equation*}
Using Hansen-Wright inequality for sub-Gaussian random variables (Theorem 1.1 of \cite{Rudelson2013}), we know that there exists a positive absolute constant $c$ such that for any $x\geq 0$,
\begin{equation}\label{eq:HW_ineq}
\begin{split}
\mathbb{P}\left( \left|\xi_1(t)\right| >x \right)& = \mathbb{P}\left( \left|\beps^{\top} \bA \beps - \mathbb{E}\beps^{\top} \bA \beps\right|>x \right)\\
    &\leq 2\exp\left[ -c\min\left(\frac{x}{\eta^2\|\bA \|_2} , \frac{x^2}{\eta^4\|\bA \|_{\mathrm{F}}^2} \right) \right]
\\
 &\leq 2\exp\left[ -c\min\left(x , x^2 \right) \right],
\end{split}
\end{equation}
where the second inequality follows from $\|\bA \|_2 = 1/(\sigma^2\sqrt{k_t})\leq 1/\sigma^2$ and $\|\bA \|_{\mathrm{F}}^2=\tr(\bA^{\top}\bA)=1/\sigma^4$. The inequality (\ref{eq:HW_ineq}) also implies that
\begin{equation*}
\mathbb{P}\left( \left|\xi_1(t)\right|>\frac{\log x}{\lambda} \right)\leq \left\{\begin{array}{ll}
2x^{-\frac{c}{\lambda^2}\log x} &\quad 0\leq x<\exp(\lambda), \\
2x^{-\frac{c}{\lambda}} &\quad x\geq \exp(\lambda),\\
\end{array}\right.
\end{equation*}
where $\lambda>0$. Thus we have
\begin{equation}\label{eq:expect_xi}
\begin{split}
   \mathbb{E}\exp(\lambda| \xi_1(t) |) & = \int_{0}^{\infty}\mathbb{P}(\exp(\lambda| \xi_1(t) |)>x)dx = \int_{0}^{\infty}\mathbb{P}\left(| \xi_1(t) |>\frac{\log x}{\lambda}\right)dx \\
     & \leq 2\int_{0}^{\exp(\lambda)}x^{-\frac{c}{\lambda^2}\log x}dx+2\int_{\exp(\lambda)}^{\infty}x^{-\frac{c}{\lambda}}dx.
\end{split}
\end{equation}
When $0<\lambda<c$, the first term of (\ref{eq:expect_xi}) is upper bounded by
\begin{equation}\label{eq:expect_xi_1}
\begin{split}
   2\int_{0}^{\exp(\lambda)}x^{-\frac{c}{\lambda^2}\log x}dx & = \frac{2\lambda^2}{c}\int_{-\frac{c}{\lambda}}^{\infty}\exp\left[ -\frac{\lambda^2(u^2+u)}{c} \right]du\\
     & \leq \frac{2\lambda^2}{c}\exp\left(\frac{\lambda^2}{4c} \right)\sqrt{\frac{\pi c}{\lambda^2}}\\
     &\leq 2\exp\left(\frac{c}{4}\right)\sqrt{\pi c}<\infty.
\end{split}
\end{equation}
And the second term of (\ref{eq:expect_xi}) is
\begin{equation}\label{eq:expect_xi_2}
  2\int_{\exp(\lambda)}^{\infty}x^{-\frac{c}{\lambda}}dx = \frac{2}{\frac{c}{\lambda}-1}\exp(-c+\lambda)< \infty.
\end{equation}
Combining (\ref{eq:expect_xi_1})--(\ref{eq:expect_xi_2}) with (\ref{eq:expect_xi}), we see that when $0<\lambda<c$, $\mathbb{E}\exp(\lambda| \xi_1(t) |)$ is uniformly upper bounded for any $t=1,\ldots,M_n$, which meets the condition (\ref{eq:lemma_con}) of Lemma~\ref{lemma}. Thus we have
$$
\left(\mathbb{E}\kappa_1^2\right)^{\frac{1}{2}} \leq C (1+\log M_n),
$$
and the term $\mathbb{E}\left|\sum_{j=1}^{M_n}\sum_{l=k_{j-1}+1}^{k_j}\hat{\gamma}_j( z_l^2-1) \right|$ is upper bounded by
\begin{equation}\label{eq:step211}
  \mathbb{E}\left|\sum_{j=1}^{M_n}\sum_{l=k_{j-1}+1}^{k_j}\hat{\gamma}_j( z_l^2-1) \right| \leq \frac{C\sqrt{n}}{\sigma}[\mathbb{E}R_n(\hat{\bw}|\mathcal{M},\bbf)]^{\frac{1}{2}}\left(1+ \sum_{j=1}^{M_n-1}\frac{k_{j+1} - k_{j}}{4k_{j}}\right)^{\frac{1}{2}}(1+\log M_n).
\end{equation}

Now we construct another upper bound for $\mathbb{E}\left|\sum_{j=1}^{M_n}\sum_{l=k_{j-1}+1}^{k_j}\hat{\gamma}_j( z_l^2-1) \right|$. Define a random variable $\kappa_2=\max_{1\leq j \leq M_n}\{| \sum_{l=k_{j-1}+1}^{k_j}(z_l^2-1) |(k_j-k_{j-1})^{-1/2}\}$. Note that
\begin{equation*}\label{eq:step22}
  \begin{split}
     \mathbb{E}\left|\sum_{j=1}^{M_n}\sum_{l=k_{j-1}+1}^{k_j}\hat{\gamma}_j( z_l^2-1) \right|
      &= \mathbb{E}\left|\sum_{j=1}^{M_n}\hat{\gamma}_j\sum_{l=k_{j-1}+1}^{k_j}(z_l^2-1)\right|\\
     &\leq \mathbb{E}\left\{\kappa_2\sum_{j=1}^{M_n}\hat{\gamma}_j\left(k_j-k_{j-1}\right)^{\frac{1}{2}}\right\}\\
     &\leq\mathbb{E}\left\{ \kappa_2 \left[\sum_{j=1}^{M_n}\hat{\gamma}_j^2\left(k_j-k_{j-1} \right)\right]^{\frac{1}{2}}  M_n^{\frac{1}{2}}\right\}\\
      &\leq \frac{\sqrt{n}}{\sigma}\left(\mathbb{E}\kappa_2^2\right)^{\frac{1}{2}}\left[\mathbb{E}R_n(\hat{\bw}|\mathcal{M},\bbf)\right]^{\frac{1}{2}}M_n^{\frac{1}{2}},\\
  \end{split}
\end{equation*}
where the first inequality follows from the definition of $\kappa_2$, the second inequality follows from the Cauchy-Schwarz inequality, and the third inequality follows from the Cauchy-Schwarz inequality and (\ref{eq:riskw}). Based on the same technique used to bound $\left(\mathbb{E}\kappa_1^2\right)^{\frac{1}{2}}$, we can derive
$$
\left(\mathbb{E}\kappa_2^2\right)^{\frac{1}{2}} \leq C (1+\log M_n).
$$
Consequently,
\begin{equation}\label{eq:step2111}
  \mathbb{E}\left|\sum_{j=1}^{M_n}\sum_{l=k_{j-1}+1}^{k_j}\hat{\gamma}_j( z_l^2-1) \right| \leq \frac{C\sqrt{n}}{\sigma}[\mathbb{E}R_n(\hat{\bw}|\mathcal{M},\bbf)]^{\frac{1}{2}}M_n^{\frac{1}{2}}(1+\log M_n).
\end{equation}

Then, combining (\ref{eq:step211}) with (\ref{eq:step2111}), we obtain
\begin{equation}\label{eq:upper_bounding_term_1}
  \mathbb{E}\left|\sum_{j=1}^{M_n}\sum_{l=k_{j-1}+1}^{k_j}\hat{\gamma}_j( z_l^2-1) \right| \leq \frac{C\sqrt{n}}{\sigma}[\mathbb{E}R_n(\hat{\bw}|\mathcal{M},\bbf)]^{\frac{1}{2}}\left[M_n\wedge\left(1+ \sum_{j=1}^{M_n-1}\frac{k_{j+1} - k_{j}}{4k_{j}}\right)\right]^{\frac{1}{2}}(1+\log M_n).
\end{equation}

\subsubsection{Upper bounding $\mathbb{E}\left|\sum_{j=1}^{M_n}\sum_{l=k_{j-1}+1}^{k_j}(1-\hat{\gamma}_j)\theta_lz_l\right|$}\label{sec:upper_bounding_term_2}

Define $S_t=\sum_{l=k_t+1}^{k_{M_n}}\theta_l^2$ for $0\leq t \leq M_n-1$ and $S_{M_n}=0$, and a random variable $\kappa_3=\max_{0\leq t \leq M_n-1}\left\{\left| \sum_{l=k_t+1}^{k_{M_n}}\theta_lz_l \right|S_t^{-1/2}\right\}$. In addition, define $M_n'=\min\{1 \leq m \leq M_n: S_m=0 \}$. Note that
\begin{equation}\label{eq:F2}
  \begin{split}
\sum_{j=1}^{M_n}\frac{\left(S_{j-1}^{\frac{1}{2}}-S_{j}^{\frac{1}{2}}\right)^2}{S_{j-1} - S_{j}}&= \sum_{j=1}^{M_n'}\frac{\left(S_{j-1}^{\frac{1}{2}}-S_{j}^{\frac{1}{2}}\right)^2}{S_{j-1} - S_{j}} + \sum_{j=M_n'+1}^{M_n}\frac{\left(S_{j-1}^{\frac{1}{2}}-S_{j}^{\frac{1}{2}}\right)^2}{S_{j-1} - S_{j}} \\
       & = 1 + \sum_{j=1}^{M_n'-1}\left(\frac{S_{j-1}^{\frac{1}{2}}-S_{j}^{\frac{1}{2}}}{S_{j-1} - S_{j}}\right)^2\left( S_{j-1} - S_{j} \right)\\
       & \leq 1 + \sum_{j=1}^{M_n'-1} \frac{S_{j-1} - S_{j}}{4S_{j}},
  \end{split}
\end{equation}
where the second equality follows from the definition $0/0 = 0$, and the last inequality is due to the concavity of $h(x)=x^{1/2}$. Using summation by parts again, we see that
\begin{equation}\label{eq:important2}
\begin{split}
   \mathbb{E}\left|\sum_{j=1}^{M_n}\sum_{l=k_{j-1}+1}^{k_j}(1-\hat{\gamma}_j)\theta_lz_l\right|
     &= \mathbb{E}\left|\sum_{j=2}^{M_n}(\hat{\gamma}_{j-1}-\hat{\gamma}_j)\sum_{l=k_{j-1}+1}^{k_{M_n}}\theta_lz_l \right|\\
   &\leq \mathbb{E}\left\{\kappa_3 \sum_{j=2}^{M_n}(\hat{\gamma}_{j-1}-\hat{\gamma}_j)S_{j-1}^{\frac{1}{2}}\right\}\\
   &= \mathbb{E}\left\{\kappa_3 \sum_{j=1}^{M_n}(1-\hat{\gamma}_j)\left(S_{j-1}^{\frac{1}{2}}-S_{j}^{\frac{1}{2}}\right)\right\}\\
   & \leq \mathbb{E}\left\{\kappa_3 \left[\sum_{j=1}^{M_n}(1-\hat{\gamma}_j)^2(S_{j-1} - S_{j} )\right]^{\frac{1}{2}}\left[\sum_{j=1}^{M_n}\frac{\left(S_{j-1}^{\frac{1}{2}}-S_{j}^{\frac{1}{2}}\right)^2}{S_{j-1} - S_{j}}\right]^{\frac{1}{2}}\right\}\\
   & \leq (\mathbb{E} \kappa_3^2)^{\frac{1}{2}}[\mathbb{E}R_n(\hat{\bw}|\mathcal{M},\bbf)]^{\frac{1}{2}}\left(1 + \sum_{j=1}^{M_n'-1} \frac{S_{j-1} - S_{j}}{4S_{j}}\right)^{\frac{1}{2}},\\
\end{split}
\end{equation}
where the first inequality is due to the definition of $\kappa_3$, the second inequality follows from the Cauchy-Schwarz inequality and the definition $0/0=0$, and the third inequality follows from the Cauchy-Schwarz inequality, (\ref{eq:riskw}), and (\ref{eq:F2}).

Now we construct upper bound for $(\mathbb{E} \kappa_3^2)^{1/2}$ by Lemma~\ref{lemma}. Consider the stochastic process $\{\xi_3(t) = ( \sum_{l=k_t+1}^{k_{M_n}}\theta_lz_l )S_t^{-1/2},t=0,\ldots,M_n-1\}$. Recall that $z_l=\bphi_l^{\top}\beps/\sigma$. Define an $n$-dimensional vector
$$
\ba\triangleq \frac{1}{\sigma S_t^{\frac{1}{2}}}\left(\bphi_{k_t+1},\ldots, \bphi_{k_{M_n}} \right)\begin{pmatrix}
\theta_{k_t+1} \\
\vdots \\
\theta_{k_{M_n}}
\end{pmatrix}.
$$
We write $\xi_3(t)$ as
\begin{equation*}
  \xi_3(t)= \frac{1}{\sigma S_t^{\frac{1}{2}}}\left(\theta_{k_t+1},\ldots,\theta_{k_{M_n}}  \right)\begin{pmatrix}
\bphi_{k_t+1}^{\top} \\
\vdots \\
\bphi_{k_{M_n}}^{\top}
\end{pmatrix}\beps= \ba^{\top}\beps.
\end{equation*}
Since the elements of $\beps$ are i.i.d. $\eta$-sub-Gaussian variables, from Theorem 2.6 in \cite{Wainwright2019}, we have for any $\lambda\in \mathbb{R}$,
\begin{equation*}
  \mathbb{E}\exp[\lambda\xi_3(t)]=\mathbb{E}\exp(\lambda\ba^{\top}\beps)\leq \exp\left(\frac{\lambda^2\eta^2\|\ba\|^2}{2}\right) = \exp\left(\frac{\lambda^2\eta^2}{2\sigma^2}\right),
\end{equation*}
where the last equality is due to $\|\ba\|^2=1/\sigma^2$. This leads to
\begin{equation*}
\begin{split}
    & \mathbb{E}\exp(\lambda| \xi_3(t) |) \leq \mathbb{E}\exp[\lambda \xi_3(t) ] + \mathbb{E}\exp[-\lambda \xi_3(t) ]  = 2\exp\left(\frac{\lambda^2\eta^2}{2\sigma^2}\right)< \infty,\\
\end{split}
\end{equation*}
which verifies the condition (\ref{eq:lemma_con}) of Lemma~\ref{lemma}. Thus, there exists a constant $C>0$ such that
\begin{equation}\label{eq:kappa_upper_2}
  (\mathbb{E} \kappa_3^2)^{\frac{1}{2}} \leq C\left(1+ \log M_n \right).
\end{equation}
Thus combining (\ref{eq:kappa_upper_2}) with (\ref{eq:important2}), we have the second term
\begin{equation}\label{eq:important211}
  \mathbb{E}\left|\sum_{j=1}^{M_n}\sum_{l=k_{j-1}+1}^{k_j}(1-\hat{\gamma}_j)\theta_lz_l\right| \leq C [\mathbb{E}R_n(\hat{\bw}|\mathcal{M},\bbf)]^{\frac{1}{2}} \left(1 + \sum_{j=1}^{M_n'-1} \frac{S_{j-1} - S_{j}}{4S_{j}}\right)^{\frac{1}{2}}\left(1+ \log M_n \right).
\end{equation}

We now construct another upper bound for $\mathbb{E}\left|\sum_{j=1}^{M_n}\sum_{l=k_{j-1}+1}^{k_j}(1-\hat{\gamma}_j)\theta_lz_l\right|$. Define $\Delta_t=\sum_{l=k_t+1}^{k_{t+1}}\theta_l^2$ for $0\leq t \leq M_n-1$ and a random variable $\kappa_4=\max_{0\leq t \leq M_n-1}\left\{\left| \sum_{l=k_t+1}^{k_{t+1}}\theta_lz_l \right|\Delta_t^{-1/2}\right\}$. Then we have that
\begin{equation}\label{eq:important22}
\begin{split}
   \mathbb{E}\left|\sum_{j=1}^{M_n}\sum_{l=k_{j-1}+1}^{k_j}(1-\hat{\gamma}_j)\theta_lz_l\right|
     &= \mathbb{E}\left|\sum_{j=1}^{M_n}(1-\hat{\gamma}_j)\sum_{l=k_{j-1}+1}^{k_{j}}\theta_lz_l \right|\\
   &\leq \mathbb{E}\left\{\kappa_4 \sum_{j=1}^{M_n}(1-\hat{\gamma}_j)\Delta_{j-1}^{\frac{1}{2}}\right\}\\
   & \leq \mathbb{E}\left\{\kappa_4 \left[\sum_{j=1}^{M_n}(1-\hat{\gamma}_j)^2\Delta_{j-1}\right]^{\frac{1}{2}}M_n^{\frac{1}{2}}\right\}\\
   & \leq (\mathbb{E} \kappa_4^2)^{\frac{1}{2}}[\mathbb{E}R_n(\hat{\bw}|\mathcal{M},\bbf)]^{\frac{1}{2}}M_n^{\frac{1}{2}},\\
\end{split}
\end{equation}
where the first inequality is due to the definition of $\kappa_4$, the second inequality follows from Jensen's inequality, and the last inequality follows from the Cauchy-Schwarz inequality and (\ref{eq:riskw}). Based on the same technique used to upper bound $(\mathbb{E} \kappa_3^2)^{\frac{1}{2}}$, we know there also exists a constant $C>0$ such that
\begin{equation}\label{eq:kappa_upper_3}
  (\mathbb{E} \kappa_4^2)^{\frac{1}{2}} \leq C\left(1+ \log M_n \right).
\end{equation}
Thus combining (\ref{eq:kappa_upper_3}) with (\ref{eq:important22}), we have the second term is also upper bounded by
\begin{equation}\label{eq:important212}
  \mathbb{E}\left|\sum_{j=1}^{M_n}\sum_{l=k_{j-1}+1}^{k_j}(1-\hat{\gamma}_j)\theta_lz_l\right| \leq C [\mathbb{E}R_n(\hat{\bw}|\mathcal{M},\bbf)]^{\frac{1}{2}} M_n^{\frac{1}{2}}\left(1+ \log M_n \right).
\end{equation}
Combining (\ref{eq:important211}) with (\ref{eq:important212}), we have
\begin{equation}\label{eq:important213}
  \mathbb{E}\left|\sum_{j=1}^{M_n}\sum_{l=k_{j-1}+1}^{k_j}(1-\hat{\gamma}_j)\theta_lz_l\right| \leq C [\mathbb{E}R_n(\hat{\bw}|\mathcal{M},\bbf)]^{\frac{1}{2}} \left[M_n \wedge \left(1 + \sum_{j=1}^{M_n'-1} \frac{S_{j-1} - S_{j}}{4S_{j}}\right)\right]^{\frac{1}{2}}\left(1+ \log M_n \right).
\end{equation}

\subsubsection{Upper bounding $\mathbb{E}\left|\sum_{j=1}^{M_n}\sum_{l=k_{j-1}+1}^{k_j}\hat{\gamma}_j^2( z_l^2-1) \right|$}

Based on the same reasoning adopted in (\ref{eq:step211}) and (\ref{eq:important211}), and the fact that $0\leq \hat{\gamma}_j\leq 1,j=1,\ldots,M_n$, we can also prove that
\begin{equation}\label{eq:important3}
  \begin{split}
     \mathbb{E}\left|\sum_{j=1}^{M_n}\sum_{l=k_{j-1}+1}^{k_j}\hat{\gamma}_j^2( z_l^2-1) \right|
      &= \mathbb{E}\left|\sum_{j=1}^{M_n}(\hat{\gamma}_j^2-\hat{\gamma}_{j+1}^2)\sum_{l=1}^{k_j}(z_l^2-1)\right|\\
     &\leq \mathbb{E}\left\{\kappa_1\sum_{j=1}^{M_n}(\hat{\gamma}_j^2-\hat{\gamma}_{j+1}^2)k_j^{\frac{1}{2}}\right\}\\
     &=\mathbb{E}\left\{ \kappa_1\sum_{j=1}^{M_n}\hat{\gamma}_j^2\left(  k_j^{\frac{1}{2}}-k_{j-1}^{\frac{1}{2}}\right)\right\}\\
     &\leq\mathbb{E}\left\{ \kappa_1 \left[\sum_{j=1}^{M_n}\hat{\gamma}_j^4\left(k_j-k_{j-1} \right)\right]^{\frac{1}{2}}  \left[\sum_{j=1}^{M_n}\frac{\left(  k_j^{\frac{1}{2}} -k_{j-1}^{\frac{1}{2}}\right)^2}{k_j-k_{j-1}}\right]^{\frac{1}{2}}\right\}\\
     & \leq \mathbb{E}\left\{ \kappa_1 \left[\sum_{j=1}^{M_n}\hat{\gamma}_j^2\left(k_j-k_{j-1} \right)\right]^{\frac{1}{2}}  \left[\sum_{j=1}^{M_n}\frac{\left(  k_j^{\frac{1}{2}} -k_{j-1}^{\frac{1}{2}}\right)^2}{k_j-k_{j-1}}\right]^{\frac{1}{2}}\right\}\\
      &\leq \frac{\sqrt{n}}{\sigma}\left(\mathbb{E}\kappa_1^2\right)^{\frac{1}{2}}\left[\mathbb{E}R_n(\hat{\bw}|\mathcal{M},\bbf)\right]^{\frac{1}{2}}\left(1+ \sum_{j=1}^{M_n-1}\frac{k_{j+1} - k_{j}}{4k_{j}}\right)^{\frac{1}{2}},\\
  \end{split}
\end{equation}
and
\begin{equation}
  \begin{split}
     \mathbb{E}\left|\sum_{j=1}^{M_n}\sum_{l=k_{j-1}+1}^{k_j}\hat{\gamma}_j^2( z_l^2-1) \right|
      &= \mathbb{E}\left|\sum_{j=1}^{M_n}\hat{\gamma}_j^2\sum_{l=k_{j-1}+1}^{k_j}(z_l^2-1)\right|\\
     &\leq \mathbb{E}\left\{\kappa_2\sum_{j=1}^{M_n}\hat{\gamma}_j^2\left(k_j-k_{j-1}\right)^{\frac{1}{2}}\right\}\\
     &\leq\mathbb{E}\left\{ \kappa_2 \left[\sum_{j=1}^{M_n}\hat{\gamma}_j^4\left(k_j-k_{j-1} \right)\right]^{\frac{1}{2}}  M_n^{\frac{1}{2}}\right\}\\
     & \leq\mathbb{E}\left\{ \kappa_2 \left[\sum_{j=1}^{M_n}\hat{\gamma}_j^2\left(k_j-k_{j-1} \right)\right]^{\frac{1}{2}}  M_n^{\frac{1}{2}}\right\}\\
      &\leq \frac{\sqrt{n}}{\sigma}\left(\mathbb{E}\kappa_2^2\right)^{\frac{1}{2}}\left[\mathbb{E}R_n(\hat{\bw}|\mathcal{M},\bbf)\right]^{\frac{1}{2}}M_n^{\frac{1}{2}}.\\
  \end{split}
\end{equation}
Combining with the upper bounds on $\left(\mathbb{E}\kappa_1^2\right)^{\frac{1}{2}}$ and $\left(\mathbb{E}\kappa_2^2\right)^{\frac{1}{2}}$ in Section~\ref{sec:upper_bounding_term_1}, we see
\begin{equation}\label{eq:upper_bounding_term_3}
  \mathbb{E}\left|\sum_{j=1}^{M_n}\sum_{l=k_{j-1}+1}^{k_j}\hat{\gamma}_j^2( z_l^2-1) \right| \leq \frac{C\sqrt{n}}{\sigma}[\mathbb{E}R_n(\hat{\bw}|\mathcal{M},\bbf)]^{\frac{1}{2}}\left[M_n\wedge\left(1+ \sum_{j=1}^{M_n-1}\frac{k_{j+1} - k_{j}}{4k_{j}}\right)\right]^{\frac{1}{2}}(1+\log M_n).
\end{equation}

\subsubsection{Upper bounding $\mathbb{E}\left|\sum_{j=1}^{M_n}\sum_{l=k_{j-1}+1}^{k_j}(1-\hat{\gamma}_j)^2\theta_lz_l\right|$}

Based on the similar technique in Section~\ref{sec:upper_bounding_term_2}, we obtain
\begin{equation}\label{eq:important4}
\begin{split}
   \mathbb{E}\left|\sum_{j=1}^{M_n}\sum_{l=k_{j-1}+1}^{k_j}(1-\hat{\gamma}_j)^2\theta_lz_l\right|
     &= \mathbb{E}\left|\sum_{j=2}^{M_n}[(1-\hat{\gamma}_j)^2 - (1-\hat{\gamma}_{j-1})^2]\sum_{l=k_{j-1}+1}^{k_{M_n}}\theta_lz_l \right|\\
   &\leq \mathbb{E}\left\{\kappa_3 \sum_{j=2}^{M_n}[(1-\hat{\gamma}_j)^2 - (1-\hat{\gamma}_{j-1})^2]S_{j-1}^{\frac{1}{2}}\right\}\\
   &= \mathbb{E}\left\{\kappa_3 \sum_{j=1}^{M_n}(1-\hat{\gamma}_j)^2\left(S_{j-1}^{\frac{1}{2}}-S_{j}^{\frac{1}{2}}\right)\right\}\\
   & \leq \mathbb{E}\left\{\kappa_3 \left[\sum_{j=1}^{M_n}(1-\hat{\gamma}_j)^4(S_{j-1} - S_{j} )\right]^{\frac{1}{2}}\left[\sum_{j=1}^{M_n}\frac{\left(S_{j-1}^{\frac{1}{2}}-S_{j}^{\frac{1}{2}}\right)^2}{S_{j-1} - S_{j}}\right]^{\frac{1}{2}}\right\}\\
   & \leq \mathbb{E}\left\{\kappa_3 \left[\sum_{j=1}^{M_n}(1-\hat{\gamma}_j)^2(S_{j-1} - S_{j} )\right]^{\frac{1}{2}}\left[\sum_{j=1}^{M_n}\frac{\left(S_{j-1}^{\frac{1}{2}}-S_{j}^{\frac{1}{2}}\right)^2}{S_{j-1} - S_{j}}\right]^{\frac{1}{2}}\right\}\\
   & \leq (\mathbb{E} \kappa_3^2)^{\frac{1}{2}}[\mathbb{E}R_n(\hat{\bw}|\mathcal{M},\bbf)]^{\frac{1}{2}}\left(1 + \sum_{j=1}^{M_n'-1} \frac{S_{j-1} - S_{j}}{4S_{j}}\right)^{\frac{1}{2}},\\
\end{split}
\end{equation}
where the third inequality follows from the fact that $0\leq \hat{\gamma}_j\leq 1,j=1,\ldots,M_n$. In addition, we have
\begin{equation}
\begin{split}
   \mathbb{E}\left|\sum_{j=1}^{M_n}\sum_{l=k_{j-1}+1}^{k_j}(1-\hat{\gamma}_j)^2\theta_lz_l\right|
     &= \mathbb{E}\left|\sum_{j=1}^{M_n}(1-\hat{\gamma}_j)^2\sum_{l=k_{j-1}+1}^{k_{j}}\theta_lz_l \right|\\
   &\leq \mathbb{E}\left\{\kappa_4 \sum_{j=1}^{M_n}(1-\hat{\gamma}_j)^2\Delta_{j-1}^{\frac{1}{2}}\right\}\\
   & \leq \mathbb{E}\left\{\kappa_4 \left[\sum_{j=1}^{M_n}(1-\hat{\gamma}_j)^4\Delta_{j-1}\right]^{\frac{1}{2}}M_n^{\frac{1}{2}}\right\}\\
   & \leq \mathbb{E}\left\{\kappa_4 \left[\sum_{j=1}^{M_n}(1-\hat{\gamma}_j)^2\Delta_{j-1}\right]^{\frac{1}{2}}M_n^{\frac{1}{2}}\right\}\\
   & \leq (\mathbb{E} \kappa_4^2)^{\frac{1}{2}}[\mathbb{E}R_n(\hat{\bw}|\mathcal{M},\bbf)]^{\frac{1}{2}}M_n^{\frac{1}{2}},\\
\end{split}
\end{equation}
where the third inequality follows from $0\leq \hat{\gamma}_j\leq 1,j=1,\ldots,M_n$. Then utilizing the same upper bounds on $\left(\mathbb{E}\kappa_3^2\right)^{\frac{1}{2}}$ and $\left(\mathbb{E}\kappa_4^2\right)^{\frac{1}{2}}$ given in Section~\ref{sec:upper_bounding_term_2}, we obtain
\begin{equation}\label{eq:important214}
  \mathbb{E}\left|\sum_{j=1}^{M_n}\sum_{l=k_{j-1}+1}^{k_j}(1-\hat{\gamma}_j)^2\theta_lz_l\right| \leq C [\mathbb{E}R_n(\hat{\bw}|\mathcal{M},\bbf)]^{\frac{1}{2}} \left[M_n \wedge \left(1 + \sum_{j=1}^{M_n'-1} \frac{S_{j-1} - S_{j}}{4S_{j}}\right)\right]^{\frac{1}{2}}\left(1+ \log M_n \right).
\end{equation}

\subsubsection{Upper bounding $\mathbb{E}\left|\sum_{j=1}^{M_n}\sum_{l=k_{j-1}+1}^{k_j}\hat{\gamma}_j\left(\frac{\sigma^2}{n}-\frac{\hat{\sigma}^2}{n}\right)\right|$}

We observe that
\begin{equation}\label{eq:important5}
\begin{split}
\mathbb{E}\left|\sum_{j=1}^{M_n}\sum_{l=k_{j-1}+1}^{k_j}\hat{\gamma}_j\left(\frac{\sigma^2}{n}-\frac{\hat{\sigma}^2}{n}\right)\right| &= \mathbb{E}\left|\sum_{j=1}^{M_n}\sum_{l=k_{j-1}+1}^{k_j}\hat{\gamma}_j\left(\frac{\sigma^2}{n} \right)^{\frac{1}{2}}\left(\frac{ n}{\sigma^2} \right)^{\frac{1}{2}}\left(\frac{\sigma^2}{n}-\frac{\hat{\sigma}^2}{n}\right)\right| \\
&\leq \mathbb{E}\left\{\left(\sum_{j=1}^{M_n}\sum_{l=k_{j-1}+1}^{k_j}\hat{\gamma}_j^2\frac{\sigma^2}{n}\right)^{\frac{1}{2}}\left[ \sum_{j=1}^{M_n}\sum_{l=k_{j-1}+1}^{k_j} \frac{ n}{\sigma^2}\left(\frac{\sigma^2}{n}-\frac{\hat{\sigma}^2}{n}\right)^2 \right]^{\frac{1}{2}} \right\} \\
& = \mathbb{E}\left\{\left(\sum_{j=1}^{M_n}\sum_{l=k_{j-1}+1}^{k_j}\hat{\gamma}_j^2\frac{\sigma^2}{n}\right)^{\frac{1}{2}}\left[\frac{k_{M_n}}{n\sigma^2}\left( \sigma^2-\hat{\sigma}^2 \right)^2\right]^{\frac{1}{2}} \right\}\\
     &\leq [\mathbb{E}R_n(\hat{\bw}|\mathcal{M},\bbf)]^{\frac{1}{2}}\left[\frac{k_{M_n}}{n\sigma^2}\mathbb{E}\left( \sigma^2-\hat{\sigma}^2 \right)^2\right]^{\frac{1}{2}},
\end{split}
\end{equation}
where the first inequality follows from the Cauchy-Schwarz inequality, and the second inequality follows from (\ref{eq:riskw}) and the Cauchy-Schwarz inequality.

\subsubsection{Upper bounding $\mathbb{E}\left|\sum_{j=1}^{M_n}\sum_{l=k_{j-1}+1}^{k_j}\gamma_j^*\left(\frac{\sigma^2}{n}-\frac{\hat{\sigma}^2}{n}\right)\right|$}

Similar to (\ref{eq:important5}), we have
\begin{equation}\label{eq:important6}
\begin{split}
   \mathbb{E}\left|\sum_{j=1}^{M_n}\sum_{l=k_{j-1}+1}^{k_j}\gamma_j^*\left(\frac{\sigma^2}{n}-\frac{\hat{\sigma}^2}{n}\right)\right| & = \mathbb{E}\left|\sum_{j=1}^{M_n}\sum_{l=k_{j-1}+1}^{k_j}\gamma_j^*\left(\frac{\sigma^2}{n} \right)^{\frac{1}{2}}\left(\frac{ n}{\sigma^2} \right)^{\frac{1}{2}}\left(\frac{\sigma^2}{n}-\frac{\hat{\sigma}^2}{n}\right)\right| \\
   & \leq \mathbb{E}\left\{ \left[\sum_{j=1}^{M_n}\sum_{l=k_{j-1}+1}^{k_j}\left(\gamma_j^*\right)^2\frac{\sigma^2}{n}\right]^{\frac{1}{2}} \left[ \sum_{j=1}^{M_n}\sum_{l=k_{j-1}+1}^{k_j} \frac{ n}{\sigma^2}\left(\frac{\sigma^2}{n}-\frac{\hat{\sigma}^2}{n}\right)^2 \right]^{\frac{1}{2}}\right\} \\
   & = \left[\sum_{j=1}^{M_n}\sum_{l=k_{j-1}+1}^{k_j}\left(\gamma_j^*\right)^2\frac{\sigma^2}{n}\right]^{\frac{1}{2}} \mathbb{E}\left[\frac{k_{M_n}}{n\sigma^2}\left( \sigma^2-\hat{\sigma}^2 \right)^2\right]^{\frac{1}{2}} \\
   &\leq [R_n(\bw^*|\mathcal{M},\bbf)]^{\frac{1}{2}}\left[\frac{k_{M_n}}{n\sigma^2}\mathbb{E}\left( \sigma^2-\hat{\sigma}^2 \right)^2\right]^{\frac{1}{2}},
\end{split}
\end{equation}
where the first inequality follows from the Cauchy-Schwarz inequality, and the second inequality follows from (\ref{eq:riskw}) and Jensen's inequality.

\subsubsection{Completing the proof of Theorem~\ref{theorem:main}}

Define
\begin{equation*}
  \psi(\mathcal{M}) = \left[M_n\wedge\left(1+ \sum_{j=1}^{M_n-1}\frac{k_{j+1} - k_{j}}{4k_{j}}+\sum_{j=1}^{M_n'-1} \frac{S_{j-1} - S_{j}}{4S_{j}}\right)\right](1+\log M_n)^2.
\end{equation*}
Substituting (\ref{eq:upper_bounding_term_1}), (\ref{eq:important213}), (\ref{eq:upper_bounding_term_3}), (\ref{eq:important214}), (\ref{eq:important5}), and (\ref{eq:important6}) into (\ref{eq:difference1}) and (\ref{eq:difference2}) yields
\begin{equation}\label{eq:inequality1}
\begin{split}
Q_n(\hat{\bw}|\mathcal{M},\bbf)&\leq R_n(\bw^*|\mathcal{M},\bbf) +\frac{C\sigma}{\sqrt{n}}[\mathbb{E}R_n(\hat{\bw}|\mathcal{M},\bbf)]^{\frac{1}{2}}\left[\psi(\mathcal{M})\right]^{\frac{1}{2}}\\
     &+\left[\frac{k_{M_n}}{n\sigma^2}\mathbb{E}\left( \sigma^2-\hat{\sigma}^2 \right)^2\right]^{\frac{1}{2}}\left[[\mathbb{E}R_n(\hat{\bw}|\mathcal{M},\bbf)]^{\frac{1}{2}}+[R_n(\bw^*|\mathcal{M},\bbf)]^{\frac{1}{2}}\right].
\end{split}
\end{equation}
In particular, when $Q_n(\hat{\bw}|\mathcal{M},\bbf)$ represents $\mathbb{E}R_n(\hat{\bw}|\mathcal{M},\bbf)$, (\ref{eq:inequality1}) also implies that
\begin{equation}\label{eq:inequality2}
\begin{split}
\mathbb{E}R_n(\hat{\bw}|\mathcal{M},\bbf)&\leq 2\left\{R_n(\bw^*|\mathcal{M},\bbf)+\left[\frac{k_{M_n}}{n\sigma^2}\mathbb{E}\left( \sigma^2-\hat{\sigma}^2 \right)^2\right]^{\frac{1}{2}}[R_n(\bw^*|\mathcal{M},\bbf)]^{\frac{1}{2}}\right\}\\
     &+\left\{\frac{C\sigma}{\sqrt{n}}\left[\psi(\mathcal{M})\right]^{\frac{1}{2}}+ \left[\frac{2k_{M_n}}{n\sigma^2}\mathbb{E}\left( \sigma^2-\hat{\sigma}^2 \right)^2\right]^{\frac{1}{2}}\right\}^2.
\end{split}
\end{equation}
Therefore, after inserting (\ref{eq:inequality2}) into the right side of (\ref{eq:inequality1}) and some additional algebra, we see that (\ref{eq:risk_bound_general}) holds.

\subsection{Proof of (\ref{eq:variance_risk})}\label{sec:proof:eq:variance_risk}

For completeness, we provide a brief proof for (\ref{eq:variance_risk}). We first decompose $\mathbb{E}(\hat{\sigma}_{m_n}^2-\sigma^2)^2$ as the variance term and the bias term
\begin{equation}\label{eq:variance_1}
  \mathbb{E}(\hat{\sigma}_{m_n}^2-\sigma^2)^2=\mathbb{E}(\hat{\sigma}_{m_n}^2-\mathbb{E}\hat{\sigma}_{m_n}^2)^2 + (\mathbb{E}\hat{\sigma}_{m_n}^2-\sigma^2)^2.
\end{equation}
Note that
\begin{equation}\label{eq:variance_2}
\begin{split}
   \hat{\sigma}_{m_n}^2 & =\frac{1}{n-m_n}\left\| \by - \hat{\bbf}_{m_n|\mathcal{M}_A} \right\|^2 \\
     & =\frac{n\| \btheta_{-m_n} \|^2}{n-m_n}+\frac{\beps^{\top}(\bI-\bP_{m_n|\mathcal{M}_A})\beps}{n-m_n}+\frac{2\bbf^{\top}(\bP_{p_n|\mathcal{M}_A}-\bP_{m_n|\mathcal{M}_A})\beps}{n-m_n},
\end{split}
\end{equation}
where $\btheta_{-m_n}=(\theta_{m_n+1},\ldots, \theta_{p_n})^{\top}$. Thus, the bias term of (\ref{eq:variance_1}) equals to
\begin{equation}\label{eq:variance_3}
  (\mathbb{E}\hat{\sigma}_{m_n}^2-\sigma^2)^2 = \left( \frac{n\| \btheta_{-m_n} \|^2}{n-m_n} + \sigma^2- \sigma^2  \right)^2 = \frac{n^2 \|\btheta_{-m_n} \|^4}{(n-m_n)^2}.
\end{equation}
We proceed to construct an upper bound for the variance term $\mathbb{E}(\hat{\sigma}_{m_n}^2-\mathbb{E}\hat{\sigma}_{m_n}^2)^2$. According to Theorem 1.1 of \cite{Rudelson2013}, we have
\begin{equation}\label{eq:HW_ineq_2}
\begin{split}
   &\mathbb{P}\left( \left|\frac{\beps^{\top}(\bI-\bP_{m_n|\mathcal{M}_A})\beps}{n-m_n} - \mathbb{E}\frac{\beps^{\top}(\bI-\bP_{m_n|\mathcal{M}_A})\beps}{n-m_n}\right|>x \right)\\
    &\leq 2\exp\left[ -c(n-m_n)(x \wedge x^2 ) \right].
\end{split}
\end{equation}
And due to the sub-Gaussian property of $\beps$, we have
\begin{equation}\label{eq:v_o_1}
  \mathbb{P}\left( \left|  \frac{2\bbf^{\top}(\bP_{p_n|\mathcal{M}_A}-\bP_{m_n|\mathcal{M}_A})\beps}{n-m_n} \right|>x  \right)\leq 2\exp\left[ -\frac{c(n-m_n)^2x^2}{n\| \btheta_{-m_n} \|^2} \right].
\end{equation}
Combining (\ref{eq:HW_ineq_2})--(\ref{eq:v_o_1}) with (\ref{eq:variance_2}) yields
\begin{equation*}\label{eq:v_o_2}
\begin{split}
    & \mathbb{P}\left( |\hat{\sigma}_{m_n}^2-\mathbb{E}\hat{\sigma}_{m_n}^2|>x \right)  \\
     & \leq 4\exp\left\{ -c\min\left[(n-m_n)x , \frac{(n-m_n)^2x^2}{(n-m_n)\vee(n\| \btheta_{-m_n} \|^2)} \right] \right\}.
\end{split}
\end{equation*}
By integrating the tail probability, we have
\begin{equation}\label{eq:variance_term}
\begin{split}
    \mathbb{E}(\hat{\sigma}_{m_n}^2-\mathbb{E}\hat{\sigma}_{m_n}^2)^2 & = \int_{0}^{\infty}\mathbb{P}\left( |\hat{\sigma}_{m_n}^2-\mathbb{E}\hat{\sigma}_{m_n}^2|>\sqrt{x}  \right)dx \\
    &\lesssim  \frac{1}{n-m_n} \vee \frac{n\| \btheta_{-m_n} \|^2}{(n-m_n)^2}.\\
\end{split}
\end{equation}
Combining (\ref{eq:variance_3}) with (\ref{eq:variance_term}) gives (\ref{eq:variance_risk}).

\subsection{Proof of Theorem~\ref{theorem:aop}}\label{sec:proof:theorem:aop}

In view of Theorem~\ref{theorem:main}, we see
\begin{equation}\label{eq:proof:aop:1}
\begin{split}
   \frac{Q_n(\hat{\bw}|\mathcal{M},\bbf)}{R_n(\bw^*|\mathcal{M},\bbf)}\leq 1&+\frac{C\sigma^2\psi(\mathcal{M})}{nR_n(\bw^*|\mathcal{M},\bbf)}+C\sigma\left[\frac{\psi(\mathcal{M})}{nR_n(\bw^*|\mathcal{M},\bbf)}\right]^{\frac{1}{2}}\\
   &+\frac{k_{M_n}\mathbb{E}\left( \hat{\sigma}^2-\sigma^2 \right)^2}{\sigma^2 n R_n(\bw^*|\mathcal{M},\bbf)}+ \left[ \frac{k_{M_n}\mathbb{E}\left( \hat{\sigma}^2-\sigma^2 \right)^2}{\sigma^2 n R_n(\bw^*|\mathcal{M},\bbf)} \right]^{\frac{1}{2}}.\\
\end{split}
\end{equation}
When Conditions~(\ref{eq:variance_rate})--(\ref{eq:minimum_marisk_rate}) hold, the risk bound (\ref{eq:proof:aop:1}) leads to
\begin{equation*}
  Q_n(\hat{\bw}|\mathcal{M},\bbf) \leq [1+o(1)]R_n(\bw^*|\mathcal{M},\bbf),
\end{equation*}
which proves the first part of Theorem~\ref{theorem:aop}.

Now, we prove that (\ref{eq:variance_rate_2}) is a sufficient condition for (\ref{eq:variance_rate}) when the variance estimator (\ref{eq:sigma_estimator_1}) with $m_n = \lfloor \kappa n \rfloor \wedge p_n$ ($0<\kappa<1$) is adopted. Based on (\ref{eq:variance_risk}), we have
\begin{equation}\label{eq:proof:aop:2}
\begin{split}
   \mathbb{E}(\hat{\sigma}_{m_n}^2-\sigma^2)^2 & \lesssim \frac{1}{n-m_n}\vee\frac{n\|\btheta_{-m_n} \|^2}{(n-m_n)^2} \vee  \frac{n^2 \|\btheta_{-m_n} \|^4}{(n-m_n)^2} \\
     & \asymp \frac{1}{n}\vee\frac{\|\btheta_{-m_n} \|^2}{n} \vee \|\btheta_{-m_n} \|^4\\
     & \asymp \frac{1}{n} \vee \|\btheta_{-m_n} \|^4,
\end{split}
\end{equation}
where the last approximation is due to $\|\btheta_{-m_n} \|^2 \leq \|\btheta \|^2 < C$. Then we have
\begin{equation*}
\begin{split}
   k_{M_n}\mathbb{E}\left( \hat{\sigma}^2_{m_n}-\sigma^2 \right)^2  & \leq n \mathbb{E}\left( \hat{\sigma}^2_{m_n}-\sigma^2 \right)^2  \lesssim n \left( \frac{1}{n} \vee \|\btheta_{-m_n} \|^4\right)\\
     & = o\left[nR_n(\bw^*|\mathcal{M}_A,\bbf)\right],
\end{split}
\end{equation*}
where the first inequality follows from $k_{M_n} \leq n$, the second step follows from (\ref{eq:proof:aop:2}), and the last step is due to the condition (\ref{eq:variance_rate_2}). Thus, the condition (\ref{eq:variance_rate}) is established.

To show (\ref{eq:minimum_marisk_rate_2}) is a sufficient condition for (\ref{eq:minimum_marisk_rate}) with $\mathcal{M}=\mathcal{M}_A$, we need to upper bound $\psi(\mathcal{M}_A)$. When $\mathcal{M}_A$ is used, we obtain $k_j = j$, and
\begin{equation*}
  \begin{split}
       1+ \sum_{j=1}^{M_n-1}\frac{k_{j+1} - k_{j}}{4k_{j}}& = \frac{5}{4} + \sum_{j=2}^{p_n-1}\frac{1}{4j}\leq \frac{5}{4}+ \int_{x=1}^{p_n-1}\frac{1}{4x}dx \leq C\log p_n.
  \end{split}
\end{equation*}
Define $p_n'=\min\{1 \leq j \leq p_n: \sum_{l=j+1}^{p_n}\theta_l^2=0 \}$. Then another term in (\ref{eq:psi_M}) is upper bounded by
\begin{equation*}
  1 + \sum_{j=1}^{M_n'-1} \frac{S_{j-1} - S_{j}}{4S_{j}} \leq C\sum_{j=1}^{p_n'-1} \frac{\theta_j^2}{\sum_{l=j+1}^{p_n}\theta_l^2}.
\end{equation*}
Thus, we see
\begin{equation*}
  \psi(\mathcal{M}_A) \precsim \left( \log p_n +  \sum_{j=1}^{p_n'-1} \frac{\theta_j^2}{\sum_{l=j+1}^{p_n}\theta_l^2} \right)(\log p_n)^2.
\end{equation*}

\subsection{Proof of Examples~1--2 in Section~\ref{subsec:aop}}\label{sec:proof:example:aop}

We first investigate the rate of $\|\btheta_{-m_n} \|^4$ in two examples, respectively. When $\theta_j=j^{-\alpha_1}$, $\alpha_1>1/2$, as per Theorem~1 and Example~1 in \cite{Peng2021}, we have
$$
R_n(\bw^*|\mathcal{M}_A,\bbf)\asymp R_n(m^*|\mathcal{M}_A,\bbf) \asymp \frac{m_n^*}{n}\asymp n^{-1+\frac{1}{2\alpha_1}}.
$$
Recalling the definition $\btheta_{-m_n}=(\theta_{m_n+1},\ldots, \theta_{n})^{\top}$, we see
\begin{equation*}
  \begin{split}
     \|\btheta_{-m_n} \|^4 = & \left(\sum_{j=m_n+1}^{n}\theta_j^2\right)^2 \asymp \left(\int_{\kappa n}^{n}x^{-2\alpha_1}dx\right)^2 \asymp n^{-4\alpha_1+2},
  \end{split}
\end{equation*}
and
\begin{equation*}
  \frac{n^{-4\alpha_1+2}}{m_n^*/n}\asymp \frac{n^{-4\alpha_1+2}}{n^{-1+1/(2\alpha_1)}} = n^{-\left(4\alpha_1+\frac{1}{2\alpha_1}-3\right)} \to 0,
\end{equation*}
where the last approximation is due to $4\alpha_1+\frac{1}{2\alpha_1}>3$ when $\alpha_1>1/2$. For the scenario $\theta_j=\exp(-j^{\alpha_2})$, $\alpha_2>0$, according to Theorem~1 and Example~2 in \cite{Peng2021}, we have
$$
R_n(\bw^*|\mathcal{M}_A,\bbf)\asymp R_n(m^*|\mathcal{M}_A,\bbf) \asymp \frac{m_n^*}{n} \asymp \frac{(\log n)^{\frac{1}{\alpha_2}}}{n}.
$$
Recall the definition $\mbox{Ga}(x;a)=\int_{t=x}^{\infty}t^{a-1}\exp(-t)dt$ for $x>0$. Similarly, we observe
\begin{equation*}
  \begin{split}
     \|\btheta_{-m_n} \|^4 & \asymp \left(\int_{\kappa n}^{n}\exp(-2x^{\alpha_2})dx\right)^2 \asymp \left(\int_{2(\kappa n)^{\alpha_2}}^{2n^{\alpha_2} } t^{\frac{1}{\alpha_2}-1}\exp(-t)dt\right)^2 \\
       & = \left[ \mbox{Ga}\left(2(\kappa n)^{\alpha_2};\frac{1}{\alpha_2}\right) - \mbox{Ga}\left(2n^{\alpha_2};\frac{1}{\alpha_2}\right) \right]^2 \leq \left[ \mbox{Ga}\left(2(\kappa n)^{\alpha_2};\frac{1}{\alpha_2}\right)  \right]^2\\
       & \sim \left[2(\kappa n)^{\alpha_2}\right]^{\frac{2}{\alpha_2}-2}\exp\left[ - 4(\kappa n)^{\alpha_2} \right] \asymp n^{2-2\alpha_2}\exp\left[ - 4(\kappa n)^{\alpha_2} \right]=o(1/n),
  \end{split}
\end{equation*}
where the fifth step follows from $\mbox{Ga}(x;a)/[x^{a-1}\exp(-x)] \to 1$ as $x \to \infty$. Hence, we have $\|\btheta_{-m_n} \|^4= o(m_n^*/n)$.

Next, we upper bound $\varphi(\btheta)$. In both examples, we have $p_n'=p_n$. Thus, we first construct a general upper bound for $1+\sum_{j=1}^{M_n'-1} (S_{j-1} - S_{j})/(4S_{j})$ when $M_n' = M_n$, which is also useful in the proofs of other examples in this paper. We have
\begin{equation}\label{eq:upper_S}
\begin{split}
& 1 + \sum_{j=1}^{M_n-1} \frac{S_{j-1} - S_{j}}{4S_{j}} \\
     & = 1 + \sum_{j=2}^{M_n-1}\left(\frac{S_{j-1} - S_{j}}{4S_{j-1}}\cdot \frac{S_{j-2} - S_{j-1}}{S_{j-1} - S_{j}}\right) + \frac{S_{M_n-2} - S_{M_n-1}}{4S_{M_n-1}}\\
     & \leq  1 + \max_{2 \leq j \leq M_n-1}\left\{ \frac{S_{j-2} - S_{j-1}}{S_{j-1} - S_{j}}\right\}\sum_{j=2}^{M_n-1}\left(\frac{S_{j-1} - S_{j}}{4S_{j-1}}\right) + \frac{S_{M_n-2} - S_{M_n-1}}{4S_{M_n-1}}\\
     & \leq 1 + C\max_{1 \leq j \leq M_n-2}\left\{ \frac{S_{j-1} - S_{j}}{S_{j} - S_{j+1}}\right\}\log \frac{1}{S_{M_n-1}}+ \frac{S_{M_n-2} - S_{M_n-1}}{4(S_{M_n-1} - S_{M_n})}\\
     & \leq C\max_{1 \leq j \leq M_n-1}\left\{ \frac{S_{j-1} - S_{j}}{S_{j} - S_{j+1}}\right\}\log \frac{1}{S_{M_n-1}},
\end{split}
\end{equation}
where the second inequality follows from
\begin{equation*}
\begin{split}
   \sum_{j=2}^{M_n-1}\frac{S_{j-1} - S_{j}}{4S_{j-1}} & \leq \int_{x= S_{M_n-1}}^{S_1} \frac{1}{4x}dx \leq \frac{1}{4} \left(\log S_{1} - \log S_{M_n-1} \right) \leq C \log \frac{1}{S_{M_n-1}}. \\
\end{split}
\end{equation*}
When the candidate model set $\mathcal{M}_A$ is adopted, (\ref{eq:upper_S}) is reduced to
\begin{equation*}
  \begin{split}
     \varphi(\btheta) \lesssim \max_{1\leq j\leq p_n-1}\left\{ \frac{\theta_j^2}{\theta_{j+1}^2} \right\} \log \frac{1}{\theta_{p_n}^2}.\\
  \end{split}
\end{equation*}
When $\theta_j=j^{-\alpha_1}$, $\alpha_1>1/2$, we see
\begin{equation*}
  \begin{split}
       \max_{1\leq j\leq n-1}\left\{ \frac{\theta_j^2}{\theta_{j+1}^2} \right\} \log \frac{1}{\theta_{n}^2} & =  2\alpha_1\max_{1\leq j\leq n-1}\left\{ \left( 1+\frac{1}{j} \right)^{2\alpha_1} \right\} \log n\leq C  \log n.
  \end{split}
\end{equation*}
Hence the condition (\ref{eq:minimum_marisk_rate_2}) is proved. Thus, the full AOP can be achieved by directly using $\mathcal{M}_A$ in Example~1.

For the scenario $p_n=n$ and $\theta_j=\exp(-j^{\alpha_2})$, $\alpha_2>0$, we show that the theory in this paper is not sufficient to establish the full AOP of $\mathcal{M}_A$. We first observe that
\begin{equation*}
  \begin{split}
       & \sum_{j=1}^{M_n}\frac{\left(  k_j^{\frac{1}{2}} -k_{j-1}^{\frac{1}{2}}\right)^2}{k_j-k_{j-1}}= 1+ \sum_{j=2}^{M_n}\left(\frac{  k_j^{\frac{1}{2}} -k_{j-1}^{\frac{1}{2}}}{k_j-k_{j-1}}\right)^2(k_j-k_{j-1}) \\
       & \geq  1+ \sum_{j=2}^{M_n}\frac{k_j-k_{j-1}}{4k_j} = 1 + \sum_{j=2}^{n}\frac{1}{4j}  \geq C \log n.
  \end{split}
\end{equation*}
Thus, we have $\psi(\mathcal{M}_A) \gtrsim (\log n)^3$, which implies that the full AOP is possible only for the case $0 < \alpha_2 < 1/3$. The following analysis shows that this is also impossible. Note that
\begin{equation*}
  \begin{split}
       & \sum_{j=1}^{M_n}\frac{\left(S_{j-1}^{\frac{1}{2}}-S_{j}^{\frac{1}{2}}\right)^2}{S_{j-1} - S_{j}} \geq 1 + \sum_{j=1}^{M_n-1}\frac{S_{j-1} - S_{j}}{4S_{j-1}} \\
       & = 1 + \frac{S_{0} - S_{1}}{4S_{0}}  + \sum_{j=2}^{M_n-1}\frac{S_{j-2} - S_{j-1}}{4S_{j-1}}\cdot \frac{S_{j-1} - S_{j}}{S_{j-2} - S_{j-1}}\\
       & \geq 1 + \frac{S_{0} - S_{1}}{4S_{0}} + \min_{1 \leq j \leq M_n-2}\left\{ \frac{S_{j} - S_{j+1}}{S_{j-1} - S_{j}} \right\}\sum_{j=1}^{M_n-2}\frac{S_{j-1} - S_{j}}{4S_{j}}\\
       & \geq 1 + \frac{S_{0} - S_{1}}{4S_{0}} + C\min_{1 \leq j \leq M_n-2}\left\{ \frac{S_{j} - S_{j+1}}{S_{j-1} - S_{j}} \right\}\log \frac{1}{S_{M_n-2}}.
  \end{split}
\end{equation*}
When $\mathcal{M}_A$ is adopted, we have
\begin{equation*}
  \begin{split}
       & \min_{1 \leq j \leq M_n-2}\left\{ \frac{S_{j} - S_{j+1}}{S_{j-1} - S_{j}} \right\}\log \frac{1}{S_{M_n-2}} = \min_{1 \leq j \leq n-2}\left\{ \frac{\theta_{j+1}^2}{\theta_j^2} \right\}\log \frac{1}{\theta_{n-1}^2+\theta_{n}^2}   \\
       & =\min_{1 \leq j \leq n-2}\left\{ \exp\left[ - 2(j+1)^{\alpha_2}+ 2j^{\alpha_2} \right] \right\} \log \frac{1}{\exp(-2(n-1)^{\alpha_2})+\exp(-2n^{\alpha_2})}\\
       & \geq C \log \frac{1}{\exp(-2(n-1)^{\alpha_2})+\exp(-2n^{\alpha_2})} \geq C \log \frac{1}{2\exp(-2(n-1)^{\alpha_2})}\\
       & =- C \log [2\exp(-2(n-1)^{\alpha_2})] \geq Cn^{\alpha_2}.
  \end{split}
\end{equation*}
Thus, the full AOP cannot be established due to $\psi(\mathcal{M}_A) \gtrsim n^{\alpha_2}$.

\subsection{Proof of Theorem~\ref{cor:grouped}}\label{sec:proof_grouped}

The proof of this theorem follows from the techniques in Lemma~3.11 of \cite{tsybakov2008introduction}. We first show that
\begin{equation}\label{eq:ppc}
  R_n(\bw^*|\mathcal{M}_G,\bbf)\leq (1+\zeta_n)R_n(\bw^*|\mathcal{M}_A,\bbf)+\frac{k_1\sigma^2}{n}.
\end{equation}
Define an $M_n$-dimensional weight vector $\bar{\bw}=(\bar{w}_1,\ldots,\bar{w}_{M_n})^{\top}$, where $\bar{w}_m=\sum_{j=k_{m-1}+1}^{k_m}w_j^*$, $\bar{\gamma}_m=\sum_{j=m}^{M_n}\bar{w}_m$, and $w_j^*$ is the $j$-th element of $\bw^*|\mathcal{M}_A$. According to (\ref{eq:riskw}), we have
\begin{equation}\label{eq:a38}
\begin{split}
   R_n(\bar{\bw}|\mathcal{M}_G,\bbf) & =\sum_{j=1}^{M_n}\left[ \left(1-\bar{\gamma}_j \right)^2\sum_{l=k_{j-1}+1}^{k_j}\theta_l^2+\frac{(k_j - k_{j-1})\sigma^2}{n}\bar{\gamma}_j^2 \right] \\
     &    \leq \sum_{j=1}^{p_n}(1-\gamma_j^*)^2\theta_j^2+\frac{\sigma^2}{n}\sum_{j=1}^{M_n}(k_j-k_{j-1})\bar{\gamma}_j^2,
\end{split}
\end{equation}
where the inequality follows the fact that $\bar{\gamma}_m\geq \gamma_j^*$ for any $k_{m-1}+1\leq j \leq k_m$. Note that
\begin{equation}\label{eq:a39}
\begin{split}
    \sum_{j=1}^{M_n}(k_j-k_{j-1})\bar{\gamma}_j^2& \leq k_1+(1+\zeta_n)\sum_{j=2}^{M_n}(k_{j-1}-k_{j-2})\bar{\gamma}_j^2\\
     & \leq k_1+(1+\zeta_n)\sum_{j=1}^{p_n}(\gamma_j^*)^2,
\end{split}
\end{equation}
where the second inequality is due to $\bar{\gamma}_m\leq \gamma_j^*$ when $k_{m-2}+1\leq j \leq k_{m-1}$.
Substituting (\ref{eq:a39}) into (\ref{eq:a38}), we obtain
\begin{equation*}
  \begin{split}
     R_n(\bw^*|\mathcal{M}_G,\bbf) & \leq  R_n(\bar{\bw}|\mathcal{M}_G,\bbf) \leq \sum_{j=1}^{p_n}(1-\gamma_j^*)^2\theta_j^2 + (1+\zeta_n)\sum_{j=1}^{p_n}\frac{\sigma^2}{n}(\gamma_j^*)^2 + \frac{k_1\sigma^2}{n} \\
       & \leq (1+\zeta_n)R_n(\bw^*|\mathcal{M}_A,\bbf)+\frac{k_1\sigma^2}{n}.
  \end{split}
\end{equation*}
Combining the risk bound in Theorem~\ref{theorem:main}, we see
\begin{equation}\label{eq:oracle_MG}
\begin{split}
   Q_n(\hat{\bw}|\mathcal{M}_G,\bbf)&\leq (1+\zeta_n)R_n(\bw^*|\mathcal{M}_A,\bbf)+\frac{k_1\sigma^2}{n}+\frac{C\sigma^2}{n}\psi(\mathcal{M}_G)\\
   &+\frac{C\sigma}{\sqrt{n}}[\psi(\mathcal{M}_G)]^{\frac{1}{2}}\left[(1+\zeta_n)R_n(\bw^*|\mathcal{M}_A,\bbf)+\frac{k_1\sigma^2}{n}\right]^{\frac{1}{2}}\\
   &+\frac{C p_n}{n\sigma^2}\mathbb{E}\left( \hat{\sigma}^2-\sigma^2 \right)^2+ \left[ \frac{C p_n}{n\sigma^2}\mathbb{E}\left( \hat{\sigma}^2-\sigma^2 \right)^2 \right]^{\frac{1}{2}}\left[(1+\zeta_n)R_n(\bw^*|\mathcal{M}_A,\bbf)+\frac{k_1\sigma^2}{n}\right]^{\frac{1}{2}}.\\
\end{split}
\end{equation}
Then provided $\zeta_n =o(1)$, $k_1=o[nR_n(\bw^*|\mathcal{M}_A,\bbf)]$, $\psi(\mathcal{M}_G)=o[nR_n(\bw^*|\mathcal{M}_A,\bbf)]$, and $p_n\mathbb{E}\left( \hat{\sigma}^2-\sigma^2 \right)^2=o[nR_n(\bw^*|\mathcal{M}_A,\bbf)]$, we have $Q_n(\hat{\bw}|\mathcal{M}_G,\bbf)=[1+o(1)] R_n(\bw^*|\mathcal{M}_A,\bbf)$.

\subsection{Proof of Examples~1--2 in Section~\ref{subsec:grouped}}

We first construct an upper bound on $M_n$ based on Lemma~3.12 of \cite{tsybakov2008introduction}. Suppose that $n$ is sufficiently large such that $\zeta_n<1$. Note the relation $\lfloor x \rfloor \geq x - 1 \geq x (1 - \zeta_n)$
when $x \geq \zeta_n^{-1}$. Thus, we have
$$
\lfloor k_1(1+\zeta_{n})^{j-1}\rfloor \geq k_1(1+\zeta_{n})^{j-1}(1 - \zeta_n)
$$
due to $k_1(1+\zeta_{n})^{j-1} \geq k_1 = \lceil \zeta_{n}^{-1} \rceil \geq \zeta_n^{-1}$. The definition of $M_n$ gives
\begin{equation*}
  \begin{split}
     n &\geq k_1 + \sum_{j=2}^{M_n-1}\lfloor k_1(1+\zeta_{n})^{j-1}\rfloor \geq k_1 + \sum_{j=2}^{M_n-1}k_1(1+\zeta_{n})^{j-1}(1 - \zeta_n) \\
       & = k_1 \left[ 1+ (1 - \zeta_n)\sum_{j=2}^{M_n-1}(1+\zeta_{n})^{j-1} \right]\\
       & = k_1 \left[ 1+ (1 - \zeta_n)\frac{(1+\zeta_n)[(1+\zeta_n)^{M_n-2}-1]}{\zeta_n} \right]\\
       & \geq \zeta_n^{-1} \left\{ 1 + \zeta_n^{-1}[(1+\zeta_n)^{M_n-2}-1] (1-\zeta_n^2) \right\}\\
       & \gtrsim \zeta_n^{-2}(1+\zeta_n)^{M_n-2}.
  \end{split}
\end{equation*}
Thus, we have
\begin{equation*}
  \begin{split}
     M_n & \lesssim \frac{\log (n\zeta_n^2)}{\log(1+\zeta_n)} \sim \frac{\log (n\zeta_n^2)}{\zeta_n} \lesssim \frac{\log n }{\zeta_n}  = (\log n)^{2}.
  \end{split}
\end{equation*}

It remains to evaluate $\psi(\mathcal{M}_{G1})$. Recalling the definition (\ref{eq:psi_M}), we obtain
\begin{equation}\label{eq:upper_psi_G1}
  \begin{split}
     \psi(\mathcal{M}_{G1})  & \lesssim M_n(1+\log M_n)^2\lesssim (\log n)^{2}(\log\log n)^2.
  \end{split}
\end{equation}
We conclude that the condition $\psi(\mathcal{M}_{G1})=o[nR_n(\bw^*|\mathcal{M}_A,\bbf)]$ is satisfied in Example~1 since $nR_n(\bw^*|\mathcal{M}_A,\bbf) \asymp m_n^* \asymp n^{1/(2\alpha_1)}$. In Example~2, when the coefficient decays as $\theta_j=\exp(-j^{\alpha_2})$, $0<\alpha_2< 1/2$, $\psi(\mathcal{M}_{G1})=o[nR_n(\bw^*|\mathcal{M}_A,\bbf)]$ also holds due to $nR_n(\bw^*|\mathcal{M}_A,\bbf) \asymp m_n^* \asymp (\log n)^{1/\alpha_2}$.

\subsection{Proof of Theorem~\ref{cor:minimum_1}}

Note that $R_n(\bw^*|\hat{\mathcal{M}}_{MS1},\bbf)$ is defined by directly plugging $\hat{\mathcal{M}}_{MS1}$ into the expression of $R_n(\bw^*|\mathcal{M},\bbf)$, i.e.,
\begin{equation*}
  R_n(\bw^*|\hat{\mathcal{M}}_{MS1},\bbf)=\frac{\hat{l}_n\sigma^2}{n}+\sum_{j=\hat{l}_n+1}^{\hat{u}_n}\frac{\theta_j^2\sigma^2}{n\theta_j^2+\sigma^2}+\sum_{j=\hat{u}_n+1}^{p_n}\theta_j^2.
\end{equation*}
Similarly, $Q_n(\hat{\bw}|\hat{\mathcal{M}}_{MS1},\bbf)$ is defined by plugging $\hat{\mathcal{M}}_{MS1}$ into $Q_n(\hat{\bw}|\mathcal{M},\bbf)$. Define the random variable $\Delta_{n1}=R_n(\bw^*|\hat{\mathcal{M}}_{MS1},\bbf)-R_n(\bw^*|\mathcal{M}_A,\bbf)$, which measures the risk increment of using the reduced candidate model set $\hat{\mathcal{M}}_{MS1}$. In view of the risk bound (\ref{eq:risk_bound_general}), it suffices to prove
\begin{equation}\label{eq:411}
  \frac{\mathbb{E}\Delta_{n1}}{R_n(\bw^*|\mathcal{M}_A,\bbf)}=\frac{\mathbb{E}(\Delta_{n1}1_{\bar{F}_n})+\mathbb{E}(\Delta_{n1}1_{F_n})}{R_n(\bw^*|\mathcal{M}_A,\bbf)} \to 0
\end{equation}
and
\begin{equation}\label{eq:condition_psi}
  \frac{\mathbb{E}\psi(\hat{\mathcal{M}}_{MS1})}{nR_n(\bw^*|\mathcal{M}_A,\bbf)}\to 0.
\end{equation}

The condition (\ref{eq:condition_psi}) is satisfied due to $nR_n(\bw^*|\mathcal{M}_A,\bbf)\asymp m_n^*$ and the assumption in Theorem~\ref{cor:minimum_1}. Then our main task is to prove (\ref{eq:411}). When $\mathbb{P}(\bar{F}_n)= o\left(m_n^*/n\right)$ holds, the first part on the right side of (\ref{eq:411}) is upper bounded by
\begin{equation}\label{eq:first_part_411}
  \frac{\mathbb{E}(\Delta_{n1}1_{\bar{F}_n})}{R_n(\bw^*|\mathcal{M}_A,\bbf)} \asymp \frac{\mathbb{E}(\Delta_{n1}1_{\bar{F}_n})}{m_n^*/n} \lesssim \frac{\mathbb{P}(\bar{F}_n)}{m_n^*/n}\to 0,
\end{equation}
where the first approximation is due to $R_n(\bw^*|\mathcal{M}_A,\bbf)\asymp m_n^*/n$, and the second step follows from
\begin{equation*}
\begin{split}
   \Delta_{n1} & \leq R_n(\bw^*|\hat{\mathcal{M}}_{MS1},\bbf)\leq \max_{\mathcal{M}\subseteq\{1,\ldots,p_n\}}R_n(\bw^*|\mathcal{M},\bbf) \\
     & \leq \max_{m\in\{1,\ldots,p_n\}}R_n(m|\mathcal{M}_A,\bbf)< C.
\end{split}
\end{equation*}

Now we turn to bound the second part of (\ref{eq:411}). From (\ref{eq:riskw}), we have
\begin{equation*}
  R_n(\bw^*|\mathcal{M}_A,\bbf)=\frac{\sigma^2}{n}+\sum_{j=2}^{p_n}\frac{\theta_j^2\sigma^2}{n\theta_j^2+\sigma^2}.
\end{equation*}

When $F_n$ holds, $\Delta_{n1}$ is upper bounded by
\begin{equation*}
\begin{split}
   \Delta_{n1}&=R_n(\bw^*|\hat{\mathcal{M}}_{MS1},\bbf)-R_n(\bw^*|\mathcal{M}_A,\bbf)=\sum_{j=2}^{\hat{l}_n}\left( \frac{\sigma^2}{n} - \frac{\sigma^2}{n+\frac{\sigma^2}{\theta_j^2}} \right) + \sum_{j=\hat{u}_n+1}^{p_n}\frac{\theta_j^2}{1+\frac{\sigma^2}{n\theta_j^2}}\\
    &\leq \frac{\hat{l}_n}{n}\sigma^2+\sum_{j=\hat{u}_n+1}^{p_n}\frac{\theta_j^2}{1+\frac{\theta_{m_n^*+1}^2}{\theta_j^2}}\leq \frac{\hat{l}_n}{n}\sigma^2+\sum_{j=m_n^*+1}^{p_n}\frac{\theta_j^2}{1+\frac{\theta_{m_n^*+1}^2}{\theta_{\hat{u}_n}^2}}\\
    &\leq \frac{c_2m_n^*}{n\kappa_l }+ \sum_{j=m_n^*+1}^{p_n}\frac{\theta_j^2}{1+\frac{\theta_{m_n^*+1}^2}{\theta_{\lfloor c_1m_n^*\kappa_u\rfloor}^2}},
\end{split}
\end{equation*}
where the first inequality follows from (\ref{eq:mnstar}), and the last step is due to the definitions of $\hat{l}_n$, $\hat{u}_n$, and the event $F_n$.
From this, we see that when $\kappa_l\to \infty$ and $\kappa_u \to \infty$,
\begin{equation*}
\begin{split}
   \mathbb{E}(\Delta_{n1}1_{F_n}) & \leq \frac{c_2m_n^*}{n\kappa_l }+ \sum_{j=m_n^*+1}^{p_n}\frac{\theta_j^2}{1+\frac{\theta_{m_n^*+1}^2}{\theta_{\lfloor c_1m_n^*\kappa_u\rfloor}^2}} =o\left(\frac{m_n^*}{n} \right)+o\left(\sum_{j=m_n^*+1}^{p_n}\theta_j^2  \right)\\
     &=o\left[R_n(m^*|\mathcal{M}_A,\bbf) \right]=o\left[R_n(\bw^*|\mathcal{M}_A,\bbf) \right],
\end{split}
\end{equation*}
where the first equality is due to $\kappa_l \to \infty$, $\kappa_u \to \infty$, and Assumption~\ref{asmp:regressor_order2}, the second equality follows from (\ref{eq:optimal_ms_risk}), and the last equality is due to \cite{Peng2021}. Thus, we have proved the theorem.

\subsection{Proof of Examples~1--2 in Section~\ref{subsec:minimum}}

We first verify the condition $\mathbb{P}(\bar{F}_n)= o\left(m_n^*/n\right)$ in Examples~1--2 when Mallows’ $C_p$ MS criterion \citep{Mallows1973} is adopted. Suppose $\sigma^2$ is known. From \cite{Kneip1994}, we see
\begin{equation}\label{eq:kneip}
\begin{split}
     & \mathbb{P}\left(\left| R_n(\hat{m}_n|\mathcal{M}_A,\bbf)-R_n(m^*|\mathcal{M}_A,\bbf) \right|> n^{-1}[x^2\vee x(m_n^*)^{1/2}] \right) \\
     & \leq C_1\exp(-C_2x)\quad \mbox{for}\,x\geq 0,
\end{split}
\end{equation}
where $C_1$ and $C_2$ are two constants that depend only on $\sigma^2$. Based on the technique in \cite{Peng2021}, it is easily to check that
\begin{equation}\label{eq:easy_check}
  \begin{split}
     \varpi_n & \triangleq[R_n(c_1m^*|\mathcal{M}_A,\bbf)-R_n(m^*|\mathcal{M}_A,\bbf)]\wedge[R_n(c_2m^*|\mathcal{M}_A,\bbf)-R_n(m^*|\mathcal{M}_A,\bbf)] \\
       & \gtrsim \frac{m_n^*}{n},
  \end{split}
\end{equation}
where $c_1$ and $c_2$ are two constants appearing in $F_n$. When $\bar{F}_n$ holds, we have $\hat{m}_n<\lfloor c_1m_n^*\rfloor$ or $\hat{m}_n>\lfloor c_2m_n^*\rfloor$. Thus, $\mathbb{P}\left(\bar{F}_n\right)$ is upper bounded by
\begin{equation}\label{eq:relate}
\begin{split}
   \mathbb{P}\left(\bar{F}_n\right)& \leq \mathbb{P}\left(\left|R_n(\hat{m}_n|\mathcal{M}_A,\bbf)-R_n(m^*|\mathcal{M}_A,\bbf)\right|>\varpi_n\right)\lesssim \exp{\left[-C(m_n^*)^{\frac{1}{2}}\right]},
\end{split}
\end{equation}
where $C$ is a fixed constant, the first inequality follows from the definition of $\varpi_n$, and the last inequality follows from (\ref{eq:kneip}) and (\ref{eq:easy_check}). It suffices to show
\begin{equation}\label{eq:suff_psi}
  \exp{\left[-C(m_n^*)^{\frac{1}{2}}\right]}=o\left(\frac{m_n^*}{n}\right)
\end{equation}
in these two examples. In Example~1, the coefficient $\theta_j=j^{-\alpha_1}$, $1/2<\alpha_1<\infty$, and $m_n^* \asymp n^{1/(2\alpha_1)}$. In this case, (\ref{eq:suff_psi}) is satisfied for any fixed $\alpha_1$. In Example~2, the coefficient decays as $\theta_j=\exp(-j^{\alpha_2})$, and $m_n^* \asymp (\log n)^{1/\alpha_2}$. When $0<\alpha_2 <1/2$, we have
\begin{equation*}
  \begin{split}
     \exp{\left[-C(m_n^*)^{\frac{1}{2}}\right]} & = \exp{\left[-C(\log n)^{\frac{1}{2\alpha_2}}\right]} = o \left[ \exp(-\log n) \right]\\
       & =o\left( \frac{1}{n} \right) = o\left( \frac{m_n^*}{n} \right),
  \end{split}
\end{equation*}
where the second equality holds for any constant $C$.

Next, we investigate the condition $\mathbb{E}\psi(\hat{\mathcal{M}}_{MS1}) = o(m_n^*)$ in Examples~1--2. Based on (\ref{eq:psi_M}) and (\ref{eq:upper_S}), we have
\begin{equation}\label{eq:E_psi}
\begin{split}
    \mathbb{E}\psi(\hat{\mathcal{M}}_{MS1})& \lesssim \mathbb{E}\left\{\left[\log(\kappa_l \kappa_u) + \max_{\hat{l}_n\leq j \leq \hat{u}_n -1}\left\{\frac{\theta_j^2}{\theta_{j+1}^2} \right\}\log \frac{1}{\theta_{\hat{u}_n}^2}\right] \left[\log\left((\kappa_u - \kappa_l^{-1})\hat{m}_n\right)\right]^2\right\}\\
    & = \log(\kappa_l \kappa_u)\mathbb{E}\left[\log\left((\kappa_u - \kappa_l^{-1})\hat{m}_n\right)\right]^2 \\
    & + \mathbb{E}\left\{\left[\max_{\hat{l}_n\leq j \leq \hat{u}_n -1}\left\{\frac{\theta_j^2}{\theta_{j+1}^2} \right\}\log \frac{1}{\theta_{\hat{u}_n}^2}\right] \left[\log\left((\kappa_u - \kappa_l^{-1})\hat{m}_n\right)\right]^2\right\}.\\
    \end{split}
\end{equation}
The first term on the right side of (\ref{eq:E_psi}) is upper bounded by
\begin{equation}\label{eq:ms_proof_1}
  \begin{split}
     \log(\kappa_l \kappa_u)\mathbb{E}\left[\log\left((\kappa_u - \kappa_l^{-1})\hat{m}_n\right)\right]^2 & \leq \log(\kappa_l \kappa_u)\left[\log\left((\kappa_u - \kappa_l^{-1})\mathbb{E}\hat{m}_n\right)\right]^2\\
     & \lesssim (\log \kappa_l + \log \kappa_u) \left[ \log(\kappa_u - \kappa_l^{-1}) + \log m_n^*  \right]^2\\
     & \asymp (\log\log n)(\log\log n + \log m_n^*)^2,
  \end{split}
\end{equation}
where the first inequality follows from Jensen's inequality, the second inequality is due to
$\mathbb{E}\hat{m}_n=\mathbb{E}(\hat{m}_n1_{F_n})+\mathbb{E}(\hat{m}_n1_{\bar{F}_n})\lesssim c_2m_n^*+n\cdot\frac{m_n^*}{n}\lesssim m_n^*$, and the last approximation in (\ref{eq:ms_proof_1}) follows from $\kappa_l=\kappa_u=\log n$. The second term in (\ref{eq:E_psi}) is upper bounded by
\begin{equation}\label{eq:ms_proof_2}
  \begin{split}
       & \mathbb{E}\left\{\left[\max_{\hat{l}_n\leq j \leq \hat{u}_n -1}\left\{\frac{\theta_j^2}{\theta_{j+1}^2} \right\}\log \frac{1}{\theta_{\hat{u}_n}^2}\right] \left[\log\left((\kappa_u - \kappa_l^{-1})\hat{m}_n\right)\right]^2\right\} \\
       & \leq \left\{\mathbb{E}\left[\max_{\hat{l}_n\leq j \leq \hat{u}_n -1}\left\{\frac{\theta_j^4}{\theta_{j+1}^4} \right\}\left(\log \frac{1}{\theta_{\hat{u}_n}^2}\right)^2\right]\right\}^{\frac{1}{2}}\left\{\mathbb{E}\left[\log\left((\kappa_u - \kappa_l^{-1})\hat{m}_n\right)\right]^4 \right\}^{\frac{1}{2}}\\
       & \leq \left\{\mathbb{E}\left[\max_{\hat{l}_n\leq j \leq \hat{u}_n -1}\left\{\frac{\theta_j^4}{\theta_{j+1}^4} \right\}\left(\log \frac{1}{\theta_{\hat{u}_n}}\right)^2\right]\right\}^{\frac{1}{2}}(\log\log n + \log m_n^*)^2,
  \end{split}
\end{equation}
where the first inequality follows from the Cauchy-Schwarz inequality, and the second inequality is due to Jensen's inequality and the same technique in (\ref{eq:ms_proof_1}).

When $\theta_j = j^{-\alpha_1}$, $1/2 < \alpha_1 < \infty$, we have
\begin{equation*}
  \begin{split}
       & \max_{\hat{l}_n\leq j \leq \hat{u}_n -1}\left\{\frac{\theta_j^4}{\theta_{j+1}^4} \right\} = \max_{\hat{l}_n\leq j \leq \hat{u}_n -1} \left\{ \left(\frac{j+1}{j}\right)^{4\alpha_1} \right\} < C.
  \end{split}
\end{equation*}
Thus, the term in right side of (\ref{eq:ms_proof_2}) is upper bounded by
\begin{equation*}
\begin{split}
   \left\{\mathbb{E}\left[\max_{\hat{l}_n\leq j \leq \hat{u}_n -1}\left\{\frac{\theta_j^4}{\theta_{j+1}^4} \right\}\left(\log \frac{1}{\theta_{\hat{u}_n}}\right)^2\right]\right\}^{\frac{1}{2}} & \leq C\left[\mathbb{E}\left(\log\frac{1}{\theta_{\hat{u}_n}}\right)^2\right]^{\frac{1}{2}} \\
     & \lesssim \left[\mathbb{E} \left(\log \hat{u}_n\right)^2\right]^{\frac{1}{2}}\leq \log \left(\mathbb{E}\hat{u}_n\right)\\
     & \lesssim \log\log n + \log m_n^*,
\end{split}
\end{equation*}
where the third step follows from Jensen's inequality, and the last step is due to $\kappa_u=\log n$ and $\mathbb{E}\hat{m}_n\lesssim m_n^*$. Combining all the above bounds, we have
\begin{equation*}
  \mathbb{E}\psi(\hat{\mathcal{M}}_{MS1}) \lesssim (\log\log n + \log m_n^*)^3 = o(m_n^*),
\end{equation*}
where the last equality follows from $m_n^*\asymp n^{1/(2\alpha_1)}$ in Example~1.

Consider Example~2: $\theta_j = \exp(-j^{\alpha_2})$ with $0 < \alpha_2 <1/2$. We have
\begin{equation*}
  \begin{split}
       & \left\{\mathbb{E}\left[\max_{\hat{l}_n\leq j \leq \hat{u}_n -1}\left\{\frac{\theta_j^4}{\theta_{j+1}^4} \right\}\left(\log\frac{1}{\theta_{\hat{u}_n}}\right)^2\right]\right\}^{\frac{1}{2}} \\
       & \lesssim \left[\mathbb{E}\left(\log\frac{1}{\theta_{\hat{u}_n}}\right)^2\right]^{\frac{1}{2}}=\left(\mathbb{E}\hat{u}_n^{2\alpha_2}\right)^{\frac{1}{2}}\\
       & \leq \left(\mathbb{E} \hat{u}_n\right)^{\alpha_2} \leq (\log n)^{\alpha_2} (m_n^*)^{\alpha_2} \asymp (\log n)^{\alpha_2+1},\\
  \end{split}
\end{equation*}
where the first inequality follows form $\max_{\hat{l}_n\leq j \leq \hat{u}_n -1}\left\{\exp\left[ 4(j+1)^{\alpha_2} - 4j^{\alpha_2} \right] \right\}<C$, the second inequality follows from Jensen's inequality, and the last equality is due to $m_n^* \asymp (\log n)^{1/\alpha_2}$.
Combining all the above results, we have
\begin{equation*}
\begin{split}
   \mathbb{E}\psi(\hat{\mathcal{M}}_{MS1}) & \lesssim \left[\log\log n+ (\log n)^{\alpha_2+1}\right](\log\log n + \log m_n^*)^2 \\
     & \lesssim (\log n)^{\alpha_2+1} (\log\log n)^2\\
     & = o(m_n^*),
\end{split}
\end{equation*}
where the last equality is due to $\alpha_2+1 < 1/\alpha_2$ when $0<\alpha_2<1/2$.

\section{Proofs of the results under the non-nested setup}\label{sec:a:proof_non_nested}

\subsection{An equivalent sequence model}\label{sec:preliminaries_non}

Recall that $\{\bphi_1,\ldots,\bphi_{p_n} \}$ is an orthonormal basis of the column space of $\bX$. Similar to the results in Section~\ref{sec:a:proof:s3:preliminaries}, the regression model (\ref{eq:model_matrix}) can be rewritten as an equivalent sequence model
\begin{equation}\label{eq:sequence2}
  \hat{\theta}_{j} = \theta_j + e_j, \quad j=1,\ldots,p_n,
\end{equation}
where $\hat{\theta}_{j} =\bphi_j^{\top}\by/\sqrt{n}$, $\theta_j=\bphi_j^{\top}\bbf/\sqrt{n}$, and $e_j=\bphi_j^{\top}\beps/\sqrt{n}$.

Based on the sequence model (\ref{eq:sequence2}), we present the equivalent forms for the MS and MA estimators based on $\mathcal{M}_{AS}=\{ \mathcal{I}: \mathcal{I} \subseteq \{1,\ldots,p_n \} \}$. Define $\mathcal{I}_m$ the index set of the regressors in the $m$-th non-nested candidate model. Then, the MS estimator $\hat{\bbf}_{m | \mathcal{M}_{AS}}$ equals to
\begin{equation*}
  \begin{split}
     \hat{\bbf}_{m | \mathcal{M}_{AS}} & = \sum_{j\in \mathcal{I}_m}\bphi_j\bphi_j^{\top}\by=\sqrt{n}\sum_{j\in \mathcal{I}_m}\bphi_j\hat{\theta}_{j}.\\
  \end{split}
\end{equation*}
The $\ell_2$ risk of $\hat{\bbf}_{m | \mathcal{M}_{AS}}$ is
\begin{equation}\label{eq:MS_risk_non}
\begin{split}
   R_n(m|\mathcal{M}_{AS},\bbf) & =\frac{1}{n}\mathbb{E}\left\|\hat{\bbf}_{m | \mathcal{M}_{AS}}-\bbf \right\|^2=\mathbb{E} \left\|  \sum_{j\in \mathcal{I}_m}\bphi_j\hat{\theta}_{j} - \sum_{j=1}^{p_n}\bphi_j\theta_{j} \right\|^2\\
    &=\mathbb{E}\left\|\sum_{j\in \mathcal{I}_m}\bphi_je_j-\sum_{j\in \{1,\ldots,p_n \}\setminus \mathcal{I}_m}\bphi_j\theta_j \right\|^2\\
    &=\frac{|\mathcal{I}_m|\sigma^2}{n}+\sum_{j\in \{1,\ldots,p_n \}\setminus \mathcal{I}_m}\theta_j^2.
\end{split}
\end{equation}

Define $\lambda_j \triangleq \sum_{m: j\in \mathcal{I}_m}w_m$ for $1\leq j \leq p_n$. Then, the non-nested MA estimator can be represented by
\begin{equation*}
  \hat{\bbf}_{\bw|\mathcal{M}_{AS}} = \sqrt{n}\sum_{m=1}^{2^{p_n}}w_m \sum_{j \in \mathcal{I}_m}\hat{\theta}_j \bphi_j = \sqrt{n} \sum_{j=1}^{p_n} \left(\sum_{m: j\in \mathcal{I}_m}w_m\right) \hat{\theta}_j \bphi_j = \sqrt{n} \sum_{j=1}^{p_n} \lambda_j \hat{\theta}_j \bphi_j.
\end{equation*}
The risk function of the non-nested MA estimator is
\begin{equation}\label{eq:MA_risk_non}
  \begin{split}
       R_n(\bw|\mathcal{M}_{AS},\bbf)& = n^{-1}\mathbb{E}\left\|\hat{\bbf}_{\bw|\mathcal{M}_{AS}}-\bbf \right\|^2 = n^{-1}\mathbb{E} \left\| \sqrt{n} \sum_{j=1}^{p_n} \lambda_j \hat{\theta}_j \bphi_j - \sqrt{n}\sum_{j=1}^{p_n}\theta_j\bphi_j \right\|^2 \\
       & = \mathbb{E} \left\|  \sum_{j=1}^{p_n} \left(\lambda_j \hat{\theta}_j - \theta_j\right) \bphi_j \right\|^2 = \sum_{j=1}^{p_n} \mathbb{E}\left(\lambda_j \hat{\theta}_j - \theta_j\right)^2\\
        &= \sum_{j=1}^{p_n}\left[(1-\lambda_j)^2\theta_j^2+\frac{\sigma^2\lambda_j^2}{n}\right].
  \end{split}
\end{equation}

%
%
%
%
%
%
%
%

\subsection{Proof of Theorem~\ref{theo:improvability}}\label{sec:proof_impro}

We begin by evaluating the ideal MS and MA risks. Consider a nested candidate model set $\mathcal{M}_{N}=\{\mathcal{S}_m: m =1,\ldots,p_n \}$ based the order of $|\theta_j|, j=1,\ldots,p_n$, where
$$
\mathcal{S}_m = \operatorname{argmax}_{\mathcal{I}\subseteq \{1,\ldots,p_n \}}\sum_{\substack{|\mathcal{I}|=m\\j \in \mathcal{I}}}|\theta|_{(j)}^2.
$$
Thus, $\mathcal{S}_m$ contains the indices of $m$ regressors with the strongest magnitude $|\theta_j|$.

From (\ref{eq:MS_risk_non}), the ideal MS risk is lower bounded by
\begin{equation*}
  \begin{split}
     R_n(m^*|\mathcal{M}_{AS},\bbf) & = \min_{m \in \{1,\ldots,2^{p_n} \}}\left\{\frac{|\mathcal{I}_m|\sigma^2}{n}+\sum_{j\in \{1,\ldots,p_n \}\setminus \mathcal{I}_m}\theta_j^2\right\}\\
       & = \min_{1\leq k \leq p_n}\left\{\frac{k\sigma^2}{n}+\sum_{\substack{|\mathcal{I}_m|=k\\j\in \{1,\ldots,p_n \}\setminus \mathcal{I}_m}}\theta_j^2\right\}\\
       & \geq \min_{1\leq k \leq p_n}\left\{\frac{k\sigma^2}{n}+\sum_{j=k+1}^{p_n}|\theta|_{(j)}^2\right\}.
  \end{split}
\end{equation*}
Since $\mathcal{M}_{N}\subset \mathcal{M}_{AS}$, the ideal MS risk is upper bounded by
\begin{equation*}
  R_n(m^*|\mathcal{M}_{AS},\bbf) \leq R_n(m^*|\mathcal{M}_{N},\bbf) = \min_{1\leq k \leq p_n}\left\{\frac{k\sigma^2}{n}+\sum_{j=k+1}^{p_n}|\theta|_{(j)}^2\right\},
\end{equation*}
where the last equality follows from \cite{Peng2021}. Thus, we have
\begin{equation}\label{eq:non_MS_equal}
  R_n(m^*|\mathcal{M}_{AS},\bbf) = R_n(m^*|\mathcal{M}_{N},\bbf)
\end{equation}

From (\ref{eq:MA_risk_non}), the ideal MA risk can be lower bounded by
\begin{equation*}\label{eq:ideal_lower}
  \begin{split}
     R_n(\bw^*|\mathcal{M}_{AS},\bbf) & \geq \sum_{j=1}^{p_n}\min_{0\leq \lambda_j \leq 1}\left[(1-\lambda_j)^2\theta_j^2+\frac{\sigma^2\lambda_j^2}{n}\right]  \\
       & = \sum_{j=1}^{p_n}\frac{\theta_j^2\sigma^2}{n\theta_j^2+\sigma^2}\\
       & = \sum_{j=1}^{p_n}\frac{|\theta|_{(j)}^2\sigma^2}{n|\theta|_{(j)}^2+\sigma^2}.
  \end{split}
\end{equation*}
The ideal MA risk is upper bounded by
\begin{equation*}
  \begin{split}
      R_n(\bw^*|\mathcal{M}_{AS},\bbf) & \leq   R_n(\bw^*|\mathcal{M}_N,\bbf)=\frac{\sigma^2}{n}+\sum_{j=2}^{p_n}\frac{|\theta|_{(j)}^2\sigma^2}{n|\theta|_{(j)}^2+\sigma^2},
  \end{split}
\end{equation*}
where the equality follows from \cite{Peng2021}. Since
\begin{equation*}
  \begin{split}
       \frac{\sigma^2}{n}+\sum_{j=2}^{p_n}\frac{|\theta|_{(j)}^2\sigma^2}{n|\theta|_{(j)}^2+\sigma^2} - \sum_{j=1}^{p_n}\frac{|\theta|_{(j)}^2\sigma^2}{n|\theta|_{(j)}^2+\sigma^2}& = \frac{\sigma^2}{n} - \frac{|\theta|_{(1)}^2\sigma^2}{n|\theta|_{(1)}^2+\sigma^2}\\
       & = \frac{\sigma^4}{n(n|\theta|_{(1)}^2+\sigma^2)} = O\left( \frac{1}{n^2} \right),
  \end{split}
\end{equation*}
and $R_n(\bw^*|\mathcal{M}_N,\bbf) \gtrsim 1/n$, we have
\begin{equation}\label{eq:non_MA_equal}
  R_n(\bw^*|\mathcal{M}_{AS},\bbf) \sim R_n(\bw^*|\mathcal{M}_N,\bbf).
\end{equation}
Combining (\ref{eq:non_MS_equal})--(\ref{eq:non_MA_equal}) with Theorem~2 of \cite{Peng2021}, we obtain the desired results.

\subsection{Proof of Theorem~\ref{theo:suff}}

The proof of this theorem follows from the techniques in Section~\ref{sec:proof_grouped}. We consider a group-wise nested candidate model set $\mathcal{M}_G=\{ \{1,2,\ldots,d_l \}: 1\leq d_1<d_2<\cdots<d_{D_n}= p_n\}$, where $d_l,l=1,\ldots,D_n$ are the parameters given in Assumption~\ref{asmp:partial_order}. We first show that
\begin{equation}\label{eq:ppc_2}
  R_n(\bw^*|\mathcal{M}_G,\bbf)\leq (1+z_n)R_n(\bw^*|\mathcal{M}_N,\bbf)+\frac{d_1\sigma^2}{n},
\end{equation}
where $\mathcal{M}_N$ is the nested candidate model set defined in Section~\ref{sec:proof_impro}.
Define a $D_n$-dimensional weight vector $\bar{\bw}=(\bar{w}_1,\ldots,\bar{w}_{D_n})^{\top}$, where $\bar{w}_l=\sum_{m=d_{l-1}+1}^{d_l}w_m^*$, $\bar{\gamma}_l=\sum_{k=l}^{D_n}\bar{w}_k$, $\gamma_j^*=\sum_{m=j}^{p_n}w^*_m$, and $w_m^*$ is the $m$-th element of $\bw^*|\mathcal{M}_N$. Thus, $R_n(\bw^*|\mathcal{M}_G,\bbf)$ is upper bounded by
\begin{equation}\label{eq:a38_2}
\begin{split}
   R_n(\bw^*|\mathcal{M}_G,\bbf) &\leq R_n(\bar{\bw}|\mathcal{M}_G,\bbf)\\
    &=\sum_{l=1}^{D_n}\left[ \left(1-\bar{\gamma}_l \right)^2\sum_{j=d_{l-1}+1}^{d_l}\theta_j^2+\frac{(d_l - d_{l-1})\sigma^2}{n}\bar{\gamma}_l^2 \right] \\
    & =\sum_{l=1}^{D_n}\left[ \left(1-\bar{\gamma}_l \right)^2\sum_{j=d_{l-1}+1}^{d_l}|\theta|_{(j)}^2+\frac{(d_l - d_{l-1})\sigma^2}{n}\bar{\gamma}_l^2 \right] \\
     &    \leq \sum_{j=1}^{p_n}(1-\gamma_j^*)^2|\theta|_{(j)}^2+\frac{\sigma^2}{n}\sum_{l=1}^{D_n}(d_l-d_{l-1})\bar{\gamma}_l^2,
\end{split}
\end{equation}
where the first equality follows from (\ref{eq:riskw}), the second equality follows from Assumption~\ref{asmp:partial_order}, and the last inequality is due to $\bar{\gamma}_l\geq \gamma_j^*$ for any $d_{l-1}+1\leq j \leq d_l$. Note that
\begin{equation}\label{eq:a39_2}
\begin{split}
    \sum_{l=1}^{D_n}(d_l-d_{l-1})\bar{\gamma}_l^2& \leq d_1+(1+z_n)\sum_{l=2}^{D_n}(d_{l-1}-d_{l-2})\bar{\gamma}_l^2\\
     & \leq d_1+(1+z_n)\sum_{j=1}^{p_n}(\gamma_j^*)^2,
\end{split}
\end{equation}
where the second inequality is due to $\bar{\gamma}_l\leq \gamma_j^*$ when $d_{l-2}+1\leq j \leq d_{l-1}$.
Substituting (\ref{eq:a39_2}) into (\ref{eq:a38_2}), we obtain
\begin{equation*}
  \begin{split}
     R_n(\bw^*|\mathcal{M}_G,\bbf) & \leq \sum_{j=1}^{p_n}(1-\gamma_j^*)^2|\theta|_{(j)}^2 + (1+z_n)\sum_{j=1}^{p_n}\frac{\sigma^2}{n}(\gamma_j^*)^2 + \frac{d_1\sigma^2}{n} \\
       & \leq (1+z_n)R_n(\bw^*|\mathcal{M}_N,\bbf)+\frac{d_1\sigma^2}{n},
  \end{split}
\end{equation*}
which proves (\ref{eq:ppc_2}). Since $\mathcal{M}_G \subset \mathcal{M}_{A} \subset \mathcal{M}_{AS}$, we see
$$
R_n(\bw^*|\mathcal{M}_{AS},\bbf)\leq R_n(\bw^*|\mathcal{M}_{A},\bbf) \leq R_n(\bw^*|\mathcal{M}_G,\bbf) \leq (1+z_n)R_n(\bw^*|\mathcal{M}_N,\bbf)+\frac{d_1\sigma^2}{n}.
$$
Combining (\ref{eq:non_MA_equal}), we see
\begin{equation}\label{eq:a65}
  R_n(\bw^*|\mathcal{M}_{A},\bbf) =\left[1+o(1) \right] R_n\left(\bw^*|\mathcal{M}_{AS},\bbf\right)
\end{equation}
when $z_n \to 0$ and $d_1 = o\left[nR_n(\bw^*|\mathcal{M}_{AS},\bbf)\right]$.

The second part of the theorem can be proved by applying Theorems~\ref{theorem:aop}--\ref{cor:grouped} to (\ref{eq:a65}).

\subsection{Proof of Example~3}

To prove the results in this example, we just need to upper bound $\psi(\mathcal{M}_A)$ and $\psi(\mathcal{M}_{G1})$. The upper bound of $\psi(\mathcal{M}_{G1})$ follows from (\ref{eq:upper_psi_G1}). From (\ref{eq:upper_S}), we have
\begin{equation}\label{eq:psi_MA_upper}
  \psi(\mathcal{M}_A) \lesssim \left( \log p_n +  \max_{1\leq j\leq p_n-1}\left\{ \frac{\theta_j^2}{\theta_{j+1}^2} \right\} \log \frac{1}{\theta_{p_n}^2} \right)\left(\log p_n\right)^2,
\end{equation}
Since $(|\theta_{d_{l-1}+1}|,\ldots,|\theta_{d_{l}}|)^{\top} \in \mathfrak{S}_{(|\theta|_{(d_{l-1}+1)},\ldots,|\theta|_{(d_{l})})^{\top}}$ for $1\leq l \leq D_n$, we see
\begin{equation}\label{eq:upper_exam_3}
  \begin{split}
     \max_{1 \leq j \leq p_n - 1}\left\{ \frac{\theta_j^2}{\theta_{j+1}^2} \right\} & \leq \max_{1 \leq l \leq D_n-1}\left\{\frac{\max_{d_{l-1}+1 \leq j \leq d_l}|\theta|_{(j)}^2}{\min_{d_{l}+1 \leq j \leq d_{l+1}}|\theta|_{(j)}^2}\right\} \\
     & = \max_{1 \leq l \leq D_n-1}\left\{\frac{|\theta|_{(d_{l-1}+1)}^2}{|\theta|_{(d_{l+1})}^2}\right\} = \max_{1 \leq l \leq D_n-1}\left\{ \left(\frac{d_{l+1}}{d_{l-1}+1}\right)^{2\alpha_1} \right\}\\
     & \leq d_2^{2\alpha_1}  \vee \max_{2 \leq l \leq D_n-1}\left\{ \left(\frac{d_{l+1}}{d_{l-1}}\right)^{2\alpha_1} \right\}.
  \end{split}
\end{equation}
Note that $d_2^{2\alpha_1} \lesssim (\log n)^{2\alpha_1}$. In addition, when $2\leq l\leq D_n-1$, we have
\begin{equation*}
  \begin{split}
       \left(\frac{d_{l+1}}{d_{l-1}}\right)^{2\alpha_1} &\leq \left[\frac{d_{l-1}+d_l-d_{l-1} + \left(1+\frac{1}{\log n} \right)(d_l-d_{l-1})}{d_{l-1}}\right]^{2\alpha_1}\\
       & = \left[1+ \left(2+\frac{1}{\log n} \right)\left( \frac{d_{l}}{d_{l-1}} -1 \right)\right]^{2\alpha_1}< C,
  \end{split}
\end{equation*}
where the last inequality follows from
\begin{equation*}
  \frac{d_{l}}{d_{l-1}} - 1 \leq \left( 1+ \frac{1}{\log n} \right)\left(1 - \frac{d_{l-2}}{d_{l-1}} \right) < C.
\end{equation*}
Thus, we conclude
\begin{equation*}
  \psi(\mathcal{M}_A) \lesssim (\log n)^{3+2\alpha_1}.
\end{equation*}

\subsection{Proof of Theorem~\ref{prop:nece}}

For simplicity of notation, let $\Theta$ denote $\mathfrak{S}_{(|\theta|_{(1)},\ldots,|\theta|_{(n)})^{\top}}$, where $|\theta|_{(j)}= j^{-\alpha_1},j=1,\ldots,n$. Define the corresponding space of regression mean vector $\mathcal{F}=\{\bbf=\sqrt{n}\sum_{j=1}^{n}\theta_j\bphi_j:\btheta \in \Theta\}$. In order to prove Theorem~\ref{prop:nece}, we just need to prove the following minimax argument:
\begin{equation}\label{eq:minimax_equal}
  \lim_{n\to \infty} \inf_{\hat{\bbf}}\sup_{\bbf \in \mathcal{F}}\frac{R_n(\hat{\bbf},\bbf)}{R_n(\bw^*|\mathcal{M}_{AS},\bbf)} \to \infty,
\end{equation}
where $\hat{\bbf}$ is any measurable estimator based on the observation $\by$. Note that the term in (\ref{eq:minimax_equal}) is lower bounded by
\begin{equation}\label{eq:minimax_equal_lower}
  \inf_{\hat{\bbf}}\sup_{\bbf \in \mathcal{F}}\frac{R_n(\hat{\bbf},\bbf)}{R_n(\bw^*|\mathcal{M}_{AS},\bbf)} \geq \frac{\inf_{\hat{\bbf}}\sup_{\bbf \in \mathcal{F}} R_n(\hat{\bbf},\bbf)}{\sup_{\bbf \in \mathcal{F}}R_n(\bw^*|\mathcal{M}_{AS},\bbf)}.
\end{equation}
Based on the results in Section~\ref{sec:proof_impro}, we find the denominator of (\ref{eq:minimax_equal_lower}) has the order
\begin{equation*}
  \sup_{\bbf \in \mathcal{F}}R_n(\bw^*|\mathcal{M}_{AS},\bbf) \sim \sum_{j=1}^{n}\frac{|\theta|_{(j)}^2\sigma^2}{n|\theta|_{(j)}^2+\sigma^2}\asymp n^{-1+\frac{1}{2\alpha_1}},
\end{equation*}
where the last approximation is due to the Example 1 of \cite{Peng2021}. Therefore, the main task is to lower bound the minimax risk $\inf_{\hat{\bbf}}\sup_{\bbf \in \mathcal{F}} R_n(\hat{\bbf},\bbf)$.

Recall that the true regression mean vector has the form $\bbf= \sqrt{n}\sum_{j=1}^{n}\theta_j\bphi_j$, where $\theta_j=\theta_j(\bbf)=\bphi_j^{\top}\bbf/\sqrt{n},j=1,\ldots,n$. For any estimator $\hat{\bbf}$, define
$$
\theta_j(\hat{\bbf})=\frac{\bphi_j^{\top}\hat{\bbf}}{\sqrt{n}},\quad j=1,\ldots,n,
$$
which are the transformed coefficients of $\hat{\bbf}$. Then we have
\begin{equation*}
  \begin{split}
       & \inf_{\hat{\bbf}}\sup_{\bbf \in \mathcal{F}} R_n(\hat{\bbf},\bbf) = \inf_{\hat{\bbf}}\sup_{\bbf \in \mathcal{F}} \frac{1}{n}\mathbb{E}\left\|\hat{\bbf}-\bbf \right\|^2\\
       & = \inf_{\hat{\bbf}}\sup_{\bbf \in \mathcal{F}} \frac{1}{n}\mathbb{E}\left\|\sqrt{n}\sum_{j=1}^{n}\left[\theta_j(\hat{\bbf}) - \theta_j\right]\bphi_j+  \hat{\bbf} - \sqrt{n}\sum_{j=1}^{n}\theta_j(\hat{\bbf})\bphi_j \right\|^2\\
       & \geq \inf_{\hat{\bbf}}\sup_{\bbf \in \mathcal{F}} \mathbb{E}\sum_{j=1}^{n}\left[\theta_j(\hat{\bbf}) - \theta_j\right]^2=  \inf_{\hat{\bbf}}\sup_{\btheta \in \Theta} \mathbb{E}\sum_{j=1}^{n}\left[\theta_j(\hat{\bbf}) - \theta_j\right]^2,
  \end{split}
\end{equation*}
where the inequality follows from the orthogonality of $\{\bphi_1,\ldots,\bphi_{n} \}$. Recall the Gaussian sequence model \begin{equation}\label{eq:sequence2_recall}
  \hat{\theta}_{j} = \theta_j + e_j, \quad j=1,\ldots,n,
\end{equation}
where $\hat{\theta}_{j} =\bphi_j^{\top}\by/\sqrt{n}$, $\theta_j=\bphi_j^{\top}\bbf/\sqrt{n}$, $e_j=\bphi_j^{\top}\beps/\sqrt{n}$, and $e_1,\ldots,e_n$ i.i.d. $N(0,\sigma^2/n)$. Note that we can replace the infimum over arbitrary estimators $\hat{\bbf}$ by the infimum over estimators $\hat{\bv}=\hat{\bv}(\hat{\theta}_{1},\ldots,\hat{\theta}_{n})$ depending only on the statistics $\hat{\theta}_{j} =\bphi_j^{\top}\by/\sqrt{n},j=1,\ldots,n$. Indeed, by defining $\tilde{v}_j=\mathbb{E}(\theta_j(\hat{\bbf})|\hat{\theta}_{1},\ldots,\hat{\theta}_{n})$, we have
\begin{equation*}
  \inf_{\hat{\bbf}}\sup_{\btheta \in \Theta} \mathbb{E}\sum_{j=1}^{n}\left[\theta_j(\hat{\bbf}) - \theta_j\right]^2 \geq \inf_{\hat{\bbf}}\sup_{\btheta \in \Theta} \mathbb{E}\sum_{j=1}^{n}\left(\tilde{v}_j - \theta_j\right)^2\geq\inf_{\hat{\bv}}\sup_{\btheta \in \Theta} \mathbb{E}\| \hat{\bv}- \btheta \|^2,
\end{equation*}
where the first inequality follows from the conditional expectation formula.

Now we lower bound $\inf_{\hat{\bv}}\sup_{\btheta \in \Theta} \mathbb{E}\| \hat{\bv}- \btheta \|^2$ based on the local metric entropy method in \cite{Yang1999Information} and \cite{Wang2014}. Define $m_n^{**}= \lfloor ( n/\log n )^{1/(2\alpha_1)} \rfloor$. Let $\{0, 1 \}^{n-m_n^{**}} \triangleq \{ \bomega=(\omega_1,\ldots,\omega_{n-m_n^{**}}), \omega_i \in \{0,1 \} \}$ denote the set of all binary sequences of length $n-m_n^{**}$. Suppose $n$ is large enough such that $1 \leq m_n^{**} \leq (n-m_n^{**})/4$. Define $\{0, 1 \}^{n-m_n^{**}}_{m_n^{**}}\triangleq \{\bomega \in \{0, 1 \}^{n-m_n^{**}}: d_{H}(\boldsymbol{0}, \bomega)=m_n^{**} \}$, where $d_{H}(\bomega, \bomega') = \sum_{i=1}^{n-m_n^{**}}I_{\{\omega_i \neq \omega_i'\}}$ is the Hamming distance between the binary sequences $\bomega$ and $\bomega'$. Then, the sparse Varshamov-Gilbert bound \citep[Lemma 4.10]{Massart2007} states that there exists a subset $\mathcal{S}$ of $\{0, 1 \}^{n-m_n^{**}}_{m_n^{**}}$ such that
$$
\log |\mathcal{S}|\geq \frac{m_n^{**}}{5}\log\frac{2(n-m_n^{**})}{m_n^{**}},
$$
and any two distinct vectors in $\mathcal{S}$ are separated by at least $m_n^{**}/2$ in the Hamming distance.

We now map every $\bs \in \mathcal{S}$ to a vector $\btheta=(\theta_1,\ldots,\theta_n)^{\top}$ in $\Theta$. Define an $m_n^{**}$-dimensional vector $\bI_1=\bI_1(\bs)\triangleq\{j:s_j=1 \}$ and an $(n-2m_n^{**})$-dimensional vector $\bI_0=\bI_0(\bs)\triangleq\{j:s_j=0 \}$. Then for a $\bs \in \mathcal{S}$, set $\theta_j=|\theta|_{(j)}$ and $\theta_{m_n^{**}+j}=|\theta|_{(m_n^{**}+i_{1j})}$ for $1\leq j \leq m_n^{**}$, where $i_{1j}$ is the $j$-th element of $\bI_1$. Set $\theta_{2m_n^{**}+j}=|\theta|_{(m_n^{**}+i_{0j})}$ for $1\leq j \leq n-2m_n^{**}$, where $i_{0j}$ is the $j$-th element of $\bI_0$. Denote by $\mathcal{P}$ the image of $\mathcal{S}$ under this mapping. And for any two different vectors $\btheta, \btheta' \in \mathcal{P}$, we have
\begin{equation*}
\begin{split}
     & \| \btheta - \btheta' \|^2 \geq 2\sum_{j=\frac{7m_n^{**}}{4}}^{2m_n^{**}} \left[ j^{-\alpha_1} - \left( j+ \frac{m_n^{**}}{4} \right)^{-\alpha_1} \right]^2 \\
     & \geq 2\sum_{j=\frac{7}{4}m_n^{**}}^{2m_n^{**}} \left[ j^{-\alpha_1} - \left( \frac{9}{8}j \right)^{-\alpha_1} \right]^2\\
     & = 2\left[1 - \left(\frac{8}{9} \right)^{\alpha_1} \right]^2\sum_{j=\frac{7}{4}m_n^{**}}^{2m_n^{**}}j^{-2\alpha_1} = C_1 \left(m_n^{**}\right)^{-2\alpha_1+1}.
\end{split}
\end{equation*}
Define $\btheta_0$ a vector that $\theta_{0j}=|\theta|_{(j)}$ for $1\leq j \leq m_n^{**}$ and $\theta_{0j}=0$ for $m_n^{**}+1\leq j \leq n$. For any $\btheta \in \mathcal{P}$, we have
\begin{equation*}
  \begin{split}
       & \| \btheta - \btheta_0 \| ^2 = \sum_{j=m_n^{**}+1}^{n}|\theta|_{(j)}^2= \sum_{j=m_n^{**}+1}^{n}j^{-2\alpha_1}\\
       & \leq  C_2 \left(m_n^{**}\right)^{-2\alpha_1+1},
  \end{split}
\end{equation*}
where $C_2$ can be chosen sufficiently large than $C_1$.  Thus, $\mathcal{P}$ is a $C_1^{\frac{1}{2}} \left(m_n^{**}\right)^{-\alpha_1+\frac{1}{2}}$ packing set of the Euclidean ball $B(\btheta_0, C_2^{\frac{1}{2}} \left(m_n^{**}\right)^{-\alpha_1+\frac{1}{2}})$. In addition, we have
$$
\log|\mathcal{P}|=\log|\mathcal{S}|=\frac{m_n^{**}}{5}\log\frac{2(n-m_n^{**})}{m_n^{**}} \gtrsim n  \left(m_n^{**}\right)^{-2\alpha_1+1}.
$$
Thus, based on Proposition~16 of \cite{Wang2014} (see also Section~7 of \cite{Yang1999Information}), we know
\begin{equation*}
  \inf_{\hat{\bv}}\sup_{\btheta \in \Theta} \mathbb{E}\| \hat{\bv}- \btheta \|^2 \gtrsim \left(m_n^{**}\right)^{-2\alpha_1+1} = \left( \frac{n}{\log n} \right)^{-1+\frac{1}{2\alpha_1}}.
\end{equation*}
Thus, we have
\begin{equation*}
  \inf_{\hat{\bbf}}\sup_{\bbf \in \mathcal{F}}\frac{R_n(\hat{\bbf},\bbf)}{R_n(\bw^*|\mathcal{M}_{AS},\bbf)} \gtrsim \frac{\left( \frac{n}{\log n} \right)^{-1+\frac{1}{2\alpha_1}}}{n^{-1+\frac{1}{2\alpha_1}}} = (\log n)^{1-\frac{1}{2\alpha_1}}\to \infty,
\end{equation*}
which completes the proof.

\section{Proof of Theorem~\ref{tho:minimax}}\label{sec:proof_minimax}

We first give some well-established minimax results when $\btheta$ lies in the ellipsoid $\Theta(\alpha, R)=\{ \btheta \in \mathbb{R}^{n}:\sum_{j=1}^{n}j^{2\alpha}\theta_j^2\leq R \}$. Similar to Section~\ref{sec:a:proof:s3:preliminaries}, we can rewrite the model (\ref{eq:model_matrix}) as a form of Gaussian sequence model
\begin{equation}\label{eq:sequence_mm}
  \hat{\theta}_{j} = \theta_j + e_j, \quad j=1,\ldots,n,
\end{equation}
where $\hat{\theta}_{j} =\bphi_j^{\top}\by/\sqrt{n}$, $\theta_j=\bphi_j^{\top}\bbf/\sqrt{n}$, $e_j=\bphi_j^{\top}\beps/\sqrt{n}$, and $e_1,\ldots,e_n$ are i.i.d. $N(0,\sigma^2/n)$. According to (\ref{eq:riskw}), we have
\begin{equation}\label{eq:risk_minimax_proof}
  R_n(\bw|\mathcal{M}_A,\bbf)=\sum_{j=1}^{n}\left[(1-\gamma_j)^2\theta_j^2+\frac{\sigma^2\gamma_j^2}{n}\right],
\end{equation}
where $\gamma_j=\sum_{m=j}^{n}w_m$. Note that the MA risk (\ref{eq:risk_minimax_proof}) coincides with the risk of the linear estimator $\hat{\btheta}(\bgamma)=(\gamma_1\hat{\theta}_1,\ldots,\gamma_n\hat{\theta}_n)^{\top}$ in the Gaussian sequence model (\ref{eq:sequence_mm}), i.e., $R_n(\bw|\mathcal{M}_A,\bbf)=\mathbb{E}\|\hat{\btheta}(\bgamma) -\btheta \|^2$. For the Gaussian sequence model, \citet{Pinsker1980} obtained an exact evaluation for the linear minimax risk over the ellipsoid $\Theta(\alpha, R)$ and showed that the optimal minimax risk is asymptotically equivalent to the optimal linear minimax risk. \citet{Pinsker1980}'s results yield the minimax risk and the linear-combined minimax risk of MA
\begin{equation}\label{eq:pinsker}
  R_M\left[\mathcal{F}_{\Theta(\alpha, R)}\right]\sim R_L\left[\mathcal{F}_{\Theta(\alpha, R)}\right] \sim C_1\left(\frac{\sigma^2}{n}  \right)^{\frac{2\alpha}{2\alpha+1}},
\end{equation}
where $C_1$ is the Pinsker constant which only depends on $\alpha$ and $R$. Define $x_{+}=\max(x, 0)$. The minimax optimal weights are given by $\tilde{w}_j^*=\tilde{\gamma}_j^*-\tilde{\gamma}^*_{j-1}, j=1\ldots,n$, where
\begin{equation}\label{eq:pinsker2}
  \tilde{\gamma}_j^*=\left[1-C_2\left(\frac{\sigma^2}{n} \right)^{\frac{\alpha}{2\alpha+1}} j^{\alpha}\right]_+,
\end{equation}
and $C_2$ a constant that depends on $\alpha$ and $R$. Since $\tilde{\gamma}_1^* \to 1$ and $\tilde{\gamma}_j^*\geq \tilde{\gamma}_{j+1}^*$, we see that when $n$ is large enough, $(\tilde{w}_1^*,\ldots,\tilde{w}_n^*)^{\top}$ lies in the unit simplex $\mathcal{W}_n$.

Now we consider the candidate model set $\mathcal{M}_{G1}$ with $k_1=\lceil \log n \rceil$ and $\zeta_n = 1/\log n$. The relation (\ref{eq:upper_psi_G1}) shows $\psi(\mathcal{M}_{G1}) \lesssim (\log n\log\log n)^2$. Based on the risk bound (\ref{eq:oracle_MG}), we have
\begin{equation}\label{eq:oracle_MG_mm}
\begin{split}
   \mathbb{E}L_n(\hat{\bw}|\mathcal{M}_{G1},\bbf)&\leq \left(1+\frac{1}{\log n}\right)R_n(\bw^*|\mathcal{M}_A,\bbf)+\frac{C(\log n\log\log n)^2}{n}\\
   &+C\left[\frac{(\log n\log\log n)^2}{n}\right]^{\frac{1}{2}}\left[\left(1+\frac{1}{\log n}\right)R_n(\bw^*|\mathcal{M}_A,\bbf)\right]^{\frac{1}{2}}\\
   &+C\mathbb{E}\left( \hat{\sigma}^2-\sigma^2 \right)^2+ C\left[ \mathbb{E}\left( \hat{\sigma}^2-\sigma^2 \right)^2 \right]^{\frac{1}{2}}\left[\left(1+\frac{1}{\log n}\right)R_n(\bw^*|\mathcal{M}_A,\bbf)\right]^{\frac{1}{2}}.\\
\end{split}
\end{equation}
Then, taking the upper bound on both sides of (\ref{eq:oracle_MG_mm}) with respect to $\bbf \in \mathcal{F}_{\Theta(\alpha, R)}$ gives
\begin{equation}\label{eq:mini1}
\begin{split}
\sup_{\bbf \in \mathcal{F}_{\Theta(\alpha, R)}}\mathbb{E}L_n(\hat{\bw}|\mathcal{M}_{G1},\bbf) & \leq \left(1+\frac{1}{\log n}\right) \sup_{\bbf \in \mathcal{F}_{\Theta(\alpha, R)}}R_n(\bw^*|\mathcal{M}_A,\bbf)+\frac{C(\log n\log\log n)^2}{n}\\
&
    +C\left[\frac{(\log n\log\log n)^2}{n}\right]^{\frac{1}{2}}\left[\left(1+\frac{1}{\log n}\right)\sup_{\bbf \in \mathcal{F}_{\Theta(\alpha, R)}}R_n(\bw^*|\mathcal{M}_A,\bbf)\right]^{\frac{1}{2}}\\
    &+ C\sup_{\bbf \in \mathcal{F}_{\Theta(\alpha, R)}}\mathbb{E}\left( \hat{\sigma}^2-\sigma^2 \right)^2\\
    &+ C\left[\sup_{\bbf \in \mathcal{F}_{\Theta(\alpha, R)}} \mathbb{E}\left( \hat{\sigma}^2-\sigma^2 \right)^2 \right]^{\frac{1}{2}}\left[\left(1+\frac{1}{\log n}\right)\sup_{\bbf \in \mathcal{F}_{\Theta(\alpha, R)}}R_n(\bw^*|\mathcal{M}_A,\bbf)\right]^{\frac{1}{2}}.
\end{split}
\end{equation}
The first term on the right side of (\ref{eq:mini1}) is upper bounded by
\begin{equation}\label{eq:mini1_1}
\begin{split}
   &\left(1+\frac{1}{\log n}\right)\sup_{\bbf \in \mathcal{F}_{\Theta(\alpha, R)}}R_n(\bw^*|\mathcal{M}_A,\bbf) \\ & \leq \left(1+\frac{1}{\log n}\right) \inf_{\bw \in \mathcal{W}_n}\sup_{\bbf \in \mathcal{F}_{\Theta(\alpha, R)}}R_n(\bw|\mathcal{M}_A,\bbf) \\
     & =R_L\left[\mathcal{F}_{\Theta(\alpha, R)}\right]+\frac{R_L\left[\mathcal{F}_{\Theta(\alpha, R)}\right]}{\log n}\\
     &+\left(1+\frac{1}{\log n}\right)\left[\inf_{\bw \in \mathcal{W}_n}\sup_{\bbf \in \mathcal{F}_{\Theta(\alpha, R)}}R_n(\bw|\mathcal{M}_A,\bbf)-\inf_{\bw}\sup_{\bbf \in \mathcal{F}_{\Theta(\alpha, R)}}R_n(\bw|\mathcal{M}_A,\bbf)\right],
\end{split}
\end{equation}
where the first inequality is due to the definition of $\bw^*|\mathcal{M}_A$, and the second equality is due to the definition of $R_L\left[\mathcal{F}_{\Theta(\alpha, R)}\right]$. Thus, it remains to prove
\begin{equation}\label{eq:remain0}
  \inf_{\bw \in \mathcal{W}_n}\sup_{\bbf \in \mathcal{F}_{\Theta(\alpha, R)}}R_n(\bw|\mathcal{M}_A,\bbf)-\inf_{\bw}\sup_{\bbf \in \mathcal{F}_{\Theta(\alpha, R)}}R_n(\bw|\mathcal{M}_A,\bbf)= o\left(R_L\left[\mathcal{F}_{\Theta(\alpha, R)}\right]\right),
\end{equation}
\begin{equation}\label{eq:remain1}
  \frac{(\log n\log\log n)^2}{n} = o\left(R_L\left[\mathcal{F}_{\Theta(\alpha, R)}\right]\right),
\end{equation}
and
\begin{equation}\label{eq:remain2}
  \sup_{\bbf \in \mathcal{F}_{\Theta(\alpha, R)}}\mathbb{E}\left( \hat{\sigma}^2-\sigma^2 \right)^2 = o\left(R_L\left[\mathcal{F}_{\Theta(\alpha, R)}\right]\right)
\end{equation}
for all $\alpha>0$ and $R>0$. For (\ref{eq:remain0}), based on Chapter~3 of \cite{tsybakov2008introduction}, we have
\begin{equation}
\begin{split}
   &\inf_{\bw \in \mathcal{W}_n}\sup_{\bbf \in \mathcal{F}_{\Theta(\alpha, R)}}R_n(\bw|\mathcal{M}_A,\bbf)-\inf_{\bw}\sup_{\bbf \in \mathcal{F}_{\Theta(\alpha, R)}}R_n(\bw|\mathcal{M}_A,\bbf)\\
    &\asymp \frac{1-\tilde{\gamma}_1^*}{n} = o\left(R_L\left[\mathcal{F}_{\Theta(\alpha, R)}\right]\right),
\end{split}
\end{equation}
where the last equality is due to (\ref{eq:pinsker})--(\ref{eq:pinsker2}). The condition (\ref{eq:remain1}) can be easily proved for all $\alpha>0$ and $R>0$ by noticing (\ref{eq:pinsker}). The condition (\ref{eq:remain2}) is satisfied when the estimator $\hat{\sigma}^2_{D}$ with the parametric rate $1/n$ is adopted. When $\hat{\sigma}^2 = \hat{\sigma}_{m_n}^2$ with $m_n = \lfloor kn \rfloor$ ($0<k<1$), we have
\begin{equation*}
\begin{split}
   &\mathbb{E}(\hat{\sigma}_{m_n}^2-\sigma^2)^2 \lesssim n^{-1}\vee \left(\sum_{j=\lfloor kn \rfloor+1}^{n}\theta_j^2\right)^2\\
   &\leq n^{-1}\vee \left[(kn)^{-2\alpha}\sum_{j=\lfloor kn \rfloor+1}^{n}j^{2\alpha}\theta_j^2\right]^2
     \lesssim n^{-1} \vee n^{-4\alpha},
\end{split}
\end{equation*}
where the first inequality follows from (\ref{eq:variance_risk}), and the third inequality follows from the definition of $\Theta(\alpha, R)$. Thus, we obtain
\begin{equation}
  \sup_{\bbf \in \mathcal{F}_{\Theta(\alpha, R)}}\mathbb{E}\left( \hat{\sigma}^2-\sigma^2 \right)^2\lesssim n^{-1} \vee n^{-4\alpha}=o\left(R_L\left[\mathcal{F}_{\Theta(\alpha, R)}\right]\right)
\end{equation}
for all $\alpha>0$ and $R>0$. Combining (\ref{eq:mini1})--(\ref{eq:remain2}), we have proved the exact linear-combined minimax adaptivity of MMA on the family of ellipsoids. According to (\ref{eq:pinsker}), MMA also achieves the exact minimax adaptivity on the family of ellipsoids.

The linear-combined minimax risk over the hyperrectangle is
\begin{equation}\label{eq:minin_risk_hyper}
\begin{split}
    & R_L\left[\mathcal{F}_{\Theta^H(c,q)}\right]=\inf_{\mathbf{w}}\sup_{\bbf\in \mathcal{F}_{\Theta^H(c,q)}}R_n(\bw|\mathcal{M}_A,\bbf) \\
     & =\sum_{j=1}^{n}\frac{c^2j^{-2q}\sigma^2}{nc^2j^{-2q}+\sigma^2}\asymp n^{-1+\frac{1}{2q}},
\end{split}
\end{equation}
where the second equality is due to the definition of $\Theta^{H}(c,q)$ and (\ref{eq:risk_minimax_proof}), and the last approximation can be obtained based on the similar technique in the proof of Theorem~1 of \cite{Peng2021}.
Likewise, by taking the upper bound on both sides of (\ref{eq:risk_bound_general}) over $\bbf \in \mathcal{F}_{\Theta^H(c,q)}$, we see that the results can be proved if we show
\begin{equation}\label{eq:remain10}
  \inf_{\bw \in \mathcal{W}_n}\sup_{\bbf \in \mathcal{F}_{\Theta^H(c,q)}}R_n(\bw|\mathcal{M}_A,\bbf)-\inf_{\bw}\sup_{\bbf \in \mathcal{F}_{\Theta^H(c,q)}}R_n(\bw|\mathcal{M}_A,\bbf)= o\left(R_L\left[\mathcal{F}_{\Theta^H(c,q)}\right]\right),
\end{equation}
\begin{equation}\label{eq:remain11}
  \frac{(\log n\log\log n)^2}{n} = o\left(R_L\left[\mathcal{F}_{\Theta^H(c,q)}\right]\right),
\end{equation}
and
\begin{equation}\label{eq:remain12}
  \sup_{\bbf \in \mathcal{F}_{\Theta^H(c,q)}}\mathbb{E}\left( \hat{\sigma}^2-\sigma^2 \right)^2 = o\left(R_L\left[\mathcal{F}_{\Theta^H(c,q)}\right]\right)
\end{equation}
for all $c>0$ and $q>1/2$. Note that
\begin{equation}
\begin{split}
    & \inf_{\bw \in \mathcal{W}_n}\sup_{\bbf \in \mathcal{F}_{\Theta^H(c,q)}}R_n(\bw|\mathcal{M}_A,\bbf)-\inf_{\bw}\sup_{\bbf \in \mathcal{F}_{\Theta^H(c,q)}}R_n(\bw|\mathcal{M}_A,\bbf) \\
     & =\frac{\sigma^4}{n^2c^2+n\sigma^2}=o\left(R_L\left[\mathcal{F}_{\Theta^H(c,q)}\right]\right),
\end{split}
\end{equation}
which implies (\ref{eq:remain10}). The equation (\ref{eq:remain11}) holds for all $c>0$ and $q>1/2$ due to (\ref{eq:minin_risk_hyper}). The condition (\ref{eq:remain12}) is naturally satisfied for the estimator $\hat{\sigma}^2_{D}$. When $\hat{\sigma}^2 = \hat{\sigma}_{m_n}^2$ with $m_n = \lfloor kn \rfloor$ ($0<k<1$) is adopted, we have
\begin{equation*}
\begin{split}
   \mathbb{E}(\hat{\sigma}_{m_n}^2-\sigma^2)^2 & \lesssim n^{-1}\vee \left(\sum_{j=\lfloor kn \rfloor+1}^{n}\theta_j^2\right)^2\leq n^{-1}\vee \left(c^2\sum_{j=\lfloor kn \rfloor+1}^{n}j^{-2q}\right)^2 \\
     & \lesssim n^{-1} \vee n^{-2q+1}=o\left(n^{-1+\frac{1}{2q}}\right)
\end{split}
\end{equation*}
for all $q>1/2$, which implies (\ref{eq:remain12}). Thus, we see that the MMA estimator is adaptive in the exact linear-combined minimax sense on the family of hyperrectangles.

\section{Additional theoretical results}\label{sec:a:add_theory}

\subsection{Improvability of discretized MA over MS}\label{sec:A_5_1}

Section~\ref{subsec:review_hansen} of the main text compares the optimal MA risk with the optimal MA risk subject to the discrete weight restriction (\ref{eq:condition_discrete}). On the other hand, MS can be viewed as MA in the discrete set $\mathcal{W}_{|\mathcal{M}_S|}(1)$. Recall that $m_n^*$ denotes the optimal single model among all nested candidate models, and $m^*|\mathcal{M}_S$ stands for the optimal model in $\mathcal{M}_S$. Thus we have $R_n\left(m^*|\mathcal{M}_S,\bbf\right)\geq R_n\left(\bw_N^*|\mathcal{M}_S,\bbf\right)$. A natural question to ask is whether $R_n\left(\bw_N^*|\mathcal{M}_S,\bbf\right)$ has a substantial improvement over $R_n\left(m^*|\mathcal{M}_S,\bbf\right)$ when $N\geq 2$.

\begin{proposition}\label{lemma:delta_11}
    Suppose Assumptions~\ref{asmp:square_summable}--\ref{asmp:regressor_order} hold. Under Condition~\ref{condition1} and $M_n\gtrsim m_n^*$, define
    $$
    r\triangleq\log_\kappa\left( \frac{m_n^*}{M_n}\vee 1 \right),
    $$
    where $\kappa$ is the constant appearing in Condition~\ref{condition1}.
    If $N>(1+\delta^{2r+2})/(2\delta^{2r+2})$, we have
    $$
    R_n\left( m^*|\mathcal{M}_{S},\bbf\right) - R_n\left( \bw_N^*|\mathcal{M}_{S}, \bbf\right) \asymp R_n\left( m^*|\mathcal{M}_{S},\bbf\right).
    $$
    Under Condition~\ref{condition2} or $M_n=o(m_n^*)$, for any $N$, we have
    $$
    R_n\left( m^*|\mathcal{M}_{S},\bbf\right) - R_n\left( \bw_N^*|\mathcal{M}_{S}, \bbf\right)=o\left[ R_n\left( m^*|\mathcal{M}_{S},\bbf\right)\right].
    $$
\end{proposition}

When $\theta_l$ decays slowly and $M_n$ is large, the optimal model size $m^*|\mathcal{M}_{S}$ is not very small relative to the sample size $n$, and MA under the discrete weight set still reduces the risk of MS substantially, although it does not provide the full potential of MA. For example, when $\theta_j=j^{-\alpha_1},\alpha_1>1/2$, Condition~\ref{condition1} is satisfied for any $\kappa>1$ and $\delta=\kappa^{-\alpha_1}$. Then, for a large candidate model set with $M_n\geq m_n^*$, the condition for improving over MS is
$
N>(1+\delta^{2})/(2\delta^{2})=(1+\kappa^{2\alpha_1})/2.
$
Due to the arbitrariness of $\kappa$, it suffices to require $N\geq 2$.


\begin{proof}

We first prove the claim under Condition~\ref{condition1} and $M_n\gtrsim m_n^*$. Recall that $\bw_N^*|\mathcal{M}_S=\arg\min_{\bw \in \mathcal{W}_{|\mathcal{M}_S|}(N)}R_n(\bw | \mathcal{M}_S,\bbf)$ denotes the optimal discrete weight vector in $\mathcal{W}_{|\mathcal{M}_S|}(N)$, and $m^*|\mathcal{M}_S=\arg\min_{m \in \{1,\ldots,M_n\}}R_n(m|\mathcal{M}_S,\bbf)$ is the size of the optimal candidate model in $\mathcal{M}_S$. Since MS can be seen as the MA on the discrete weight set with $N=1$, we have $R_n\left( m^*|\mathcal{M}_{S},\bbf\right)=R_n\left( \bw_1^*|\mathcal{M}_{S}, \bbf\right)$, where $\bw_1^*|\mathcal{M}_{S}=(w_{1,1}^*,\ldots,w_{1,M_n}^*)^{\top}$ and the optimal discrete cumulative weights for MS is $\gamma_{1,j}^*=\sum_{m=j}^{M_n}w_{1,m}^*$. From (\ref{eq:risk_m}) and (\ref{eq:riskw_large}), we have
\begin{equation}\label{eq:gamma1j}
  \gamma_{1,j}^*=\left\{\begin{array}{ll}
1 &\quad 1\leq j \leq (m_n^*\wedge M_n), \\
0 &\quad (m_n^*\wedge M_n)<j\leq M_n.
\end{array}\right.
\end{equation}

From (\ref{eq:riskw_large}), we see that the risk difference between MS and MA is
\begin{equation}\label{eq:p111}
\begin{split}
   &R_n\left( m^*|\mathcal{M}_{S},\bbf\right) - R_n\left( \bw_N^*|\mathcal{M}_{S}, \bbf\right)\\
   &=R_n\left( \bw_1^*|\mathcal{M}_{S}, \bbf\right) -R_n\left( \bw_N^*|\mathcal{M}_{S}, \bbf\right)  \\
   & =\sum_{j=1}^{M_n}\left( \frac{\sigma^2}{n} + \theta_j^2 \right) \left( \gamma_{1,j}^* - \frac{\theta_j^2}{\theta_j^2+\frac{\sigma^2}{n}} \right)^2 - \sum_{j=1}^{M_n}\left( \frac{\sigma^2}{n} + \theta_j^2 \right) \left( \gamma_{N,j}^* - \frac{\theta_j^2}{\theta_j^2+\frac{\sigma^2}{n}} \right)^2 \\
   &\geq \sum_{j=1}^{M_n\wedge m_n^*}\left( \frac{\sigma^2}{n} + \theta_j^2 \right)\left( 1 - \frac{\theta_j^2}{\theta_j^2+\frac{\sigma^2}{n}} \right)^2- \frac{1}{4N^2}\sum_{j=1}^{M_n\wedge m_n^*}\left( \frac{\sigma^2}{n} + \theta_j^2 \right),
\end{split}
\end{equation}
where the inequality is due to (\ref{eq:gamma1j}) and the fact $$\left|\gamma_{N,j}^* - \frac{\theta_j^2}{\theta_j^2+\frac{\sigma^2}{n}} \right| \leq \frac{1}{2N}.$$

Define
\begin{equation*}
  d_4^*=\left\{\begin{array}{ll}
\arg\min_{d \in
\mathbb{N}}\{G_d(m_n^*+1)<M_n\} &\quad M_n<m_n^*, \\
0 &\quad M_n\geq m_n^*.
\end{array}\right.
\end{equation*}
where the function $G_d$ is given by (\ref{eq:Gfunction}). It is easy to check that
\begin{equation*}\label{eq:check_d2star}
  d_4^* \sim \log_\kappa\left( \frac{m_n^*}{M_n}\vee 1 \right),
\end{equation*}
where $\kappa>1$ is the constant given in Condition~\ref{condition1}. When $N>(1+\delta^{2d_4^*+2})/(2\delta^{2d_4^*+2})$, there must exist a positive constant $\tau$ that satisfies $\delta^{2d_4^*+2}\geq (1+\tau)/[2N-(1+\tau)]$, where $0<\delta<1$ is the constant given in Condition~\ref{condition1}. Then we define
\begin{equation}\label{eq:d3_star}
d_5^*=\argmax_{d\in \mathbb{N}\cup\{0\}}\left\{\delta^{2d+2d_4^*+2}\geq \frac{1+\tau}{2N-(1+\tau)}\right\}
\end{equation}
and the model index $j_n^*=G_1(M_n) \wedge G_{d_5^*}\left(m_n^*+1 \right)$. When $j_n^*=G_1(M_n)$, we have
\begin{equation}\label{eq:what12}
\begin{split}
   \frac{\theta_{m_n^*+1}^2}{\theta_{j_n^*}^2} &   =\frac{\theta_{m_n^*+1}^2}{\theta_{G_1(m_n^*+1)}^2}\times\frac{\theta_{G_1(m_n^*+1)}^2}{\theta_{G_2(m_n^*+1)}^2}\times\cdots\times\frac{\theta_{G_{d_4^*}(m_n^*+1)}^2}{\theta_{G_1(M_n)}^2}
 \\
     & \geq \delta^{2d_4^*} \frac{\theta_{G_{d_4^*}(m_n^*+1)}^2}{\theta_{G_1(M_n)}^2}\geq \delta^{2d_4^*} \frac{\theta_{G_{d_4^*}(m_n^*+1)}^2}{\theta_{G_{d_4^*+1}(m_n^*+1)}^2}\\
     &\geq \delta^{2d_4^*+2},
\end{split}
\end{equation}
where the first inequality follows Condition~\ref{condition1}, and the second inequality is due to $G_{d_4^*}(m_n^*+1)<M_n$ and $G_{d_4^*+1}(m_n^*+1)<G_1(M_n)$. When $j_n^*=G_{d_5^*}\left(m_n^*+1 \right)$, we have
\begin{equation}\label{eq:what123}
\begin{split}
   \frac{\theta_{m_n^*+1}^2}{\theta_{j_n^*}^2} &   =\frac{\theta_{m_n^*+1}^2}{\theta_{G_1(m_n^*+1)}^2}\times\frac{\theta_{G_1(m_n^*+1)}^2}{\theta_{G_2(m_n^*+1)}^2}\times\cdots\times\frac{\theta_{G_{d_5^*-1}(m_n^*+1)}^2}{\theta_{G_{d_5^*}(m_n^*+1)}^2}
\geq \delta^{2d_5^*}.
\end{split}
\end{equation}
Combining (\ref{eq:what12}) with (\ref{eq:what123}), we have
\begin{equation}
\begin{split}
   \frac{\theta_{m_n^*+1}^2}{\theta_{j_n^*}^2} & \geq \delta^{2d_4^*+2}\wedge \delta^{2d_5^*}= \delta^{(2d_4^*+2)\vee(2d_5^*)} \\
     & \geq \delta^{2d_4^*+2d_5^*+2} \geq \frac{1+\tau}{2N-(1+\tau)},
\end{split}
\end{equation}
where the second inequality is due to $0<\delta<1$, and the last inequality is due to the definition (\ref{eq:d3_star}).
Thus when $j\geq j_n^*$, we have
\begin{equation}\label{eq:p112}
\begin{split}
   1 - \frac{\theta_j^2}{\theta_j^2+\frac{\sigma^2}{n}}&\geq 1 - \frac{1}{1+\frac{\theta_{m_n^*+1}^2}{\theta_j^2}} \geq 1-\frac{1}{1+\frac{\theta_{m_n^*+1}^2}{\theta_{j_n^*}^2}}\\
    &\geq  1 - \frac{1}{1+\frac{1+\tau}{2N-(1+\tau)}}= \frac{1+\tau}{2N}.
\end{split}
\end{equation}
Substituting (\ref{eq:p112}) into (\ref{eq:p111}) gives the desired claim
\begin{equation*}
\begin{split}
   &R_n\left( m^*|\mathcal{M}_{S},\bbf\right) - R_n\left( \bw_N^*|\mathcal{M}_{S}, \bbf\right)\\
    &\geq \sum_{j=1}^{M_n\wedge m_n^*}\left( \frac{\sigma^2}{n} + \theta_j^2 \right) \left[\frac{(1+\tau)^2}{4N^2}-\frac{1}{4N^2}\right] \\
     & \geq \frac{(\tau^2+2\tau)(M_n\wedge m_n^*-j_n^*)\sigma^2}{4N^2 n} \asymp \frac{m_n^*}{n}\asymp R_n(m^*|\mathcal{M}_{S},\bbf).
\end{split}
\end{equation*}

The proof of the result under Condition~\ref{condition2} or the condition $M_n=o(m_n^*)$ is straightforward based on Lemma~\ref{lemma:peng} and (\ref{eq:relation3}). This completes the proof of this proposition.

\end{proof}

\subsection{AOP in terms of loss}\label{sec:a:aop_loss}

Theorems~\ref{theorem:main}--\ref{theorem:aop} in the main text focus on the squared risk of the MMA estimator. Note that the definitions of AOP in terms of statistical loss have also been commonly adopted in MS \citep{Stone1984, Li1987, Shao1997} and MA literature \citep{Hansen2007, Wan2010}. The following corollary shows that under the same assumptions in Theorem~\ref{theorem:aop}, MMA is optimal in the sense that its squared loss asymptotically converges to that of the oracle MA estimator in probability.
\begin{corollary}\label{corollary:aop}
  Suppose Assumption~\ref{asmp:square_summable} holds. As $n \to \infty$, if the conditions (\ref{eq:variance_rate})--(\ref{eq:minimum_marisk_rate}) are satisfied, then we have
  \begin{equation*}
    \frac{L_n(\hat{\bw} | \mathcal{M},\bbf)}{\inf_{\bw\in\mathcal{W}_{M_n}}L_n(\bw | \mathcal{M},\bbf)}\to_p 1,
  \end{equation*}
  where $\to_p$ means convergence in probability.
\end{corollary}

\begin{proof}
  From (\ref{eq:addd1}), the MMA criterion can be decomposed as
\begin{equation*}\label{eq:addd10}
  \begin{split}
      & C_{n}(\bw|\mathcal{M},\by)= L_n(\bw|\mathcal{M},\bbf) -  2\sum_{j=1}^{M_n}\sum_{l=k_{j-1}+1}^{k_j}\gamma_j\left( e_l^2-\frac{\sigma^2}{n}\right)  \\
       & -2\sum_{j=1}^{M_n}\sum_{l=k_{j-1}+1}^{k_j}\gamma_j\theta_le_l-2\sum_{j=1}^{M_n}\sum_{l=k_{j-1}+1}^{k_j}\gamma_j\left( \frac{\sigma^2}{n} - \frac{\hat{\sigma}^2}{n} \right)\\
     & + \frac{1}{n}\bbf^{\top}\beps + \frac{1}{n}\|\beps \|^2.
  \end{split}
\end{equation*}
Following the technique in \cite{Li1987}, it is sufficient to verify
\begin{equation}\label{eq:wan_1}
  \sup_{\bw \in \mathcal{W}_{M_n}}\frac{\left|\sum_{j=1}^{M_n}\sum_{l=k_{j-1}+1}^{k_j}\gamma_j\left( e_l^2-\frac{\sigma^2}{n}\right)\right| }{R_n(\bw|\mathcal{M},\bbf)}\to_p 0,
\end{equation}
\begin{equation}\label{eq:wan_2}
  \sup_{\bw \in \mathcal{W}_{M_n}}\frac{\left|\sum_{j=1}^{M_n}\sum_{l=k_{j-1}+1}^{k_j}\gamma_j\theta_le_l \right| }{R_n(\bw|\mathcal{M},\bbf)}\to_p 0,
\end{equation}
\begin{equation}\label{eq:wan_4}
  \sup_{\bw \in \mathcal{W}_{M_n}}\frac{\left|\sum_{j=1}^{M_n}\sum_{l=k_{j-1}+1}^{k_j}\gamma_j\left( \frac{\sigma^2}{n} - \frac{\hat{\sigma}^2}{n} \right) \right| }{R_n(\bw|\mathcal{M},\bbf)}\to_p 0,
\end{equation}
and
\begin{equation}\label{eq:wan_3}
  \sup_{\bw \in \mathcal{W}_{M_n}}\left| \frac{L_n(\bw|\mathcal{M},\bbf)}{R_n(\bw|\mathcal{M},\bbf)}-1 \right|\to_p 0.
\end{equation}
In particular, (\ref{eq:wan_3}) is equivalent to
\begin{equation}\label{eq:wan_31}
  \sup_{\bw \in \mathcal{W}_{M_n}}\frac{\left|\sum_{j=1}^{M_n}\sum_{l=k_{j-1}+1}^{k_j}\gamma_j^2\left( e_l^2-\frac{\sigma^2}{n}\right)\right| }{R_n(\bw|\mathcal{M},\bbf)}\to_p 0
\end{equation}
and
\begin{equation}\label{eq:wan_32}
  \sup_{\bw \in \mathcal{W}_{M_n}}\frac{\left|\sum_{j=1}^{M_n}\sum_{l=k_{j-1}+1}^{k_j}\gamma_j^2\theta_le_l \right| }{R_n(\bw|\mathcal{M},\bbf)}\to_p 0.
\end{equation}

As an example, we prove (\ref{eq:wan_1}) and (\ref{eq:wan_4}). Recall that $z_l=\sqrt{n}e_l/\sigma$, $l=1,\ldots,k_{M_n}$. For any $\delta>0$, we have
\begin{equation}\label{eq:aop_loss_prove}
\begin{split}
    \mathbb{P}&\left\{ \sup_{\bw \in \mathcal{W}_{M_n}}\frac{\left|\sum_{j=1}^{M_n}\sum_{l=k_{j-1}+1}^{k_j}\gamma_j\left( e_l^2-\frac{\sigma^2}{n}\right)\right| }{R_n(\bw|\mathcal{M},\bbf)}>\delta \right\} \\
    =\mathbb{P}&\left\{ \sup_{\bw \in \mathcal{W}_{M_n}}\frac{\left|\sum_{j=1}^{M_n}\sum_{l=k_{j-1}+1}^{k_j}\sigma^2\gamma_j\left( z_l^2-1\right) \right| }{nR_n(\bw|\mathcal{M},\bbf)}>\delta \right\} \\
     \leq \mathbb{P}&\left\{ \sup_{\bw \in \mathcal{W}_{M_n}}\frac{\kappa_1 \left[\sum_{j=1}^{M_n}\sigma^4\gamma_j^2\left(k_j-k_{j-1} \right)\right]^{\frac{1}{2}}  \left[\sum_{j=1}^{M_n}\frac{\left(  k_j^{\frac{1}{2}} -k_{j-1}^{\frac{1}{2}}\right)^2}{k_j-k_{j-1}}\right]^{\frac{1}{2}} }{\left[nR_n(\bw|\mathcal{M},\bbf)\right]^{\frac{1}{2}}\left[nR_n(\bw^*|\mathcal{M},\bbf)\right]^{\frac{1}{2}} }>\delta \right\}\\
     \leq \mathbb{P}&\left\{\kappa_1\sigma\left[\sum_{j=1}^{M_n}\frac{\left(  k_j^{\frac{1}{2}} -k_{j-1}^{\frac{1}{2}}\right)^2}{k_j-k_{j-1}}\right]^{\frac{1}{2}}>\delta\left[nR_n(\bw^*|\mathcal{M},\bbf)\right]^{\frac{1}{2}}   \right\}\\
     \leq(&\mathbb{E}\kappa_1^2)\delta^{-2}\left[nR_n(\bw^*|\mathcal{M},\bbf)\right]^{-1}  \sigma^{2}\left[\sum_{j=1}^{M_n}\frac{\left(  k_j^{\frac{1}{2}} -k_{j-1}^{\frac{1}{2}}\right)^2}{k_j-k_{j-1}}\right]\\
     &\!\!\!\!\!\!\!\!\!\!\!\leq \frac{C\psi(\mathcal{M})}{nR_n(\bw^*|\mathcal{M},\bbf)}\to 0,
\end{split}
\end{equation}
where the first inequality follows from (\ref{eq:step21}), the second inequality follows from (\ref{eq:riskw}), the third inequality is due to Markov’s inequality, and the last inequality follows from the upper bound on $\mathbb{E}\kappa_1^2$ and the definition of $\psi(\mathcal{M})$. For any $\delta>0$, we have
\begin{equation*}\label{eq:aop_loss_prove2}
\begin{split}
\mathbb{P}&\left\{ \sup_{\bw \in \mathcal{W}_{M_n}}\frac{\left|\sum_{j=1}^{M_n}\sum_{l=k_{j-1}+1}^{k_j}\gamma_j\left( \frac{\sigma^2}{n} - \frac{\hat{\sigma}^2}{n} \right) \right| }{R_n(\bw|\mathcal{M},\bbf)}>\delta \right\} \\
\leq\mathbb{P}&\left\{ \sup_{\bw \in \mathcal{W}_{M_n}}\frac{\left(\sum_{j=1}^{M_n}\sum_{l=k_{j-1}+1}^{k_j}\gamma_j^2\frac{\sigma^2}{n}\right)^{\frac{1}{2}}\left[\frac{nk_{M_n}}{\sigma^2}\left( \frac{\sigma^2}{n}-\frac{\hat{\sigma}^2}{n}\right)^2\right]^{\frac{1}{2}} }{R_n(\bw|\mathcal{M},\bbf)}>\delta \right\} \\
\leq\mathbb{P}&\left\{ \sup_{\bw \in \mathcal{W}_{M_n}}\frac{\left[\frac{k_{M_n}}{n\sigma^2}\left( \sigma^2-\hat{\sigma}^2\right)^2\right]^{\frac{1}{2}} }{\left[R_n(\bw|\mathcal{M},\bbf)\right]^{\frac{1}{2}}}>\delta \right\} \\
\leq\mathbb{P}&\left\{ \left[\frac{k_{M_n}}{n\sigma^2}\left( \sigma^2-\hat{\sigma}^2\right)^2\right]^{\frac{1}{2}} >\delta\left[R_n(\bw^*|\mathcal{M},\bbf)\right]^{\frac{1}{2}} \right\} \\
&\!\!\!\!\!\!\!\!\!\!\!\!\leq  \frac{\mathbb{E}\left[\frac{k_{M_n}}{\sigma^2}\left( \sigma^2-\hat{\sigma}^2\right)^2\right]}{\delta^2 n R_n(\bw^*|\mathcal{M},\bbf)}\to 0.
\end{split}
\end{equation*}
The remaining equations in (\ref{eq:wan_1})--(\ref{eq:wan_31}) can also be proved using similar techniques in Section~\ref{sec:Proof_main} and (\ref{eq:aop_loss_prove}). Thus we skip the similar materials here.
\end{proof}

\subsection{Candidate model set with equal-sized groups}\label{sec:equal_sized}

Consider $\mathcal{M}_{G2}$ with $\zeta_n = 0$, $k_1 = \lceil(\log n)^{t}\rceil $, $k_m = mk_1$ for $m=2,\ldots,M_n-1$, and $k_{M_n}=p_n$, where $0<t<3$ and $M_n= \arg\min_{m \in \mathbb{N}} k_m \geq p_n$. We have $M_n$ is upper bounded by
$$
  M_n \lesssim \frac{n}{(\log n)^{t}}.
$$
Now we verify the condition $k_1\vee \psi(\mathcal{M}_{G2})=o[nR_n(\bw^*|\mathcal{M}_A,\bbf)]$ in the following examples.

\begin{example}[Polynomially decaying coefficients]
\label{example1}

When $\theta_j=j^{-\alpha_1}, \alpha_1>1/2$, we have $nR_n(\bw^*|\mathcal{M}_A,\bbf)\asymp m_n^* \asymp n^{1/(2\alpha_1)}$. Note that $k_1/n^{1/(2\alpha_1)} \to 0$ and $\psi(\mathcal{M}_{G2})/n^{1/(2\alpha_1)} \to 0$ due to $\psi(\mathcal{M}_{G2}) \lesssim \log n (\log\log n)^2$. Thus, the MMA estimator based on $\mathcal{M}_{G2}$ still attains the full AOP in the case of polynomially decaying coefficients.

\end{example}

\begin{example}[Exponentially decaying coefficients]
\label{example1}

Now the transformed coefficients decay fast: $\theta_j=\exp(-cj^{\alpha_2})$, $\alpha_2>0$. In this case, $nR_n(\bw^*|\mathcal{M}_A,\bbf)\asymp m_n^* \asymp (\log n)^{1/\alpha_2}$. However, $\mathcal{M}_{G2}$ is not proved to achieve the full AOP due to $\psi(\mathcal{M}_{G2})/(\log n)^{1/\alpha_2} \to \infty$.

\end{example}

\begin{proof}
  We first derive an upper bound for $\psi(\mathcal{M}_{G2})$. Note that
\begin{equation*}
\begin{split}
   1+ \sum_{j=1}^{M_n-1}\frac{k_{j+1} - k_{j}}{4k_{j}} \leq 1 + \sum_{j=1}^{M_n-1}\frac{1}{4j} \leq C\log M_n .
\end{split}
\end{equation*}
Based on (\ref{eq:psi_M}) and (\ref{eq:upper_S}), another term in $\psi(\mathcal{M}_{G2})$ is upper bounded by
\begin{equation*}
  \begin{split}
       1+\sum_{j=1}^{M_n-1} \frac{S_{j-1} - S_{j}}{4S_{j}} & \leq C\max_{1 \leq j \leq M_n-1}\left\{ \frac{S_{j-1} - S_{j}}{S_{j} - S_{j+1}}\right\}\log \frac{1}{S_{M_n-1}}.
  \end{split}
\end{equation*}
We have
\begin{equation*}
  \log \frac{1}{S_{M_n-1}} = \log \frac{1}{\sum_{j=k_{M_n-1}+1}^{k_{M_n}}\theta_j^2} \leq \log \frac{1}{\theta_n^2}=2\alpha_1\log n
\end{equation*}
and
\begin{equation*}
  \begin{split}
       & \max_{1 \leq j \leq M_n-1}\left\{ \frac{S_{j-1} - S_{j}}{S_{j} - S_{j+1}}\right\} \\
       & = \max_{1 \leq j \leq M_n-1}\left\{ \frac{\sum_{l=k_{j-1}+1}^{k_{j}}\theta_l^2}{\sum_{l=k_{j}+1}^{k_{j+1}}\theta_l^2}\right\} \leq \max_{1 \leq j \leq M_n-1}\left\{ \frac{(k_{j} - k_{j-1})\theta_{k_{j-1}}^2}{(k_{j+1} - k_{j})\theta_{k_{j}}^2}\right\}\\
       & \leq C \max_{1 \leq j \leq M_n-1}\left\{ \frac{\theta_{k_{j-1}}^2}{\theta_{k_{j}}^2}\right\} \leq C.
  \end{split}
\end{equation*}
Thus, we obtain
\begin{equation*}
  \psi(\mathcal{M}_{G2}) \leq \left( \log M_n + \log n \right) (\log M_n)^2 = o(m_n^*)
\end{equation*}
in Example~1.

Following the similar technique as in Section~\ref{sec:proof:example:aop}, we see the current bound is not sufficient to establish the full AOP in Example~2.

\end{proof}

\subsection{Candidate model set based on MS with bounded $k_l\vee k_u$}\label{sec:MS_bounded}

Let $\hat{\mathcal{M}}_{MS2}=\hat{\mathcal{M}}_{MS}$ with $\kappa_l\vee \kappa_u$ being upper bounded by some positive constant $C$.

\begin{theorem}\label{cor:minimum_2}
Suppose that Assumptions~\ref{asmp:square_summable}--\ref{asmp:regressor_order} hold. Under Condition~\ref{condition1}, even with the precise knowledge of $m_n^*$, i.e., $\hat{m}_n = m_n^*$, we have
\begin{equation}\label{eq:corpart1}
    \mathbb{E}R_n(\hat{\bw}|\hat{\mathcal{M}}_{MS2},\bbf)-R_n(\bw^*|\mathcal{M}_A,\bbf) \gtrsim R_n(\bw^*|\mathcal{M}_A,\bbf).
\end{equation}

Under Condition~\ref{condition2}, if $m_n^*\to \infty$, $\mathbb{E}R_n(\hat{m}_n|\mathcal{M}_A,\bbf)/R_n(m^*|\mathcal{M}_A,\bbf) \to 1$, and there exists a constant $C_2\geq 1$ such that $\hat{u}_n-\hat{l}_n \leq C_2$ almost surely, then we get
$$
\mathbb{E}Q_n(\hat{\bw}|\hat{\mathcal{M}}_{MS2},\bbf)\leq[1+o(1)] R_n\left(\bw^*|\mathcal{M}_A,\bbf\right).
$$

\end{theorem}

This theorem states that when the coefficients decay slowly, such as in the case  $\theta_j=j^{-\alpha_1},\alpha_1>1/2$, the MMA estimator based on a restricted $\hat{\mathcal{M}}_{MS2}$ cannot achieve the full AOP of MA. However, when the coefficients decay fast, reducing the number of candidate models around $\hat{m}_n$ to a constant level is beneficial for MMA. Nevertheless, requiring $\kappa_l$ and $\kappa_u$ to increase to $\infty$ is still necessary for the full AOP in the case of polynomially decaying coefficients.

\begin{proof}
  Define the random variable $\Delta_{n2}=R_n(\bw^*|\hat{\mathcal{M}}_{MS2},\bbf)-R_n(\bw^*|\mathcal{M}_A,\bbf)$, where $\hat{\mathcal{M}}_{MS2}=\{ \{1,\ldots,m \} : (1\vee\lfloor \kappa_l^{-1}m_n^*\rfloor) \leq m \leq (p_n\wedge \lfloor \kappa_u m_n^*\rfloor) \}$. Let us first prove the results under Condition~\ref{condition1}. Note that $\hat{\mathcal{M}}_{MS2}$ is a deterministic candidate model set, and $\Delta_{n2}$ is a non-random quantity. It is evident that
\begin{equation}\label{eq:60}
\begin{split}
   &\mathbb{E}R_n(\hat{\bw}|\hat{\mathcal{M}}_{MS2},\bbf)-R_n(\bw^*|\mathcal{M}_A,\bbf) \geq \Delta_{n2} \\
     & \geq \sum_{j=2}^{\lfloor \kappa_l^{-1}m_n^*\rfloor}\left( \frac{\sigma^2}{n} - \frac{\sigma^2}{n+\frac{\sigma^2}{\theta_j^2}} \right) + \sum_{j=\lfloor \kappa_um_n^*\rfloor+1}^{p_n}\frac{\theta_j^2}{1+\frac{\sigma^2}{n\theta_j^2}}.
\end{split}
\end{equation}
Recall the function $G_d$ defined in (\ref{eq:Gfunction}). Under Condition~\ref{condition1}, there must exist two integers $d_6^*$ and $t_n^*=G_{d_6^*}(m_n^*+1)$ such that $\theta_{m_n^*+1}^2/\theta_{t_n^*}^2 \geq \delta^{2d_6^*}$ and $\lfloor \kappa_l^{-1}m_n^*\rfloor-t_n^*\asymp m_n^*$ when $\kappa_l$ is bounded. Hence the first term on the right side of (\ref{eq:60}) can be lower bounded by
\begin{equation*}
\begin{split}
&\sum_{j=2}^{\lfloor\kappa_l^{-1}m_n^*\rfloor}\left( \frac{\sigma^2}{n} - \frac{\sigma^2}{n+\frac{\sigma^2}{\theta_j^2}} \right)\\
&  =\sum_{j=2}^{\lfloor \kappa_l^{-1}m_n^*\rfloor}\frac{\sigma^2}{n}-\sum_{j=2}^{t_n^*}\frac{\sigma^2}{n+\frac{\sigma^2}{\theta_j^2}}-\sum_{j=t_n^*+1}^{\lfloor \kappa_l^{-1}m_n^*\rfloor}\frac{\sigma^2}{n+\frac{\sigma^2}{\theta_j^2}}\\
     &\geq \frac{(\lfloor \kappa_l^{-1}m_n^*\rfloor-t_n^*)\sigma^2}{n}-\frac{(\lfloor \kappa_l^{-1}m_n^*\rfloor-t_n^*)\sigma^2}{n(1+\delta^{2d_6^*})}\\
     &\asymp \frac{m_n^*}{n}.
\end{split}
\end{equation*}
Similarly, when $\kappa_u$ is bounded, the second term in (\ref{eq:60}) has a lower bound with the order $m_n^*/n$. Thus, we have
\begin{equation*}
  \mathbb{E}R_n(\hat{\bw}|\hat{\mathcal{M}}_{MS2},\bbf)-R_n(\bw^*|\mathcal{M}_A,\bbf) \geq \Delta_{n2} \gtrsim \frac{m_n^*}{n} \asymp R_n(\bw^*|\mathcal{M}_A,\bbf),
\end{equation*}
where the last approximation follows from Lemma~\ref{lemma:peng}.

Under Condition~\ref{condition2}, it is easy to see
\begin{equation}\label{eq:show}
  \mathbb{E}Q_n(\hat{\bw}|\hat{\mathcal{M}}_{MS2},\bbf) \leq  [1+o(1)]\mathbb{E}R_n(\bw^*|\hat{\mathcal{M}}_{MS2},\bbf),
\end{equation}
where $Q_n(\hat{\bw}|\hat{\mathcal{M}}_{MS2},\bbf)$ and $R_n(\bw^*|\hat{\mathcal{M}}_{MS2},\bbf)$ are defined by directly plugging $\hat{\mathcal{M}}_{MS2}$ into $Q_n(\hat{\bw}|\mathcal{M},\bbf)$ and $R_n(\bw^*|\mathcal{M},\bbf)$, respectively. Indeed, based on the risk bound (\ref{eq:risk_bound_general}), we only need to show that $\mathbb{E}\psi(\hat{\mathcal{M}}_{MS2})=o(m_n^*)$. Note that
\begin{equation*}
  \mathbb{E}\psi(\hat{\mathcal{M}}_{MS2}) \lesssim \mathbb{E}\left[M_n(\log M_n)^2\right]  < C = o(m_n^*),
\end{equation*}
where the second inequality is due to $\hat{u}_n-\hat{l}_n$ is bounded almost surely. Thus (\ref{eq:show}) is proved. Then define a candidate model set that contains a single model $\hat{\mathcal{M}}_{MS3}=\{ \hat{m}_n \}$. We see that
\begin{equation}\label{eq:final1}
\begin{split}
   &\mathbb{E}R_n(\bw^*|\hat{\mathcal{M}}_{MS2},\bbf) \leq \mathbb{E}R_n(\bw^*|\hat{\mathcal{M}}_{MS3},\bbf)\\
    &= \mathbb{E}R_n(\hat{m}_n|\mathcal{M}_A,\bbf)\sim R_n(m^*|\mathcal{M}_A,\bbf) \sim R_n(\bw^*|\mathcal{M}_{A},\bbf),
\end{split}
\end{equation}
where the last approximation follows from Lemma~\ref{lemma:peng}. On the other hand, we have
\begin{equation}\label{eq:final2}
  R_n(\bw^*|\mathcal{M}_{A},\bbf) \leq \mathbb{E}R_n(\bw^*|\hat{\mathcal{M}}_{MS2},\bbf).
\end{equation}
By combining (\ref{eq:final1})--(\ref{eq:final2}) with (\ref{eq:show}), we obtain the desired conclusion.

\end{proof}

\section{Additional Numerical Results}\label{sec:a:simu}

\subsection{Assessing the full AOP of MMA}\label{sec:simulation:subsec1}

To illustrate the full-AOP theory in Section~\ref{sec:main_results}, we focus on the MMA estimator based on the largest nested candidate model set $\mathcal{M}_A$ as a representative. Let $\bbf^{(r)}$ and $\hat{\bbf}^{(r)}$ denote the true mean vector and the estimated mean vector in the $r$-th replicate, respectively. We plot the risk ratio
\begin{equation}\label{eq:risk_ratio}
  \text{Ratio} = \frac{R^{-1}\sum_{r=1}^{R}\|\bbf^{(r)} - \hat{\bbf}^{(r)}_{\hat{\mathbf{w}}|\mathcal{M}_A} \|^2}{R^{-1}\sum_{r=1}^{R}\min_{\mathbf{w}\in\mathcal{W}_{p_n}}\|\bbf^{(r)} - \hat{\bbf}^{(r)}_{\mathbf{w}|\mathcal{M}_A} \|^2}
\end{equation}
as a function of $n$, where $\hat{\bbf}^{(r)}_{\hat{\mathbf{w}}|\mathcal{M}_A}$ is the MMA estimator in the $r$-th replicate. The optimizations involved in (\ref{eq:risk_ratio}) can be efficiently
performed by quadratic programming. For example, \verb"quadprog" package in R language is applicable. The simulation results are displayed in Figure~\ref{fig:aop}.

\begin{figure}[!h]
  \centering
  \includegraphics[width=5in]{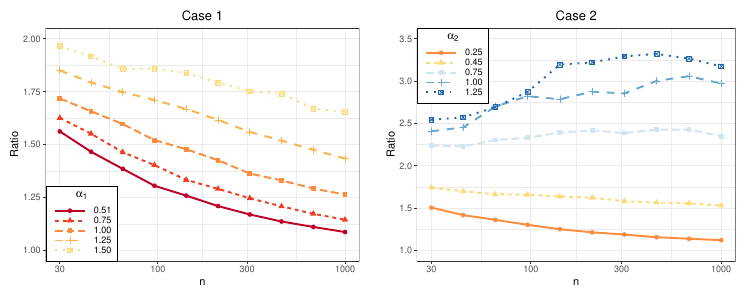}
  \caption{Assessing the full AOP of the MMA estimator based on the largest nested candidate model set $\mathcal{M}_A$ and the general continuous weight set $\mathcal{W}_{p_n}$. }
  \label{fig:aop}
\end{figure}

As shown in the left panel of Figure~\ref{fig:aop}, the curves decrease gradually and tend to $1$ as the sample size $n$ increases. This feature confirms our theoretical understanding that MMA attains the full AOP without restricting the weight or candidate model set when coefficients decay at a polynomial rate. Another observation is that when the sample size $n$ is fixed, the risk ratio increases as $\alpha_1$ increases from $0.51$ to $1.5$. This phenomenon implies that it is more difficult for MMA to achieve the full AOP when coefficients decay fast, which is expected. As observed in the right panel of the figure, the curve with $\alpha_2\leq 0.45$ still shows an apparent downward trend. However, the curves with large $\alpha_2$ exhibit quite different patterns. It seems that the risk ratio experiences a two-phase process, a sharp increase when $n\leq 300$ followed by a slight decrease when $n$ approaches $1000$. Due to the limit of computing power, it is not easy to check by simulation whether these curves will finally tend to $1$ when $n$ is sufficiently large.

\subsection{Comparing different choices of candidate model set}\label{sec:simulation:subsec3}

The primary purpose of this subsection is to compare several full-AOP MMA strategies, which are based on different candidate model sets as summarized in Table~\ref{tab:method}. The competing methods include M-G1 with $k_1 = \lceil \log n \rceil$ and $\zeta_n=1/\log n$, M-G2 with $k_1 = \lceil \log n \rceil$ and $\zeta_n=0$, M-MS1 with $\kappa_l=\kappa_u=\log n$, and M-MS2 with $\hat{l}_n = 1\vee (\hat{m}_n-5)$ and $\hat{u}_n = p_n \wedge (\hat{m}_n+5)$, where $\hat{m}_n$ in M-MS1 and M-MS2 is selected by Mallows' $C_p$ criterion. To show the differences between the competing methods, we divide the $\ell_2$ loss of these four methods by the $\ell_2$ loss of the MMA estimator based on $\mathcal{M}_A$. The simulation results are presented in Figure~\ref{fig:risk}.

\begin{figure}[!htbp]
    \centering
    \subfigure[Case 1]{
    \begin{minipage}[t]{1\linewidth}
    \centering
       \includegraphics[width=4.7in]{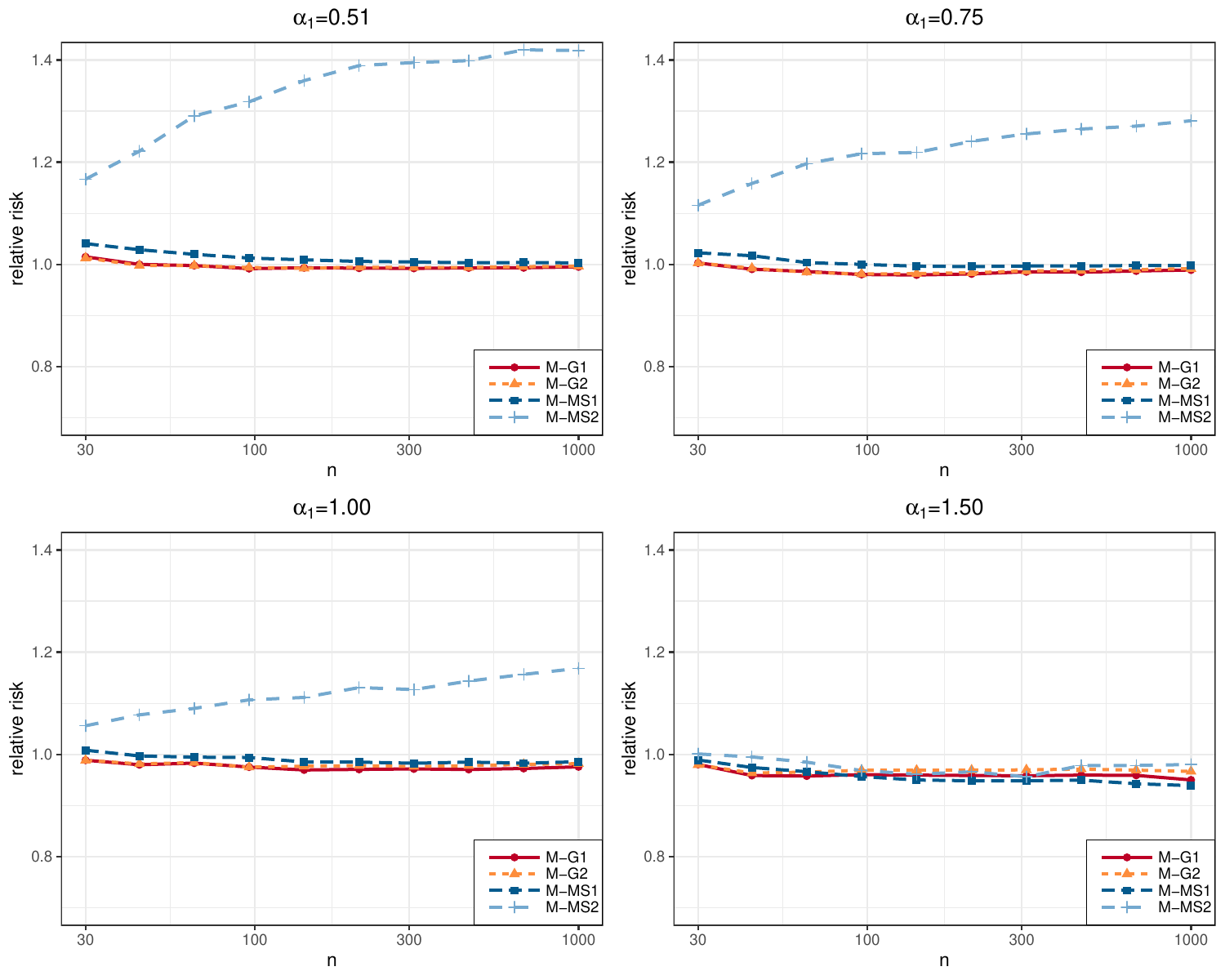}
    \end{minipage}
    }

    \subfigure[Case 2]{
    \begin{minipage}[t]{1\linewidth}
    \centering
       \includegraphics[width=4.7in]{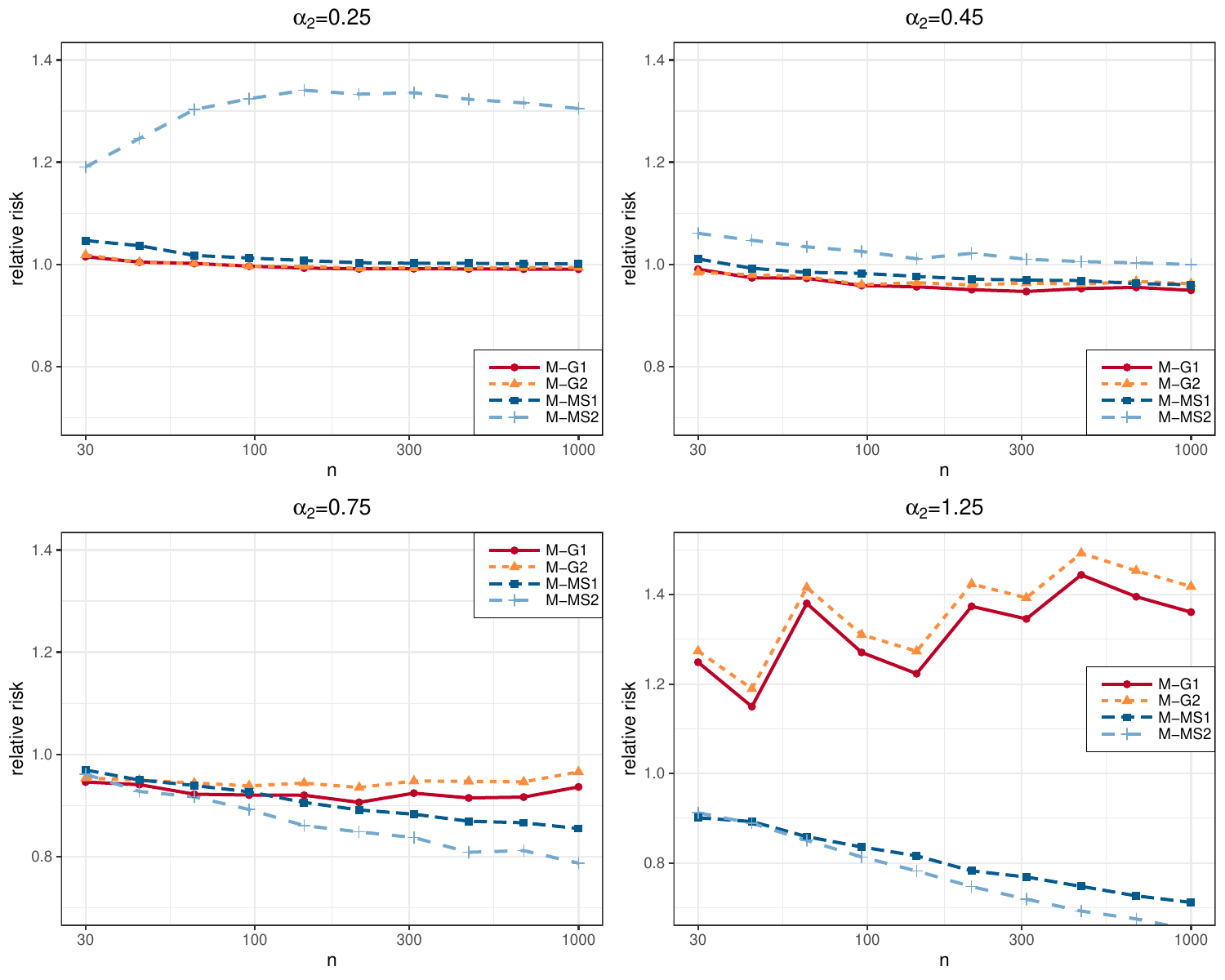}
    \end{minipage}
    }

    \caption{Relative risks of the competing methods in Cases 1--2. In each replication, the squared $\ell_2$ loss of each method is divided by the $\ell_2$ loss of the MMA estimator based on $\mathcal{M}_A$ and $\mathcal{W}_{p_n}$.}
    \label{fig:risk}
\end{figure}

As can be seen from Figure~\ref{fig:risk} (a), the relative risks of the methods M-G1, M-G2, and M-MS1 are near 1. This feature corroborates the findings in Theorems~\ref{cor:grouped}--\ref{cor:minimum_1} that the full AOP is still realized based on these properly constructed candidate model sets. Figure~\ref{fig:risk} (a) also illustrates the consequence of over-reducing the number of candidate models. The M-MS2 method, which combines at most 11 models around $\hat{m}_n$, exhibits much higher relative risks than 1 when the coefficients decay slowly. This observation accords with our statement in Theorem~\ref{cor:minimum_2} that M-MS2 cannot achieve the full potential of MA in Case 1.

From Figure~\ref{fig:risk} (b), we observe that the methods M-G1, M-G2, and M-MS1 perform slightly better than the MMA estimator based on $\mathcal{M}_A$ when $\alpha_2=0.45$ and $0.75$. In addition, the methods M-MS1 and M-MS2 show an obvious advantage when $\alpha_2 = 1.25$. These results further support our understanding in Section~\ref{sec:reduced} and Sections~\ref{sec:equal_sized}--\ref{sec:MS_bounded} that contracting the candidate model set provides certain benefits for MMA when coefficients decay fast. Interestingly, when coefficients decay extremely fast ($\alpha_2 = 1.25$), the curves of the methods M-G1 and M-G2 show an upward trend with some fluctuations. A sensible explanation is that the M-G methods exclude the best candidate model in this case. Note that their smallest candidate model has size $k_1= \lceil \log n \rceil$, while the optimal single model, in this case, is $m_n^* \asymp (\log n)^{4/5}$. Therefore, excluding the best candidate models from below can be harmful as well due to unnecessarily large variances in the models. This is in contrast to the situation of excluding the best models from above, as done in the MR methods, which induces unnecessarily large biases in the candidate models.

We also notice that the results with $\alpha_1=1.5$ in Case 1 show more similar patterns to those in Case 2, while the relative risk curves with $\alpha_2=0.25$ in Case 2 are more like those in Case 1. Indeed, this phenomenon is caused by the same reason stated at the end of Section~\ref{sec:simulation}. See \cite{Liu2011PI} and \cite{ZHANG201595} for more related theoretical and numerical discussions.

\section{Discussions on other related works}\label{sec:a:related}

\subsection{Aggregation}

It is worth mentioning that our work relates to a vast literature on aggregation procedures, which were first studied by \cite{Yang2000Mixing, Yang2001, Yang2004}, \cite{Nemirovski2000, juditsky2000functional}, and \cite{Catoni2004}, respectively. The optimal rates of aggregation have been established by \cite{Tsybakov2003, Wang2014} and various rate-optimal procedures have been proposed with different weight constraints \citep[see, e.g.,][]{Tsybakov2003, Yang2004, Bunea2007, Lounici2007, Rigollet2011, Dalalyan2012Mirror, Lecue2013, Wang2014}. A significant difference between the traditional aggregation procedures and the MMA-type methods is that the formers often focus on the step of combining models, namely, \emph{pure aggregation}, wherein one has already obtained the candidate estimates based on previous studies, or from data splitting \citep[see, e.g.,][]{Yang2001, Lecue2007, Rigollet2007}.

When candidate models and aggregation are trained on the same sample, some substantial progress has also been made in the aggregation literature. The exponential weighting (EW) methods in \cite{Leung2006, Alquier2011, Rigollet2011, Dalalyan2012} and the Q-aggregation in \cite{Dai2014, Bellec2018} are suitable for combining least squares or affine estimators from the same data. In particular, the EW method described in \citet{Dalalyan2012} can be applied for convex aggregation of a list of affine estimators. Note that the EW method can be formulated as the entropy-penalized empirical risk minimization problem
\begin{equation}\label{eq:ds1}
  \hat{\pi}_{EW}=\arg\inf_{\pi}\left\{ \int_{\mathcal{W}_{M_n}} C_n(\bw)\pi(d\bw)+\frac{\lambda}{n}D_{\mathrm{KL}}(\pi||\pi_0) \right\},
\end{equation}
where $\pi$ is a probability measure on $\mathcal{W}_{M_n}$, $C_n(\bw)$ is the MMA criterion (\ref{eq:criterion}), $\lambda$ is a temperature parameter, $\pi_0$ is a given prior, and $D_{\mathrm{KL}}$ stands for the Kullback-Leibler divergence. The final EW estimator is
\begin{equation}\label{eq:ds2}
  \hat{\bbf}_{EW}=\int_{\mathcal{W}_{M_n}}\hat{\bbf}_{\bw|\mathcal{M}}\hat{\pi}_{EW}(d\bw).
\end{equation}
When $\pi_0$ is the uniform distribution on $\mathcal{W}_{M_n}$ and $\lambda\geq8\sigma^2$, Proposition~2 of \citet{Dalalyan2012} implies that
\begin{equation}\label{eq:dsbound}
    \mathbb{E}L_n(\hat{\bbf}_{EW},\bbf) \leq R_n(\bw^*|\mathcal{M},\bbf) + \frac{CM_n\log(n)}{n}.
\end{equation}
When $M_n$ is large, \cite{Dalalyan2012} suggest a heavy tailed prior $\pi_0$ which favors sparse weight vectors. Their Proposition~3 shows that with a properly defined $\pi_0$,
\begin{equation}\label{eq:dsbound2}
    \mathbb{E}L_n(\hat{\bbf}_{EW},\bbf) \leq R_n(\bw^*|\mathcal{M},\bbf) + \frac{C\log(nM_n)}{n}.
\end{equation}
First, notice that the EW estimator (\ref{eq:ds2}) coincides with the MMA estimator (\ref{eq:MMA_estimator}) when $\lambda=0$ but differs from (\ref{eq:MMA_estimator}) when $\lambda>0$. The risk bounds (\ref{eq:dsbound}) and (\ref{eq:dsbound2}), which are obtained under the condition $\lambda\geq8\sigma^2$, are not applicable for the understanding of the MMA method as intended in this paper. Second, the core proof technique in \cite{Dalalyan2012} is based on Stein's lemma \citep{Stein1981Estimation}, which requires $\epsilon$ to follow a Gaussian distribution and the error variance is estimated based on independent data, which is typically unavailable. In contrast, our MMA approach can handle the sub-Gaussian errors with $\sigma^2$ being estimated based on the same data. It is worthy mentioning the risk bounds (\ref{eq:dsbound})--(\ref{eq:dsbound2}) also target the optimal MA risk $R_n(\bw^*|\mathcal{M},\bbf)$ as the MMA approach does. They can justify the full AOP of the EW method when the priors are properly selected.

Proposition~7.2 of \citet{Bellec2018} gives a risk bound for MMA when $\epsilon$ is normally distributed and $\sigma^2$ is known. Integrating the tail probability of their equation (7.4) yields
  \begin{equation}\label{eq:bellec}
    \mathbb{E}L_n(\hat{\bw}|\mathcal{M},\bbf) \leq R_n(\bw^*|\mathcal{M},\bbf) + \frac{C\log M_n}{n} + \left(\frac{C\log M_n}{n}\right)^{\frac{1}{2}}.
  \end{equation}
  The bound (\ref{eq:bellec}) cannot achieve MMA's AOP unless the optimal MA risk $R_n(\bw^*|\mathcal{M},\bbf)$ converges slower than $(\log M_n/n)^{1/2}$. Note that the framework in \cite{Bellec2018} allows one to combine a set of affine estimators, which may be applicable to some other MA problems. However, in our MMA context, Theorem~\ref{theorem:main} substantially improves (\ref{eq:bellec}) for AOP under much milder conditions.

Our examination of MMA is in the nested model framework, which serves as a representative setup in the MS and MA literature \cite[see, e.g.,][]{Polyak1991, Li1987, Hansen2007}. Nested models can be seen as a special case of the ordered linear smoother \citep[][]{Kneip1994}. Aggregation of ordered linear smoothers has been studied in \cite{Chernousova2013} and \cite{Bellec2020}. However, their risk bounds are in terms of the best model instead of their optimal combination. As shown in \cite{Peng2021}, the optimal MS risk can be substantially reduced by MA under certain conditions.

\subsection{Minimax adaptivity}

The minimax statement in Section~\ref{sec:minimax} is known as the exact minimax adaptivity, which was first introduced by \cite{efroimovich1984learning} in the Gaussian white noise model and was further investigated for various estimators in other specific problems \citep[see, e.g.,][]{Donoho1995, Efromovich1996, Nemirovski2000, Yang2000Mixing, Cavalier2002Sharp,Dalalyan2012,Bellec2018}. Our setup focuses on the minimax adaptivity on the spaces of the transformed parameters $\btheta$ rather than the spaces of the original regression coefficient $\bbeta$. Similar setup was adopted by \cite{Dalalyan2012} based on a discrete-cosine transformation of $\bbf$. Another goal considered in the literature is the minimax-rate adaptation, which is less demanding but more tangible with much wider applicability. Some MS and MA schemes have been considered to construct the minimax-rate optimal estimators that require almost no assumption on the behaviors of the candidate models. For example, see \cite{Barron1999}, \cite{juditsky2000functional}, and \cite{YANG2000135, Yang2000PATTERN, Yang1998ms} for early representative work.

In this paper, we show that the MMA estimator is adaptive in the exact minimax sense over the family of ellipsoids and hyperrectangles. Some other approaches, such as the blockwise constant (BC) rules \citep{efroimovich1984learning, Efromovich1996, Donoho1995, Nemirovski2000, cavalier2001penalized, Cavalier2002Sharp}, have also been used to derive the exact minimax adaptive estimators on various classes. There are two notable differences between the BC rule and the MMA method. First, the adaptivity of the BC rule can be obtained only when the orders of some hyperparameters, such as the lengths of blocks, are set correctly, while there are no parameters needed to be determined prior to implementing MMA. Second, the BC rule requires $\sigma^2$ to be known, while the MMA method can accommodate the setting with unknown $\sigma^2$, which is more applicable in regression problems. The effects of the variance estimation on MMA are seen in the risk bound (\ref{eq:risk_bound_general}). It is worth noting that the exact minimax adaptivity property over the family of ellipsoids can also be obtained by aggregation methods in \cite{Dalalyan2012} and \cite{Bellec2018}, in which the candidate models are constructed from the Pinsker filters and the variance $\sigma^2$ is assumed to be known or estimated from an independent sample.

\bigskip
\baselineskip=18pt

\bibliographystyle{Chicago}
\bibliography{mybibfile}

\end{sloppypar}
\end{document}